\documentclass[a4paper,11pt]{amsart}
\usepackage[latin1]{inputenc}
\usepackage[francais]{babel}
\usepackage[T1]{fontenc}
\usepackage{latexsym}
\usepackage{amsmath,amsthm}
\usepackage{amsfonts}
\usepackage{epsfig}
\usepackage{amssymb}
\usepackage{amscd}
\usepackage[all]{xy}
\input xy
\usepackage{hyperref}

\usepackage{geometry}

\newtheorem{Lemma}{Lemme}[section]
\newtheorem{Theo}[Lemma]{Th\'eor\`eme}
\newtheorem{Cor}[Lemma]{Corollaire}
\newtheorem{Pro}[Lemma]{Proposition}

\newtheorem{Def}[Lemma]{D\'efinition}

\newtheorem{Examples}[Lemma]{Exemples}

\theoremstyle{remark}
\newtheorem{Rem}[Lemma]{Remarque}
\theoremstyle{remark}
\newtheorem{Notation}[Lemma]{Notation}

\newcommand{\NN}{\mathbb{N}}
\newcommand{\RR}{\mathbb{R}}
\newcommand{\CC}{\mathbb{C}}

\def\dar[#1]{\ar@<2pt>[#1]\ar@<-2pt>[#1]}
\title{Moyennabilit\'e \`a l'infini et exactitude d'un groupo\"{\i}de \'etale}
\author{Ivan Lassagne}

\address{Ivan {\sc Lassagne} : Universite de Lorraine.}

\email{ivan.lassagne@math.u-psud.fr}


\begin{document}
\maketitle
\begin{abstract}
Soit $\mathcal{G}$ un groupo\"{\i}de \'etale localement compact, $\sigma$-compact et s\'epar\'e dont l'espace des unit\'es est not\'e $X$ 
. On donne la d\'efinition de la moyennabilit\'e \`a l'infini pour un tel groupo\"{\i}de en g\'en\'eralisant celle d'un groupe discret et on \'etudie dans certains cas la relation entre l'exactitude de la $C^*$-alg\`ebre r\'eduite du groupo\"{\i}de $C^*_r(\mathcal{G})$ et la moyennabilit\'e \`a l'infini de $\mathcal{G}$.
\\

\vspace{2mm}
\noindent{\sc {\large a}bstract.} 
Let $\mathcal{G}$ be a locally compact, $\sigma$-compact and Hausdorff \'etale groupoid 
and $X$ its space of units. We give a definition of amenability at infinity of such groupoid. We study in some cases the relation  between the exactness of the reduced $C^*$-algebra $C^*_r(\mathcal{G})$ and the amenability at infinity of $\mathcal{G}$.
\\

\vspace{5mm}
\end{abstract}
\tableofcontents

\section{{Introduction}}
Un groupe est dit moyennable si l'ensemble de ses parties admet une mesure finie qui est finiment additive et invariante par translation. Il existe de nombreuses d\'efinitions de la moyennabilit\'e pour les groupes localement compacts en termes de moyennes invariantes sur certains espaces, en termes de point fixes, en termes de repr\'esentations, en termes de suite de F{\o}lner etc....On trouve de nombreuses r\'ef\'erences litt\'eraires sur la moyennabilit\'e des groupes parmi lesquels \cite{Pier84}, \cite{PatersonBook88}, \cite{Runde02}, \cite{BekkaHarpeValette08} ou encore \cite{AnanRenault00} pour la g\'en\'eralisation de la moyennabilit\'e au cas des groupo\"{\i}des. On utilise la d\'efinition de moyennabilit\'e d'un groupe localement compact en termes de moyenne invariante : pour tout groupe $G$ localement compact, on appelle moyenne sur $G$ une forme lin\'eaire $m : L^\infty(G)\to\RR$  positive, c'est-\`a-dire v\'erifiant $m(f)\geq{0}$ si $f\geq{0}$ et de norme un.
\begin{Def}
Un groupe $G$ localement compact est moyennable s'il existe une moyenne sur $G$, not\'ee $m : L^\infty(G)\to\RR$, invariante \`a gauche c'est-\`a-dire v\'erifiant pour tout $g$ dans $G$ et toute fonction $f$ dans $L^\infty(G)$, on a $m(g.f)=m(f)$ et telle que $m\big(1_{L^\infty(G)}\big)=1$.
\vspace{+2mm}
\end{Def}

De nombreux travaux sur les groupes portent sur la relation entre la moyennabilit\'e d'un groupe localement compact $G$ et la nucl\'earit\'e des $C^*$-alg\`ebres pleine $C^*(G)$ et r\'eduite $C^*_r(G)$ associ\'ees au groupe $G$. On rappelle qu'une $C^*$-alg\`ebre $A$ est nucl\'eaire si pour toute $C^*$-alg\`ebre $B$, on a l'identification $A\otimes_{\textrm{min}}B=A\otimes_{\textrm{max}}B$. Dans le cas d'un groupe discret $G$, Lance prouve dans \cite{Lance73} l'\'equivalence entre la moyennabilit\'e de $G$ et la nucl\'earit\'e de $C^*_r(G)$. 
\\

Dans \cite{Effros75}, l'auteur donne une d\'efinition plus faible de la moyennabilit\'e pour un groupe localement compact et appel\'ee moyennabilit\'e int\'erieure :  on donne la d\'efinition de moyennabilit\'e int\'erieure des articles \cite{BedosHarpe86}, \cite{PatersonBook88}, \cite{LauPaterson91} ou encore \cite{Stalder06} qui  diff\`ere de celle de \cite{Effros75}
\begin{Def}
Un groupe localement compact $G$ est dit int\'erieurement moyennable s'il existe une moyenne $m : L^\infty(G)\to\RR$ sur $G$ qui soit invariante par automorphismes int\'erieurs, c'est-\`a-dire telle que pour tout $g$ dans $G$ et tout $f$ dans $L^\infty(G)$, la moyenne $m$ v\'erifie $m(gfg^{-1})=m(f)$.
\vspace{-1mm}
\end{Def}
Lau et Paterson prouvent dans \cite{LauPaterson91} que tout groupe $G$ localement compact et int\'erieurement moyennable, tel que  la $C^*$-alg\`ebre $C^*_r(G)$ est nucl\'eaire, est alors un groupe moyennable.
\\

Dans l'article \cite{Anantharaman2}, l'auteur \'etudie le cas de groupe de transformation $(X,G)$ c'est-\`a-dire le cas d'un espace localement compact muni d'une action d'un groupe localement compact. Un des objectifs principaux de l'article \cite{Anantharaman2} est la g\'en\'eralisation du r\'esultat de Lau et Paterson \'enonc\'e auparavant au cas des groupes de transformation. Pour un groupe de transformation $(X,G)$, la notion de moyennabilit\'e utilis\'ee est celle de moyennabilit\'e (topologique) du groupo\"{\i}de $X\rtimes{G}$ : Anantharaman-Delaroche donne la d\'efinition de moyennabilit\'e du groupe de transformation $(X,G)$ en termes de moyenne continue invariante approch\'ee de la mani\`ere suivante 
\begin{Def}{\cite{Anantharaman2}}
On dit que le groupe de transformation $(X,G)$ est moyennable s'il existe une suite $(m_i)_{i\in{I}}$ d'applications continues $x\mapsto{m_i^x}$ de $X$ dans l'espace $Prob(G)$ telles que 
$$\lim_{i}\lVert{s.m_i^x-m_i^{sx}}\lVert_1=0$$
uniform\'ement sur tout sous espace compact de $X\times{G}$. Une telle suite $(m_i)_{i\in{I}}$ est appel\'ee moyenne continue invariante approch\'ee.
\end{Def}
En adaptant les d\'efinitions de groupo\"{\i}des moyennables de {\cite{AnantharamanRenault}} , Anantharaman-Delaroche 
donne une caract\'erisation de la moyennabilit\'e d'un groupe de transformations en termes de suite de fonctions. La notion de groupe int\'erieurement moyennable n'admet \`a priori pas de g\'en\'eralisation pour les groupes de transformation. Anantharaman-Delaroche remplace cette condition par une notion plus faible qui est la propri\'et\'e (W) pour un groupe de transformation $(X,G)$ exprim\'ee \`a l'aide de fonctions de type positif sur  le $G\times{G}$-espace $X\times{X}$ : 
\begin{Def}
Soient $G$ un groupe localement compact et s\'epar\'e et $X$ un $G$-espace localement compact et s\'epar\'e. Le groupe de transformation $(X,G)$ a la propri\'et\'e (W) si pour tout sous espace compact $K$ de $X\times{G}$ et pour tout r\'eel $\varepsilon$ strictement positif, il existe une fonction $h$ continue, born\'ee de type positif, \`a support $\pi$-propre sur $(X\times{G})\times(X\times G)$ telle que $\lvert{h(x,t,x,t)-1}\lvert\leq\varepsilon$, pour tout $(x,t)$ dans $K$.
\end{Def}
Tout groupe de transformation $(X,G)$ moyennable poss\`ede la propri\'et\'e (W). On dit qu'un groupe poss\`ede la propri\'et\'e (W) si le groupe de transformation $\big(\{.\},G\big)$ poss\`ede la propri\'et\'e (W), o\`u le groupe $G$ agit trivialement sur l'espace \{p\} constitu\'e d'un seul point. Dans ce cas, il s'agit d'une notion plus faible que la moyennabilit\'e int\'erieure d'un groupe localement compact. 
On montre assez facilement que tout groupe discret poss\`ede la propri\'et\'e (W) ou encore que si un groupe $G$ poss\`ede la propri\'et\'e (W), alors tout groupe de transformation $(X,G)$ poss\`ede ausi la propri\'et\'e (W). 
Anantharaman-Delaroche prouve le th\'eor\`eme suivant qui g\'en\'eralise celui de Lau et Paterson \cite{LauPaterson91}
\begin{Theo}{\cite{Anantharaman2}}
Soit $(X,G)$ un groupe de transformation avec $X$ un $G$-espace localement compact et $G$ un groupe localement compact et s\'epar\'e. Les assertions suivantes sont \'equivalentes
\begin{itemize}
\item[(a)] $(X,G)$ est moyennable
\item[(b)] $C^*_r(X\rtimes{G})$ est nucl\'eaire et $(X,G)$ a la propri\'et\'e (W). 
\end{itemize}
\end{Theo}
Comme les groupes discrets poss\`edent la propri\'et\'e (W), on retrouve l'\'equivalence entre la moyennabilit\'e du groupe $G$ et la nucl\'earit\'e de $C_r^*(G)$ d\'emontr\'e dans \cite{Lance73}.
De m\^eme, dans le cas d'un groupe de transformation $(X,G)$ avec $G$ un groupe discret, le th\'eor\`eme de Anantharaman-Delaroche prouve qu'il y a \'equivalence entre la moyennabilit\'e du groupe de transformation $(X,G)$ et la nucl\'earit\'e de $C^*_r(X\rtimes{G})$. Le cas des groupes discrets attire l'attention notamment pour les liens avec la conjecture de Novikov qui est le sujet de nombreux travaux comme \cite{Higson00}, \cite{Yu00} ou \cite{RoeHigson00}. Dans ces articles, une propri\'et\'e importante pour un groupe discret est la moyennabilit\'e \`a l'infini :
\begin{Def}
 Un groupe localement compact $G$ est moyennable \`a l'infini s'il existe un $G$-espace compact $X$ pour lequel le groupe de transformation est moyennable.
 \end{Def}

Pour les groupes discrets on a une caract\'erisation des groupes moyennables \`a l'infini. On note $\beta{G}$ le compactifi\'e de Stone-C\v{e}ch de $G$, qui coincide avec le spectre de $C_b(G)$, alg\`ebre des fonctions continues et born\'ees sur $G$. Dans l'article \cite{Ozawa00}, l'auteur prouve le th\'eor\`eme suivant
\begin{Theo}[\cite{Ozawa00}]
Soit ${G}$ un groupe discret. Le groupe ${G}$ est exact  si et seulement si le groupe ${G}$ agit moyennablement sur l'espace $\beta{G}$.
\end{Theo}
Pour d\'emontrer ce th\'eor\`eme, Ozawa utilise un r\'esultat de \cite{RoeHigson00} qui dit que l'action d'un groupe discret $G$ par translation \`a gauche sur son compactifi\'e de Stone-C\v{e}ch est moyennable si et seulement si la $C^*$-alg\`ebre de Roe est nucl\'eaire. La $C^*$-alg\`ebre $C^*_r(G)$ \'etant une sous $C^*$-alg\`ebre de la $C^*$-alg\`ebre de Roe,  alors $C^*_r(G)$ est exacte. Ozawa prouve alors la r\'eciproque du r\'esultat de \cite{RoeHigson00} de Higson et Roe et montre que pour tout groupe $G$ discret dont la $C^*$-alg\`ebre $C^*_r(G)$ est exacte, la $C^*$-alg\`ebre de Roe du groupe est nucl\'eaire.
\\

Dans le chapitre 7 de l'article \cite{Anantharaman2} , Anantharaman-Delaroche \'etudie la moyennabilit\'e \`a l'infini des groupes localement compacts. L'auteur d\'efinit pour tout groupe $G$ localement compact, un $G$-espace compact, not\'e $\beta^uG$, comme le spectre d'une alg\`ebre et qui coincide dans le cas d'un groupe discret avec le compactifi\'e de Stone-C\v{e}ch $\beta{G}$. Elle prouve dans la proposition 3.4, que pour tout groupe localement compact $G$, la moyennabilit\'e \`a l'infinie de $G$ \'equivaut \`a la moyennabilit\'e du groupe de transformation $(\beta^uG,G)$. Elle d\'emontre le th\'eor\`eme suivant 
\begin{Theo}[\cite{Anantharaman2}]\label{ProEquMoyInf}
Soit ${G}$ un groupe localement compact. On consid\`ere les assertions suivantes  :
\begin{enumerate}
\item[$(1)$] le groupe ${G}$ est moyennable \`a l'infini  
\item[$(2$)] le groupe ${G}$ est exact
\item[$(3)$] Pour toute $G$-$C^*$-alg\`ebre exacte $B$, la $C^*$-alg\`ebre de produit crois\'e r\'eduit  $C_r^*(G,B)$ est une $C^*$-alg\`ebre exacte. 
\item[$(4)$] la $C^\ast$-alg\`ebre $C^\ast_r({G})$ est exacte  
\end{enumerate} 
Alors $(1)\Rightarrow(2)\Rightarrow(3)\Rightarrow(4)$. Si le groupe $G$ poss\`ede la propri\'et\'e (W), alors $(4)\Rightarrow(1)$ et les assertions sont toutes \'equivalentes.
\vspace{+3mm}
\end{Theo}
 
Le travail qui suit consiste en l'\'etude des groupo\"{\i}des \'etales localement compacts $\sigma$-compacts et s\'epar\'es qui g\'en\'eralisent les groupes discrets, les groupes de transformations $(X,G)$ avec $G$ un groupe discret et $X$ un $G$-espace localement compact $\sigma$-compact et s\'epar\'e ou encore qui apparaissent dans l'\'etude des groupo\"{\i}des d'holonomie de feuilletages. On s'int\'eresse principalement \`a des propri\'et\'es de moyennabilit\'e; la moyennabilit\'e dans le cadre des groupo\"{\i}des localement compacts et s\'epar\'es joue un r\^ole dans la th\'eorie des alg\`ebres d'op\'erateurs : on a vu auparavant les liens entre l'existence d'une action moyennable d'un groupe $G$ localement compact sur un espace compact et l'exactitude de sa $C^*$-alg\`ebre r\'eduite.

On verra notamment les relations entre des propri\'et\'es de moyennabilit\'e du groupo\"{\i}de, des propri\'et\'es d'exactitude du groupo\"{\i}de et d'exactitude de la $C^*$-alg\`ebre r\'eduite du groupo\"{\i}de. 
\begin{itemize}
\item[$\bullet$] La d\'efinition $\ref{DefMoyInf}$ donne une version de la moyennabilit\'e \`a l'infini dans le cas des groupo\"{\i}des \'etales localement compacts et s\'epar\'es.
\item[$\bullet$] L'espace $\beta_X\mathcal{G}$ est d\'efini comme le spectre d'une certaine $\mathcal{G}$-alg\`ebre, not\'ee $C_0^s(\mathcal{G})$ (d\'efinition dans la section $\ref{Galgebre}$ ) et dans le cas o\`u $\mathcal{G}$ est un groupe discret, l'espace $\beta_X\mathcal{G}$ est exactement le compactifi\'e de Stone-Cech $\beta\mathcal{G}$.
\item[$\bullet$] Un groupo\"{\i}de \'etale localement compact $\mathcal{G}$ est dit exact si le foncteur "$\rtimes_r\mathcal{G}$" est exact dans la cat\'egorie des $\mathcal{G}$-alg\`ebres et des $\mathcal{G}$-homomorphismes. 
\end{itemize}

Les r\'esultats que nous obtenons pour les groupo\"{\i}des \'etales localement compact peuvent \^etre sch\'ematis\'es par le diagramme suivant :
$$
\xymatrix{  
\mathcal{G}\textrm{ moyennable \`a l'infini}  \ar@{=>}[d] & & 
\\
\beta_X\mathcal{G}\rtimes\mathcal{G}\textrm{ moyennable }\ar@{=>}[rr] & & \mathcal{G}\textrm{ exact } \ar@{=>}[d] 
\\ 
& &  C^*_r(\mathcal{G})\textrm{ exact }  \ar@/^2pc/@{=>}[llu]^{\textrm{ }+\textrm{{(Cond)}}}
}
$$
L'implication "$C^*_r(\mathcal{G})\textrm{ exact }\Rightarrow{(\beta_X\mathcal{G},\mathcal{G})\textrm{ moyennable }}$" n'a \'et\'e obtenue dans notre cas qu'en imposant certaines conditions sur le groupo\"{\i}de $\mathcal{G}\rightrightarrows{X}$ : la premi\`ere est une version pour les groupo\"{\i}des de la Propri\'et\'e (W) introduite dans {\cite{Anantharaman2}}. La seconde intervient dans l'\'etude de la $C^*$-alg\`ebre $C^*_r(\mathcal{G})$ \`a laquelle on associe un champ de $C^*$-alg\`ebres sur $X$ que l'on supposera continu.
\\
Dans toute la section on consid\`ere $\mathcal{G}\rightrightarrows{X}$ un groupo\"{\i}de \'etale localement compact, $\sigma$-compact et s\'epar\'e. On utilise comme r\'ef\'erence {\cite{Williams07}} pour ce qui est des notions de $C_0(X)$-alg\`ebres et de champ (semi-)continu (sup\'erieurement) de $C^*$-alg\`ebres et leur relation. \\

\noindent{\bf Remerciments :} Ce travail constitue une partie de ma th\`ese de doctorat. Je tiens \`a remercier mon directeur de th\`ese Jean Louis TU pour m'avoir propos\'e ce sujet et pour ces nombreux conseils.

\section{{Pr\'eliminaires}}

\subsection{Groupo\"{\i}des topologiques}
Un groupo\"{\i}de peut \^etre vu comme une g\'en\'eralisation des groupes et des espaces. Il s'agit d'une petite cat\'egorie dans laquelle tous les morphismes sont inversibles. Il est courant d'adopter les notations $\mathcal{G}^{(0)}$ pour d\'efinir les objets et $\mathcal{G}$ (parfois not\'e $\mathcal{G}^{(1)}$) pour d\'efinir les morphismes (souvent appel\'es fl\`eches). On donne ci-dessous une d\'efinition plus concr\`ete d'un groupo\"{\i}de
\begin{Def}
Un groupo\"{\i}de consiste en la donn\'ee de deux espaces $\mathcal{G}$ et $\mathcal{G}^{(0)}$ et d'applications 
\begin{enumerate}
\item[a)] $s,r : \mathcal{G}\rightarrow{\mathcal{G}^{(0)}}$ respectivement applications source et but, 
\item[b)] $m : \mathcal{G}^{(2)}\rightarrow\mathcal{G}$, l'application produit, o\`u $\mathcal{G}^{(2)}=\{(\gamma,\eta)\in{\mathcal{G}\times\mathcal{G}}\textrm{ : }s(\gamma)=r(\eta)\}$, 
\end{enumerate}
ainsi que d'une application unit\'e not\'ee $u : \mathcal{G}^{(0)}\rightarrow\mathcal{G}$ et d'une application inverse not\'ee $i : \mathcal{G}\rightarrow\mathcal{G}$ v\'erifiant, si l'on note $m(\gamma,\eta)=\gamma.\eta$, $u(x)=x$ et $i(\gamma)=\gamma^{-1}$, les relations suivantes
\begin{enumerate}
\item[$\bullet$] $r(\gamma.\eta)=r(\gamma)$ et $s(\gamma.\eta)=s(\eta)$ 
\item[$\bullet$] $\gamma.(\eta.\delta)=(\gamma.\eta).\delta$,$\quad$ pour tout ${(\gamma,\eta)}$ et ${(\eta,\delta)}$ dans ${\mathcal{G}^{(2)}}$
\item[$\bullet$] $\gamma.x=\gamma$ et $x.\eta$,$\quad$ pour tout $({\gamma,\eta})$ dans $\mathcal{G}^{(2)}\textrm{ : }s(\gamma)=x=r(\eta)$
\item[$\bullet$] $\gamma.\gamma^{-1}=r(\gamma)$ et $\gamma^{-1}.\gamma=s(\gamma)$,$\quad$ pour tout $\gamma$ dans $\mathcal{G}$
\end{enumerate}
\end{Def}
\begin{Notation}
Il est assez courant de simplifier les notations en occultant la fonction $u$ et en consid\'erant $\mathcal{G}^{(0)}$ comme un sous espace de $\mathcal{G}$. De m\^eme, on pourra se permettre d'\'ecrire $\gamma.\eta$ voire $\gamma\eta$ pour $m(\gamma,\eta)$ afin d'all\'eger les formules.  
\vspace{+3mm}
\end{Notation}

On consid\`ere par la suite des groupo\"{\i}des topologiques c'est-\`a-dire les groupo\"{\i}des $\mathcal{G}\rightrightarrows\mathcal{G}^{(0)}$ pour lesquels les espaces $\mathcal{G}$ et $\mathcal{G}^{(0)}$ sont munis d'une topologie et les applications de structure du groupo\"{\i}de, $r, s, m, u, i$ sont continues. 

\begin{Pro}[\cite{Tu04}]
Si $\mathcal{G}\rightrightarrows\mathcal{G}^{(0)}$ est un groupo\"{\i}de s\'epar\'e localement compact (respectivement localement compact $\sigma$-compact) alors son espace des unit\'es $\mathcal{G}^{(0)}$ est localement compact (respectivement localement compact $\sigma$-compact).
\end{Pro}

\begin{proof}
On note $\Delta_\mathcal{G}$ la diagonale de l'espace $\mathcal{G}\times\mathcal{G}$. Le groupo\"{\i}de $\mathcal{G}$ \'etant s\'epar\'e, $\Delta_\mathcal{G}$ est un sous espace ferm\'e de $\mathcal{G}\times\mathcal{G}$. Alors $\mathcal{G}^{(0)}=(id,r)^{-1}(\Delta_\mathcal{G})$ est ferm\'e dans $\mathcal{G}$, donc localement compact.
\vspace{+2mm}
\\
\indent
Soit $(K_n)_{n\in\NN}$ une suite de sous espaces compacts dans $\mathcal{G}$, telle que $\mathcal{G}=\cup_{n\in\NN}K_n$. L'application $r : \mathcal{G}\rightarrow\mathcal{G}^{(0)}$ \'etant continue, $(r(K_n))_{n\in\NN}$ est une suite de sous espaces compacts de $\mathcal{G}^{(0)}$ telle que $\cup_{n\in\NN}r(K_n)=\mathcal{G}^{(0)}$. Alors $\mathcal{G}^{(0)}$ est localement compact $\sigma$-compact.
\\
\end{proof}

\begin{Def}
 Soit $\mathcal{G}\rightrightarrows\mathcal{G}^{(0)}$ un groupo\"{\i}de topologique et $X$ et $Y$ des sous espaces de $\mathcal{G}^{(0)}$. On appelle $\mathcal{G}_X$ l'ensemble des fl\`eches de $\mathcal{G}$ qui "commencent dans $X$" et $\mathcal{G}^Y$ l'ensemble des fl\`eches qui "terminent dans $Y$" c'est-\`a-dire les sous espaces de $\mathcal{G}$ d\'efinis par 
$$\mathcal{G}_X:=\{\gamma\in\mathcal{G}\textrm{ : }s(\gamma)\in{X}\}\quad\textrm{et}\quad\mathcal{G}^Y:=\{\gamma\in\mathcal{G}\textrm{ : }r(\gamma)\in{Y}\}$$ 
On note $\mathcal{G}_X^Y:=\mathcal{G}_X\cap\mathcal{G}^Y$ l'ensemble des fl\`eches qui commencent dans $X$ et finissent dans $Y$. Lorsqu'on consid\`ere $A:=\{x\}$ r\'eduit \`a un seul \'el\'ement, on dit que $\mathcal{G}_x$ (resp. $\mathcal{G}^x$) est la $s$-fibre (resp. $r$-fibre) au point $x$ et l'ensemble $\mathcal{G}_x^x$ est muni d'une structure de groupe qu'on appelle le groupe d'isotropie en $x$ du groupo\"{\i}de $\mathcal{G}$.
\vspace{+0.5cm}
\end{Def}

\begin{Examples}
 On donne quelques exemples de groupo\"{\i}des topologiques et on montre comment certains ensembles classiques peuvent \^etre vu comme des groupo\"{\i}des :
\\
\begin{itemize}
\item[(1)] \textbf{Groupe : } tout groupe topologique $G$ peut \^etre consid\'er\'e comme le groupo\"{\i}de topologique $G\rightrightarrows\{e\}$ o\`u les applications but et source s'identifient \`a la projection continue sur l'\'el\'ement neutre $e$ du groupe, la multiplication (resp. inversion) pour le groupo\"{\i}de s'identifie \`a la multiplication (resp. inversion) du groupe. Lorsque le groupe est de Lie, alors le groupo\"{\i}de induit est aussi de Lie.
\vspace{+3mm}
\item[(2)] \textbf{Espace : } tout espace topologique (resp. vari\'et\'e) $M$ peut \^etre conid\'er\'e comme un groupo\"{\i}de topologique (resp. de Lie) $\mathcal{G}\rightrightarrows{M}$, o\`u l'ensemble $\mathcal{G}$ est l'espace topologique (resp. la vari\'et\'e) $M$, c'est-\`a-dire l'espace m\^eme des unit\'es. Toutes les applications de structures s'identifient \`a l'identit\'e. On l'appelle le groupo\"{\i}de identit\'e associ\'e \`a l'espace topologique (resp. la vari\'et\'e) $M$. 
\vspace{+3mm}
\item[(3)] \textbf{Groupo\"{\i}de paire : } pour tout espace topologique $M$, on peut construire un autre groupo\"{\i}de topologique not\'e $Pair(M)\rightrightarrows{M}$ et appel\'e le groupo\"{\i}de paire de l'espace topologique $M$, dont l'espace des fl\`eches est l'espace topologique produit $Pair(M):=M\times{M}$. Les applications but et source sont d\'efinies $s(x,y):=y$ et $r(x,y):=x$, et la multiplication $(x,y).(y,z):=(x,z)$, pour tout couple $(x,y)$ et $(y,z)$ de $M\times{M}$. Les applications inverse et unit\'e sont $i(x,y):=(x,y)^{-1}=(y,x)$ et $u(x):=(x,x)$. 
\vspace{+3mm}
\item[(4)] \textbf{Action de groupe : } on consid\`ere un groupe de Lie $G$ agissant \`a gauche par diff\'eomorphismes sur une vari\'et\'e $M$. On appelle groupo\"{\i}de de transformation le groupo\"{\i}de not\'e $G\ltimes{M}\rightrightarrows{M}$, o\`u l'espace des fl\`eches $G\ltimes{M}$ est la vari\'et\'e produit $G\times{M}$. 
\vspace{+3mm}
\item[(5)]\textbf{Pullback : } soient $M$ un espace topologique, $\mathcal{G}$ un groupo\"{\i}de topologique et $\varphi : M\rightarrow\mathcal{G}^{(0)}$ une application continue et surjective. L'espace 
$$\varphi^\ast\mathcal{G}:=\{(m,\gamma,m')\in{M\times\mathcal{G}\times{M}}\textrm{ : }r(\gamma)=m{\textrm{ , }}s(\gamma)=m'\}=M\times_{\varphi,r}\mathcal{G}\times_{s,\varphi}M$$  
est muni d'une structure de groupo\"{\i}de topologique o\`u l'espace des unit\'es est $M$ et les fl\`eches reliant $m'$ et $m$ dans $M$ sont les fl\`eches de $\mathcal{G}$ reliant $\varphi(m)$ et $\varphi(m')$.
L'application but \'etant donn\'ee par la premi\`ere projection, l'application source par la troisi\`eme projection. Le produit est d\'efini par $(m_1,\gamma,m_2).(m_2,\gamma',m_3)=(m_1,\gamma\gamma',m_3)$ et l'inverse par $(m,\gamma,m')^{-1}=(m',\gamma^{-1},m)$. 
\vspace{+3mm}
\end{itemize}
\end{Examples}

\begin{Def}
On consid\`ere deux groupo\"{\i}des topologiques $\mathcal{G}\rightrightarrows\mathcal{G}^{(0)}$ et $\mathcal{H}\rightrightarrows\mathcal{H}^{(0)}$ et on note $s_\mathcal{G}$ (respectivement $s_\mathcal{H}$) l'application source de $\mathcal{G}$ (respectivement $\mathcal{H}$) et $r_\mathcal{G}$ (respectivement $r_\mathcal{H}$) l'application but de $\mathcal{G}$ (respectivement $\mathcal{H}$). Un homomorphisme de groupo\"{\i}des $\varphi : \mathcal{G}\to\mathcal{H}$ est d\'efini par deux applications continues 
$$\varphi : \mathcal{G}\to\mathcal{H}\quad\textrm{ et }\quad{\varphi^{(0)} : \mathcal{G}^{(0)}\to\mathcal{H}^{(0)} }$$
telles que les relations suivantes sont satisfaites : pour tout $(\gamma,\eta)$ dans $\mathcal{G}^{(2)}$
\begin{itemize}
\item[$\bullet$] $\quad \varphi(\gamma\eta)=\varphi(\gamma)\circ\varphi(\eta)\textrm{, }\quad$
\item[$\bullet$] $\quad s_\mathcal{H}\circ{\varphi}=\varphi^{(0)}\circ{s_\mathcal{G}}$
\item[$\bullet$] $\quad r_\mathcal{H}\circ{\varphi}=\varphi^{(0)}\circ{r_\mathcal{G}}$
\item[$\bullet$] $\quad \varphi(\gamma^{-1})=\varphi(\gamma)^{-1}$
\vspace{+2mm}
\end{itemize}
\end{Def}

%
%
%
\begin{Examples}
On donne deux exemples basiques d'homomorphismes stricts de groupo\"{\i}des : 
\begin{itemize}
\vspace{+1mm}
\item[$\mathrm{(1)}$] Il est clair qu'un homomorphisme de groupes topologique $f : G\to{H}$ est un homomorphisme (strict) de groupo\"{\i}des lorsqu'on se place du point de vue des groupo\"{\i}des.
\vspace{-1mm}
\\
\item[$\mathrm{(2)}$] Toute application continue $f : X\to{Y}$ entre espaces topologiques d\'efinit un homomorphisme strict entre les groupo\"{\i}des identit\'es induits par les espaces topologiques. Il en va de m\^eme pour les groupo\"{\i}des paires associ\'es aux espaces topologiques.
\vspace{+2mm}
\end{itemize}
\end{Examples}

\begin{Def}
Un groupo\"{\i}de topologique $\mathcal{G}\rightrightarrows\mathcal{G}^{(0)}$ est dit \'etale si les applications but $r$ et source $s$ sont des hom\'eomorphismes locaux. Dans le cas o\`u $\mathcal{G}$ est un groupo\"{\i}de de Lie, on demande \`a ce que les applications but $r$ et source $s$ soient des diff\'eomorphismes locaux.
\end{Def}
\begin{Rem}
On remarque que tout groupo\"{\i}de \'etale propre poss\`ede des groupes d'isotropie qui sont finis.
\vspace{+0.5cm}
\end{Rem}
On donne la d\'efinition d'une action \`a droite d'un groupo\"{\i}de sur un espace (l'action \`a gauche \'etant d\'efinie de mani\`ere analogue).

\begin{Def}
Soit $\mathcal{G}\rightrightarrows\mathcal{G}^{(0)}$ un groupo\"{\i}de topologique et $Z$ un espace localement compact, s\'epar\'e et $\sigma$-compact. On appelle action \`a droite du groupo\"{\i}de $\mathcal{G}$ sur l'espace $Z$, la donn\'ee d'une application continue $\sigma : Z\rightarrow{\mathcal{G}^{(0)}}$ (souvent appel\'ee \textit{application momentum})et d'une application continue de $Z\times_{\sigma,r}\mathcal{G}$ \`a valeurs dans $Z$, not\'e $(z,\gamma)=z.\gamma$ et v\'erifiant pour tout $z\in{Z}$
\begin{enumerate}
\item[a)] $\sigma(z.\gamma)=s(\gamma)$, pour tout ${\gamma}$ dans $\mathcal{G}\textrm{ tel que }\sigma(z)=r(\gamma)$.
\item[b)] $z.x=z$, pour tout ${x}$ dans ${\mathcal{G}^{(0)}}$ tel que $\sigma(z)=x$.
\item[c)] $(z.\gamma).\gamma'=z.(\gamma.\gamma')$, pour tout ${(\gamma,\gamma')}$ dans ${\mathcal{G}^{(2)}}\textrm{ tel que }\sigma(z)=r(\gamma)$
\end{enumerate}
\end{Def}

\begin{Rem}
L'application $\sigma : Z\rightarrow{\mathcal{G}^{(0)}}$ est appel\'ee parfois application momentum de l'action de $\mathcal{G}$ sur $Z$.
\\
L'action d'un groupo\"{\i}de $\mathcal{G}$ sur un espace $Z$ permet de construire un groupo\"{\i}de produit crois\'e de $Z$ par $\mathcal{G}$, not\'e $Z\rtimes\mathcal{G}$, compos\'e de l'ensemble des triplets $(z,\gamma,y)$ dans $Z\times\mathcal{G}\times{Z}$, tels que $z.\gamma=y$. Les \'el\'ements de $(Z\rtimes\mathcal{G})^{(2)}$ sont les paires $\big((x,\gamma,y),(y,\gamma',z)\big)$ et leur multiplication est d\'efinie par $(x,\gamma,y).(y,\gamma',z)=(x,\gamma.\gamma',z)$. L'inverse de $(x,\gamma,y)^{-1}=(y,\gamma^{-1},x)$.
\vspace{+0.5cm}
\end{Rem}

\subsection{Alg\`ebres associ\'ees \`a un groupo\"{\i}de}

Dans cette section, on rappelle les constructions des $C^*$-alg\`ebres pleines et r\'eduites associ\'ees \`a un groupo\"{\i}de localement compact et s\'epar\'e muni d'un syst\`eme de Haar.

\subsubsection{Alg\`ebres de convolution et repr\'esentations}
On d\'efinit, pour un groupo\"{\i}de $\mathcal{G}$ muni d'un syst\`eme de Haar, une structure $\ast$-alg\'ebrique sur l'espace vectoriel complexe $C_c(\mathcal{G})$ des fonctions continues et \`a support compact sur le groupo\"{\i}de. On d\'efinit tout d'abord le produit de convolution et l'involution sur $C_c(\mathcal{G})$; on \'etudie les $\ast$-repr\'esentations afin d'obtenir des $C^*$-normes qui nous permettront de construire des $C^*$-alg\`ebres.
\vspace{+2mm}
\\
Soit $\mathcal{G}$ un groupo\"{\i}de s\'epar\'e localement compact muni d'un syst\`eme de Haar $\nu=\{\nu^x\textrm{, }x\in\mathcal{G}^{(0)}\}$ et $C_c(\mathcal{G})$ l'espace vectoriel complexe des fonctions continues \`a support compact sur le groupo\"{\i}de et \`a valeurs complexes.
\begin{Def}
On d\'efinit une involution $\ast$ et un produit de convolution $\star$ sur $C_c(\mathcal{G})$ de la mani\`ere suivante
\begin{itemize}
\item[(i)] Convolution : pour toutes fonctions ${f,g}$ dans ${C_c(\mathcal{G})}$, on pose 
$$f\star{g}:=\int_{\mathcal{G}^{r(\gamma)}}f(\gamma')g(\gamma\gamma')d\nu^{r(\gamma)}(\gamma')$$
\item[(ii)] Involution : pour toute fonction ${f}$ dans ${C_c(\mathcal{G})}\textrm{ et tout }\gamma$ dans ${\mathcal{G}}$, on pose 
$$f^*(\gamma):=\overline{f(\gamma^{-1})}$$
\end{itemize}
\vspace{+2mm}
\end{Def}
\begin{Rem}
On peut munir la $\ast$-alg\`ebre $C_c(\mathcal{G})$ de la topologie induite par la famille de semi-normes $\{\lVert{.}\lVert_K\textrm{ : }K\subset{\mathcal{G}}\textrm{ compact}\}$, d\'efinie pour tout $K$ compact dans $\mathcal{G}$ et toute fonction $f$ de $C_c(\mathcal{G})$ par :
$$\lVert{f}\lVert_K:=\sup_{x\in{K}}\lvert{f(x)}\lvert$$

La structure d\'epend du choix du syst\`eme de Haar $\nu$ : on la note g\'en\'eralement $C_c(\mathcal{G},\nu)$ ou $C_c(\mathcal{G})$ s'il n'y a pas d'ambiguit\'e.
\end{Rem}


On d\'efinit une norme sur $C_c(\mathcal{G},\nu)$ de la mani\`ere suivante :
\begin{Def}
On note $\|.\|_I$ la norme sur la $\ast$-alg\`ebre topologique $C_c(\mathcal{G})$ d\'efinie, pour toute fonction ${f}$ dans ${C_c(\mathcal{G})}$ par
$$\lVert{f}\lVert_I:=\max\bigg\{\sup_{x\in\mathcal{G}^{(0)}}\int_{\mathcal{G}^{x}}|f(\gamma)|d\nu^{x}(\gamma)\textrm{ , }\sup_{x\in\mathcal{G}^{(0)}}\int_{\mathcal{G}_{x}}|f(\gamma)|d\nu_{x}(\gamma)\bigg\}$$
\end{Def}
\begin{Rem}
On montre facilement que $\|.\|_I$ v\'erifie les propri\'et\'es d'une norme sur $C_c(\mathcal{G})$. De plus, pour toutes fonctions ${f,g}$ dans ${C_c(\mathcal{G})}$, on a 
$$\lVert{f}\lVert_I=\lVert{f^*}\lVert_I\quad\textrm{ et }\quad\lVert{f\star{g}}\lVert_I\leqslant\lVert{f}\lVert_I\lVert{g}\lVert_I$$
\end{Rem}
On rappelle quelques d\'efinitions sur les repr\'esentations de $C^*$-alg\`ebres :
\begin{Def}
Soit $B$ une $C^*$-alg\`ebre et $H$ un espace de Hilbert. Une $\ast$-repr\'esentation de $B$ dans $H$ est un $\ast$-homomorphisme $\pi : B\to{L(H)}$. Une $\ast$-repr\'esentation est dite fid\`ele si elle est injective.
\end{Def}
D'apr\`es le th\'eor\`eme de Gelfand, pour toute $C^*$-alg\`ebre $B$, il existe une espace de Hilbert $H$ et une repr\'esentation fid\`ele $\pi : B\to{L(H)}$. Si la $C^*$-alg\`ebre $B$ est s\'eparable, alors $H$ peut \^etre l'unique (\`a isom\'etrie pr\`es) espace de Hilbert s\'eparable de dimension infinie. Ce th\'eor\`eme signifie que toute $C^*$-alg\`ebre peut \^etre vue comme une sous-alg\`ebre involutive ferm\'ee (pour la norme des op\'erateurs) de $L(H)$.

\begin{Def}
Soient $B$ une $\ast$-alg\`ebre de Banach et $(\pi_\alpha)_{\alpha\in\Lambda}$ une famille de $\ast$-repr\'esentations continues de $B$. Le compl\'et\'e de l'alg\`ebre de Banach involutive $B$ par la semi-norme 
$$\lVert{x}\lVert:=\sup_{\alpha\in\Lambda}\lVert{\pi_\alpha(x)}\lVert$$
est une $C^*$-alg\`ebre appel\'ee la $C^*$-alg\`ebre enveloppante de $B$. 
\end{Def}

\begin{Rem}
On s'int\'eresse aux $\ast$-repr\'esentations de $C_c(\mathcal{G})$ dans un espace de Hilbert $H$ c'est-\`a-dire aux $\ast$-homomorphismes
$$\pi : C_c(\mathcal{G})\longrightarrow{\mathcal{B}(H)}$$
continus pour la topologie faible sur $\mathcal{B}(H)$ et tels que l'espace vectoriel engendr\'e par $\{\pi(f)\xi\textrm{, }f\in{C_c(\mathcal{G})}\textrm{, }\xi\in{H}\}$ est dense dans $H$.
\end{Rem}

\begin{Def}
Une $\ast$-repr\'esentation $\pi$ est continue pour la norme $\lVert{.}\lVert_I$  s'il existe $M>0$ tel que pour toute fonction ${f}$ de ${C_c(\mathcal{G})}$ on a 
$$\lVert{\pi(f)}\lVert\leq{M}\lVert{f}\lVert_I$$
\end{Def}

\subsubsection{$C^*$-alg\`ebres r\'eduites et pleines d'un groupo\"{\i}de}{\label{Rappel}}

On consid\`ere $(\pi_\alpha)_{\alpha\in\Lambda(\mathcal{G})}$ l'ensemble des $\ast$-repr\'esentations continues et $I$-born\'ees
de $C_c(\mathcal{G})$. On utilise ces $\ast$-repr\'esentations pour d\'efinir des $C^*$-normes. 

\begin{Def}
Soit $\mathcal{G}$ un groupo\"{\i}de s\'epar\'e localement compact muni d'un syst\`eme de Haar $\nu$ et $C_c(\mathcal{G})$ la $\ast$-alg\`ebre topologique associ\'ee.
\\
La $C^*$-alg\`ebre pleine $C^*(\mathcal{G})$ est l'alg\`ebre enveloppante de $\overline{C_c(\mathcal{G})}^{\lVert{.}\lVert_I}$ c'est-\`a-dire le compl\'et\'e de l'alg\`ebre de Banach involutive $\overline{C_c(\mathcal{G})}^{\lVert{.}\lVert_I}$ pour la semi-norme d\'efinie pour toute fonction ${f}$ dans $\overline{C_c(\mathcal{G})}^{\lVert{.}\lVert_I}$ par 
$$\lVert{f}\lVert:=\sup_{\alpha\in{\Lambda(\mathcal{G})}}\lVert{\pi_\alpha(f)}\lVert$$
\vspace{+3mm}
\end{Def}
Pour d\'efinir la $C^*$-alg\`ebre r\'eduite du groupo\"{\i}de, on a besoin des quelques remarques qui suivent : on consid\`ere $x$ un \'el\'ement de $\mathcal{G}^{(0)}$ et $L^2\big(\mathcal{G}^x,\nu^x\big)$ l'espace de Hilbert obtenu en compl\'etant $C_c(\mathcal{G}^x)$ pour la norme 
$$\lVert{\xi}\lVert_2=\bigg(\int_{\mathcal{G}^{x}}\lvert\xi(\gamma)\lvert^2d\nu^x(\gamma)\bigg)^{1/2}$$
On note $\pi_x : C_c(\mathcal{G})\rightarrow{\mathcal{L}(L^2(\mathcal{G}^x,\nu^x))}$ la $\ast$-repr\'esentation non-d\'eg\'en\'er\'ee d\'efinie pour tout ${f}$ dans ${C_c(\mathcal{G})}\textrm{ et tout }{g}$ dans ${L^2(\mathcal{G}^x,\nu^x)}$ par
$$(\pi_x(f)(g))(\gamma):=\int_{\mathcal{G}^x}f(\gamma')g(\gamma'\gamma)d\nu^x(\gamma')=f\star{g}(\gamma)$$
Pour tout $x$ dans $\mathcal{G}^{(0)}$, la $\ast$-repr\'esentation $\pi_x$ v\'erifie 
$$\lVert{\pi_x(f)}\lVert\leq{\lVert{f}\lVert_I}$$

\begin{Def}
La $C^*$-alg\`ebre r\'eduite $C^*_r(\mathcal{G})$ est la completion de $C_c(\mathcal{G})$ pour la $C^*$-norme r\'eduite $\|.\|_r$ d\'efinie pour toute fonction ${f}$ de ${C_c(\mathcal{G})}$ par
 $$\|f\|_r:=\sup_{x\in{\mathcal{G}^{(0)}}}\|\pi_x(f)\|$$
\vspace{+1mm}
\end{Def}

\subsection{Modules de Hilbert et repr\'esentation r\'eguli\`ere}

On fait un rappel succinct sur les modules de Hilbert. On trouve de nombreuses r\'ef\'erences sur ce sujet notamment les ouvrages \cite{Lance95}, \cite{ManTro05} ou encore \cite{Kasparov80}. On rappelle quelques d\'efinitions concernant les op\'erateurs sur les modules de Hilbert ainsi que la construction du produit tensoriel interne.
\begin{Def}
Soit $B$ une $C^*$-alg\`ebre et $\mathcal{E}$ et $\mathcal{F}$ des $B$-modules de Hilbert. 
\begin{itemize}
\item[$\bullet$] Un op\'erateur $T : \mathcal{E}\rightarrow{\mathcal{F}}$ est une application qui admet un adjoint s'il existe une application $T^\ast : \mathcal{F}\rightarrow\mathcal{E}$ telle que,   $$\langle{T}\xi,\eta\rangle=\langle\xi,T^\ast\eta\rangle$$ pour tout $\xi\in\mathcal{E}$ et tout $\eta\in\mathcal{F}$.
$T^\ast$ est appel\'e l'adjoint de $T$. 
\item[$\bullet$] On note $\mathcal{L}(\mathcal{E},\mathcal{F})$ l'ensemble des applications de $\mathcal{E}$ dans $\mathcal{F}$ qui admettent un adjoint. On note $\mathcal{L}(\mathcal{E})$ pour $\mathcal{L}(\mathcal{E},\mathcal{E})$.
\end{itemize}
\end{Def}
\begin{Rem}
Une application qui admet un adjoint est automatiquement $\CC$-lin\'eaire et m\^eme $B$-lin\'eaire et est un op\'erateur born\'e.
\vspace{+2mm}
\end{Rem}
Comme dans le cas des espaces de Hilbert, il existe des op\'erateurs de rang un dans $\mathcal{L}({E,F})$. Soient $\xi$ dans $E$ et $\eta$ dans $F$, on pose  $\theta_{\xi,\eta}$ l'op\'erateur de $\mathcal{L}(E,F)$ d\'efini pour tout $\zeta$ dans $E$ par $\theta_{\xi,\eta}'(\zeta)=\xi\langle{\eta,\zeta}\rangle$ et tel que $\theta^\ast_{\xi,\eta}=\theta_{\eta,\xi}$. Si $T$ est un op\'erateur dans $\mathcal{L}(F,G)$, alors on a $T\circ\theta_{\xi,\eta}=\theta_{T\xi,\eta}$ et si $T$ est dans $\mathcal{L}(G,E)$, alors $\theta_{\xi,\eta}\circ{T}=\theta_{\xi,T^\ast\eta}$. On note $\mathcal{K}_0(E,F)$  l'espace vectoriel engendr\'e par les op\'erateurs de rang un et la fermeture de $\mathcal{K}_0(E,F)$ est appel\'ee l'ensemble des op\'erateurs "compacts" du module de Hilbert $E$ vers le module de Hilbert $F$ (m\^eme s'ils ne sont pas compacts comme op\'erateurs entre espaces de Banach) et est not\'e $\mathcal{K}(E,F)$. On note $\mathcal{K}(E)$ pour l'id\'eal ferm\'e $\mathcal{K}(E,E)$ de la $C^*$-alg\`ebre $\mathcal{L}(E)$.  
\\

Soient $B_1$ et $B_2$ des $C^*$-alg\`ebres et pour $i$ dans $\{1,2\}$, on consid\`ere $E_i$ un $B_i$-module de Hilbert muni du produit scalaire $\langle{-,-}\rangle_i$ \`a valeurs dans $B_i$. On suppose qu'il existe un $\ast$-homomorphisme $\phi : B_1\to\mathcal{L}_{B_2}(E_2)$. Le produit tensoriel interne est une construction qui permet d'associer \`a $E_1$, $E_2$ et $\phi$ un $B_2$-module de Hilbert not\'e $E_1\otimes_\phi{E_2}$ : pour cela on consid\`ere $E_2$ muni de la structure de $B_1$-module d\'efinie pour tout $b_1$ dans $B_1$ et $e_2$ dans $E_2$ par $b_1.e_2=\phi(b_1)(e_2)$ et on note $E_1\odot_{B_1}E_2$ le produit tensoriel alg\'ebrique muni d'une structure de $B_2$-module \`a droite. Le produit tensoriel interne $E_1\otimes_\phi{E_2}$ est le compl\'et\'e du produit tensoriel alg\'ebrique $E_1\odot{E_2}$ pour le produit scalaire $\langle{-,-}\rangle$ \`a valeur dans $B_2$ d\'efini pour tout $e_1$ et $e_1'$ dans $E_1$ et tout $e_2$ et $e_2'$ dans $E_2$ par
$$\langle{e_1\otimes{e_2},e_1'\otimes{e_2'}}\rangle=\big\langle{e_2,\langle{e_1,e_1'}\rangle_1.e'_2}\big\rangle_2$$ 
avec $\langle{e_1,e_1'}\rangle_1.e'_2=\phi\big(\langle{e_1,e_1'}\rangle_1\big)(e'_2)$.
\vspace{+2mm}
\\
%
%
Soit $\mathcal{G}\rightrightarrows{X}$ un groupo\"{\i}de localement compact et s\'epar\'e muni d'un syst\`eme de Haar $\nu$. On consid\`ere sur l'espace vectoriel complexe $C_c(\mathcal{G})$ le produit scalaire \`a valeurs dans $C_0(X)$ d\'efini pour toutes fonction $\xi$ et $\zeta$ dans $C_c(\mathcal{G})$ par 
$$\langle{\xi,\zeta}\rangle : x\to\int_{\mathcal{G}^x}\xi^\ast(\gamma)\zeta(\gamma^{-1})d\nu^x(\gamma)$$
qui correspond en fait \`a la restriction \`a l'espace $X$ de l'application $\xi^\ast\star\zeta$. On d\'efinit une structure de $C_0(X)$-module (\`a gauche) sur $C_c(\mathcal{G})$ en posant, $\xi.f(\gamma)=\xi(\gamma)f\big(s(\gamma)\big)$, pour toutes fonctions $\xi$ dans $C_c(\mathcal{G})$, $f$ de $C_0(X)$ et tout $\gamma$ dans $\mathcal{G}$. On obtient alors un $C_0(X)$-module pr\'ehilbertien et on note $L^2(\mathcal{G},\nu)$ le $C_0(X)$-module de Hilbert obtenu en quotientant $C_c(\mathcal{G})$ par le sous-module des \'el\'ements de norme nulle.
\vspace{+2mm}
\\


\begin{Pro}{\label{ProAlgMod}}
L'ensemble $\mathcal{L}_{C_0(X)}\big(L^2(\mathcal{G},\nu)\big)$ des op\'erateurs $C_0(X)$-lin\'eaires qui admettent un adjoint est muni d'une structure de $C_0(X)$-alg\`ebre. 
\vspace{+1mm}
\end{Pro}
\begin{proof}
L'ensemble $\mathcal{L}_{C_0(X)}\big(L^2(\mathcal{G},\nu)\big)$ des op\'erateurs $C_0(X)$-lin\'eaires sur $L^2(\mathcal{G},\nu)$ qui admettent un adjoint est muni d'une structure de $C^*$-alg\`ebre.
\\
Par d\'efinition, les op\'erateurs de $\mathcal{L}_{C_0(X)}\big(L^2(\mathcal{G},\nu)\big)$ sont $C_0(X)$-lin\'eaires, c'est-\`a-dire que 
$$T(\xi.f)=T(\xi).f$$
pour tout $T$ dans $\mathcal{L}_{C_0(X)}(L^2(\mathcal{G},\nu))$, $f$ dans ${C(X)}$ et $\xi$ dans ${L^2(\mathcal{G},\nu)}$. L'alg\`ebre des multiplicateurs  $M\big(\mathcal{L}_{C_0(X)}(L^2(\mathcal{G},\nu))\big)$ \'etant \'egal \`a $\mathcal{L}_{C_0(X)}(L^2(\mathcal{G},\nu))$, on pose
$$\phi : C_0(X)\rightarrow{\mathcal{Z}\big(M(\mathcal{L}_{C_0(X)}(L^2(\mathcal{G},\nu)))\big)}$$ 
d\'efinie par $\phi(f)(T)(\xi):=T(\xi.f)=T(\xi).f$, pour tout $T$ dans $\mathcal{L}_{C_0(X)}(L^2(\mathcal{G},\nu))$, $f$ dans ${C_0(X)}$ et $\xi$ dans ${L^2(\mathcal{G},\nu)}$.
\\
L'homomorphisme $\phi : C_0(X)\rightarrow\mathcal{L}_{C_0(X)}\big(L^2(\mathcal{G},\nu)\big)$ \'etant non-d\'eg\'en\'er\'e, il d\'efinit une structure de $C_0(X)$-alg\`ebre sur $\mathcal{L}_{C_0(X)}\big(L^2(\mathcal{G},\nu)\big)$.
\end{proof}

%
%
%
\begin{Rem}
Comme  $C_c(\mathcal{G})$ agit sur lui-m\^eme par convolution, on a un homomorphisme injectif $\lambda : C_c(\mathcal{G})\hookrightarrow{\mathcal{L}_{C_0(X)}\big(L^2(\mathcal{G},\nu)\big)}$ qui s'\'etend en une $\ast$-repr\'esentation de $\lambda : C^*(\mathcal{G})\to\mathcal{L}_{C_0(X)}\big(L^2(\mathcal{G},\nu)\big)$ dont l'image est $\ast$-isomorphe \`a la $C^*$-alg\`ebre r\'eduite $C_r^\ast(\mathcal{G})$. On obtient une repr\'esentation fid\`ele de la $C^*$-alg\`ebre $C^*_r(\mathcal{G})$ dans le $C_0(X)$-module de Hilbert $\mathcal{L}_{C_0(X)}\big(L^2(\mathcal{G},\nu)\big)$. L'image par la repr\'esentation $\lambda$ de $C_r^*(\mathcal{G})$ est une sous-$C^*$-alg\`ebre de $\mathcal{L}_{C_0(X)}\big(L^2(\mathcal{G},\nu)\big)$, qui n'est pas n\'ecessairement stable par la structure de $C_0(X)$-module de $\mathcal{L}_{C_0(X)}\big(L^2(\mathcal{G},\nu)\big)$.
\end{Rem}

\section{{Groupo\"{\i}de moyennable \`a l'infini et espace universel}}

Soit $\mathcal{G}\rightrightarrows{X}$ un groupo\"{\i}de \'etale localement compact et s\'epar\'e tel que l'espace des unit\'es $X$ est $\sigma$-compact. Le r\'esultat principal de cette section est la d\'emonstration du th\'eor\`eme {\ref{TheoPrinc3}} que l'on donne ci-dessous :
\newtheorem{Theopopop}{\bfseries{Th\'eor\`eme}}
\begin{Theopopop}
\itshape{Soit $\mathcal{G}$ un groupo\"{\i}de \'etale d'espace des unit\'es $X$. On consid\`ere les assertions suivantes :
\begin{enumerate}
\item[$(1)$] le groupo\"{\i}de $\mathcal{G}$ est moyennable \`a l'infini  
\item[$(2)$] le groupo\"{\i}de $\mathcal{G}$ agit moyennablement sur l'espace $\beta_X\mathcal{G}$
\item[$(3)$] le groupo\"{\i}de $\mathcal{G}$ est exact 
\item[$(4)$] la $C^\ast$-alg\`ebre $C^\ast_r(\mathcal{G})$ est exacte  
\end{enumerate} 
Alors on a $(1)\Longrightarrow{(2)}\Longrightarrow{(3)}\Longrightarrow{(4)}$}
\vspace{+3mm}
\end{Theopopop}

 Il s'agit d'un r\'esultat analogue \`a la proposition \ref{ProEquMoyInf} dans le cas plus g\'en\'eral des groupo\"{\i}des \'etales localement compacts et s\'epar\'es. Les parties \ref{Partie1} \`a \ref{Partie6} qui suivent permettent de d\'efinir une version de la moyennabilit\'e \`a l'infini pour les groupo\"{\i}des ainsi qu'un espace $\beta_X\mathcal{G}$ muni d'une action continue du groupo\"{\i}de $\mathcal{G}$ qui satisfait une propri\'et\'e universelle et qui joue le r\^ole du compactifi\'e de Stone-Cech dans le cas des groupes discrets.  On obtient alors le th\'eor\`eme suivant :

\begin{Theo}{\label{TheoEquivEspace}}
Soit $\mathcal{G}$ un groupo\"{\i}de \'etale localement compact et $X$ l'espace localement compact des unit\'es du groupo\"{\i}de $\mathcal{G}$. Si le groupo\"{\i}de $\mathcal{G}$ est moyennable \`a l'infini, alors le groupo\"{\i}de $\beta_X\mathcal{G}\rtimes\mathcal{G}$ est moyennable.
\end{Theo}


On donne tout d'abord une d\'efinition de la moyennabilit\'e infinie pour les groupo\"{\i}des : 
\begin{Def}{\label{DefMoyInf}}
Soit $\mathcal{G}$ un groupo\"{\i}de \'etale localement compact et $X$ l'espace des unit\'es (localement compact) du groupo\"{\i}de $\mathcal{G}$. Le groupo\"{\i}de est dit moyennable \`a l'infini s'il existe un $\mathcal{G}$-espace localement compact et s\'epar\'e $Y$, pour lequel l'application $\pi : Y\rightarrow{X}$ est continue, surjective, ouverte et propre, et tel que le groupo\"{\i}de $Y\rtimes\mathcal{G}$ est moyennable.
\vspace{+2mm}
\end{Def}

Pour d\'emontrer le th\'eor\`eme \ref{TheoEquivEspace} , on raisonne de la mani\`ere suivante : 
\begin{itemize}
\vspace{+1mm}
\item[1)] il s'agit tout d'abord de d\'efinir l'espace $\beta_X\mathcal{G}$ comme \'etant le spectre d'une $C^\ast$-alg\`ebre commutative qui sera not\'ee $C_0^s(\mathcal{G})$
\vspace{+1mm}
\item[2)] on munit l'espace topologique $\beta_X\mathcal{G}$ d'une action continue du groupo\"{\i}de $\mathcal{G}$ et on prouve que $\beta_X\mathcal{G}$ v\'erifie une propri\'et\'e  universelle
\vspace{+1mm}
\item[3)] \`a l'aide de la propri\'et\'e universelle de $\beta_X\mathcal{G}$ et des propri\'et\'es des groupo\"{\i}des moyennables, on d\'emontre de th\'eor\`eme \ref{TheoEquivEspace} .
\vspace{+3mm}
\end{itemize}

La partie \ref{Partie1} de cette section est constitu\'ee de quelques r\'esultats g\'en\'eraux sur les $C^*$-alg\`ebres qui nous seront utiles dans la suite.
\\

La  partie \ref{Partie2} consiste \`a d\'efinir la $C^*$-alg\`ebre commutative $C_0^s(\mathcal{G})$ : il s'agit d'une sous-$C^*$-alg\`ebre de la $C^*$-alg\`ebre commutative des fonctions \`a valeurs complexes, continues et born\'ees sur $\mathcal{G}$. On peut alors d\'efinir une structure de $C_0(X)$-alg\`ebre sur $C_0^s(\mathcal{G})$ et la munir d'une action continue du groupo\"{\i}de $\mathcal{G}$. On obtient ainsi un syst\`eme dynamique de groupo\"{\i}de $\big(C_0^s(\mathcal{G}),\mathcal{G},\alpha\big)$.
\\

La partie \ref{Partie3} porte sur l'\'etude de l'espace $\beta_X\mathcal{G}$ qui est d\'efini comme le spectre de $C_0^s(\mathcal{G})$. Le fait que $C_0^s(\mathcal{G})$ soit une $\mathcal{G}$-alg\`ebre commutative induit une action continue de $\mathcal{G}$ sur l'espace localement compact et s\'epar\'e $\beta_X\mathcal{G}$. On termine cette partie en prouvant que l'espace $\beta_X\mathcal{G}$ v\'erifie la propri\'et\'e universelle suivante : soit $Z$ un $\mathcal{G}$-espace localement compact et s\'epar\'e tel que $Z\to{X}$ est continue, surjective, propre et admet une section continue (non n\'ecessairement \'equivariante), alors il existe une application $\tilde{\theta} : \beta_X\mathcal{G}\to{Z}$ continue, $\mathcal{G}$-\'equivariante et propre.
\\

La propri\'et\'e universelle v\'erifi\'ee par $\beta_X\mathcal{G}$ est le r\'esultat essentiel qui nous permet de d\'emontrer le th\'eor\`eme \ref{TheoEquivEspace} : en effet, si $\mathcal{G}$ est moyennable \`a l'infini, alors il existe un $\mathcal{G}$-espace $Y$ localement compact et s\'epar\'e tel que l'application $Y\to{X}$ est continue, surjective, ouverte et propre et sur lequel l'action de $\mathcal{G}$ est moyennable. On ne peut appliquer dans ce cas la propri\'et\'e universelle car la condition d'existence d'une section de $Y\to{X}$ n'est pas (n\'ecessairement) v\'erifi\'ee. C'est l'objet des parties \ref{Partie4} et \ref{Partie5} qui consiste \`a d\'efinir un $\mathcal{G}$-espace localement compact et s\'epar\'e $M$ obtenu en consid\'erant un sous espace des mesures de probabilit\'e de $Y$. On s'emploie alors \`a d\'emontrer l'existence d'une section continue pour l'application $M\to{X}$ et prouver ainsi que l'espace $M$ satisfait les conditions de la propri\'et\'e universelle de $\beta_X\mathcal{G}$ : on a alors une application $\tilde{\theta} : \beta_X\mathcal{G}\to{M}$ continue, $\mathcal{G}$-\'equivariante et propre. 
\\

Dans la partie \ref{Partie6} , on prouve que l'action de $\mathcal{G}$ sur $M$ est moyennable et on montre comment on peut "tir\'e en arri\`ere"  le caract\`ere moyennable de $M\rtimes\mathcal{G}$ sur $\beta_X\mathcal{G}\rtimes\mathcal{G}$ par le biais de l'application $\tilde{\theta} : \beta_X\mathcal{G}\to{M}$. Enfin on termine cette section en d\'emontrant dans la partie \ref{Partie7} le th\'eor\`eme \ref{TheoPrinc3}.

\subsection{Quelques propri\'et\'es des $C^*$-alg\`ebres commutatives}\label{Partie1}
On d\'emontre ici quelques r\'esultats g\'en\'eraux dans le cadre des $C^*$-alg\`ebres commutatives, qui seront utilis\'es dans cette section.
\begin{Def}
Soit $i : C_0(X)\rightarrow{C_0(Y)}$ un homomorphisme de $C^*$-alg\`ebres; on dit que $i$ est non d\'eg\'en\'er\'e si pour toute approximation de l'unit\'e, $(e_\lambda)_{\lambda\in\Lambda}$, de $C_0(X)$, la suite g\'en\'eralis\'ee $\big(i(e_\lambda)\big)_{\lambda\in\Lambda}$ est une approximation de l'unit\'e de $C_0(Y)$.
\end{Def}

\begin{Pro}{\label{ProNonDeg}}
Si un homomorphisme $i : C_0(X)\rightarrow{C_0(Y)}$ de $C^*$-alg\`ebres est non d\'eg\'en\'er\'e, alors il existe une application $i^* : Y\rightarrow{X}$, induite par $i$.
\end{Pro}

\begin{proof}
On consid\`ere $\chi$ un caract\`ere, non nul, de $C_0(Y)$ c'est-\`a-dire un homomorphisme non nul $\chi : C_0(Y)\rightarrow{\CC}$. Il existe alors un unique $y\in{Y}$ tel que pour toute fonction $f$ dans $C_0(Y)$, on a $\chi(f)=f(y)$. On pose $\chi':=\chi\circ{i}$. C'est un homomorphisme (car composition d'homomorphismes) de $C_0(X)$ dans $\CC$, qui est non nul; en effet, supposons que $\chi'$ est l'homomorphisme nul. On consid\`ere $(e_\lambda)_{\lambda\in\Lambda}$ une approximation de l'unit\'e de $C_0(X)$. Comme l'homomorphisme $i : C_0(X)\rightarrow{C_0(Y)}$ est non d\'eg\'en\'er\'e, la suite $\big(i(e_\lambda)\big)_{\lambda\in{\Lambda}}$ est une approximation de l'unit\'e de $C_0(Y)$, donc on a 
$$\lim_{\lambda\rightarrow\infty}\chi\big(i(e_\lambda)\big)=\lim_{\lambda\rightarrow\infty}i(e_\lambda)(y)=1$$
Or on a suppos\'e $\chi'=\chi\circ{i}$, d'o\`u $\chi(i(e_\lambda))=\chi'(e_\lambda)=0$, pour tout $\lambda$ dans $\Lambda$. On obtient une contradiction. Donc l'homomorphisme $\chi' : C_0(X)\rightarrow{\CC}$ est non nul. Il existe alors un unique $x$ dans $X$ v\'erifiant pour toute fonction $g$ de $C_0(X)$, l'\'egalit\'e $\chi'(g)=g(x)$.
\\
Ainsi, par l'homomorphisme $i : C_0(X)\rightarrow{C_0(Y)}$, non d\'eg\'en\'er\'e, \`a tout $y$ dans $Y$, on peut associer un unique \'el\'ement $x$ de $X$. On a donc une application $i^* : Y\rightarrow{X}$, induite par $i$.  
\\
\end{proof}
\begin{Pro}{\label{ProInjectiveNonDeg}}
Soit $i : C_0(X)\hookrightarrow{C_0(Y)}$ un homomorphisme de $C^*$-alg\`ebres et $i^\ast : Y\rightarrow{X}$ l'application induite par $i$.
Si $i$ est injectif et non d\'eg\'en\'er\'e, alors l'application $i^\ast : Y\rightarrow{X}$ est continue, surjective et propre.
\end{Pro}
\begin{proof}
Pour montrer que l'application est continue, on consid\`ere $U$ un ouvert de $X$. On pose $V=(i^\ast)^{-1}(U)$ l'image r\'eciproque de $U$ par l'application $i^\ast$. Soit $y_0$ un \'el\'ement de $V$ et $x_0=i^\ast(y_0)$. On consid\`ere une fonction $f$ dans $C_0(X)$, \`a valeurs r\'eelles positives et \`a support compact dans $U$ telle que $f(x_0)=1$. On a alors une fonction $i(f)$ dans $C_0(Y)$, \`a valeurs r\'eelles positives et \`a support dans $V$. Soit $\varepsilon>0$, le sous espace $U_\varepsilon:=\{x\in{X}\textrm{ / }1-\varepsilon<f(x)\}$ est un sous espace ouvert dans $U$ contenant $x_0$. On a alors $({i^\ast})^{-1}(U_\varepsilon)=\{y\in{V}\textrm{ / }1-\varepsilon<i(f)(y)\}$, qui est un ouvert dans $V$ contenant $y_0$. Ainsi $V$ est un ouvert de $Y$ et l'application $i^\ast : Y\rightarrow{X}$ est continue.
\\

Pour montrer que l'application $i^\ast : Y\rightarrow{X}$ est propre, on consid\`ere $K\subset{X}$ un sous espace compact et on montre que $(i^\ast)^{-1}(K)$ est compact dans $Y$. On consid\`ere $f$ une fonction dans ${C_0(X)}$ telle que $f(x)=1$, pour tout $x$ dans $K$. L'application $i(f)$ est alors dans ${C_0(Y)}$ et on a, par d\'efinition, $i(f)(y)=1$ pour tout $y$ dans $(i^\ast)^{-1}(K)$.
\\
De plus, $i(f)$ est dans ${C_0(Y)}$, alors il existe $\tilde{K}$ un sous espace compact de $Y$ tel que pour tout $y$ dans ${Y\setminus\tilde{K}}$, on a $\lvert{f(y)}\lvert<1$; le sous espace $(i^\ast)^{-1}(K)$ est ferm\'e par continuit\'e de $i^\ast$ et contenu dans le compact ${\tilde{K}}$, donc $(i^\ast)^{-1}(K)$ est compact. L'application $i^\ast : Y\rightarrow{X}$ est donc propre.
\\

Pour montrer que l'application $i^\ast$ est surjective, on prouve tout d'abord que son image $Im(i^\ast)$ est dense dans $X$. Soit $U$ un ouvert non vide de $X$. Montrons qu'il existe $x$ dans ${U\cap{Im(i^\ast)}}$; en effet, si $Im(i^\ast)\cap{U}=\varnothing$, alors toute fonction $f$ de ${C_0(X)}$, non nulle, et \`a support compact dans $U$ a pour image par $i$ la fonction nulle sur $Y$; or ceci est en contradiction avec l'injectivit\'e de $i : C_0(X)\rightarrow{C_0(Y)}$. Ainsi pour tout ouvert $U$ de $X$, il existe un \'el\'ement $x$ qui soit dans l'image $Im(i^\ast)$. Donc $i^\ast$ est \`a image dense dans $X$. L'application $i^\ast$ \'etant propre, son image $Im(i^\ast)$ est ferm\'ee. L'image de $i^\ast$ \'etant ferm\'ee et dense, alors l'application $i^\ast$ est surjective. 
\end{proof}


\subsection  {$\mathcal{G}$-alg\`ebre}{\label{Galgebre}}\label{Partie2}

Dans cette section, on consid\`ere $\mathcal{G}\rightrightarrows{X}$ un groupo\"{\i}de \'etale localement compact et s\'epar\'e. On suppose que l'espace des unit\'es $X$ est $\sigma$-compact. On note $r : \mathcal{G}\rightarrow{X}$ et $s : \mathcal{G}\rightarrow{X}$ les applications but et source. On note $C_b(\mathcal{G})$ l'alg\`ebre de Banach des fonctions continues et born\'ees sur $\mathcal{G}$, o\`u l'involution et la multiplication sont d\'efinies de la mani\`ere suivante : 
$$f^\ast(\gamma):=\overline{f(\gamma)}\quad\textrm{et}\quad f.g(\gamma):=f(\gamma)g(\gamma)$$
pour toutes fonctions $f$ et $g$ dans $C_b(\mathcal{G})$ et tout $\gamma$ dans $\mathcal{G}$. On note $\lVert{.}\lVert_\infty$ la norme sup\'erieure, d\'efinie, pour tout $f$ dans $C_b(\mathcal{G})$, par $\lVert{f}\lVert_\infty:=\sup_{\gamma\in\mathcal{G}}\lvert{f(\gamma)}\lvert$. Munie de cette norme, l'alg\`ebre $C_b(\mathcal{G})$ est une $C^\ast$-alg\`ebre.
\\

Dans la d\'efinition qui suit, on introduit  un sous ensemble de $C_b(\mathcal{G})$ que l'on note $C_0^s(\mathcal{G})$ et on montre que c'est une sous-$C^*$-alg\`ebre de $C_b(\mathcal{G})$. L'objectif dans cette partie est de munir la $C^*$-alg\`ebre $C_0^s(\mathcal{G})$ d'une action continue du groupo\"{\i}de $\mathcal{G}$.  Pour cela, on d\'efinit tout d'abord une approximation de l'unit\'e $(e_n)_{n\in\NN}$ de la $C^*$-alg\`ebre $C_0(X)$ en utilisant la $\sigma$-compacit\'e de $X$. A l'aide de l'application source $s : \mathcal{G}\to{X}$, on construit ensuite un $\ast$-homomorphisme $s^\ast : C_0(X)\to\mathcal{Z}{\big(M(C_0^s)\big)}$ et on se sert de l'approximation de l'unit\'e $(e_n)_{n\in\NN}$ pour montrer que $s^\ast$ est non-d\'eg\'en\'er\'ee et donc que $C_0^s(\mathcal{G})$ est une $C_0(X)$-alg\`ebre. Pour finir, on d\'efinit une action continue du groupo\"{\i}de $\mathcal{G}$ sur la $C_0(X)$-alg\`ebre $C_0^s(\mathcal{G})$ en utilisant les propri\'et\'es d\^ues au caract\`ere \'etale du groupo\"{\i}de $\mathcal{G}$.

\begin{Def}
On note $C_0^s(\mathcal{G})$ le sous ensemble de $C_b(\mathcal{G})$ form\'e des fonctions $f$ continues et born\'ees telles que pour tout r\'eel $\varepsilon$ strictement positif, il existe un sous espace compact $K_\varepsilon$ de ${X}$ tel que pour tout $\gamma$ dans ${\mathcal{G}}$, tel que $s(\gamma)$ est dans $X\smallsetminus{K_\varepsilon}$, on a $\lvert f(\gamma)\lvert<\varepsilon$. 
\end{Def}

\begin{Pro}
$C^s_0(\mathcal{G})$ est une sous $C^\ast$-alg\`ebre de $C_b(\mathcal{G})$. 
\end{Pro}
\begin{proof}
Il est clair que $C_0^s(\mathcal{G})$ est un sous espace vectoriel de $C_b(\mathcal{G})$ stable par l'involution. Pour toute fonction $f$ (respectivement $g$) de $C_0^s(\mathcal{G})$ et pour tout r\'eel $\varepsilon$ strictement positif, il existe un sous espace compact $K_{\varepsilon}$ (respectivement $K'_\varepsilon$) de $X$ tel que pour tout $\gamma$ dans $\mathcal{G}$ tel que $s(\gamma)$ est dans $X\smallsetminus{K_\varepsilon}$ (respectivement $X\smallsetminus{K'_\varepsilon}$) on a $\lvert{f(\gamma)}\lvert<\sqrt\varepsilon$ (respectivement $\lvert{g(\gamma)}\lvert<\sqrt\varepsilon$). Ainsi pour tout $\gamma$ de $\mathcal{G}$ tel que $s(\gamma)$ est dans $X\smallsetminus{(K_\varepsilon\cup{K'_\varepsilon})}$, on a $\lvert{f.g(\gamma)}\lvert<\varepsilon$ et $C_0^s(\mathcal{G})$ est stable par le produit. 
\\
On consid\`ere $(f_n)_{n\in\NN}$ une suite de fonctions de $C_0^s(\mathcal{G})$ qui converge vers une fonction $f$ de $C_b(\mathcal{G})$. Soit $\varepsilon$ un r\'eel strictement positif : il existe un entier $n_1$ tel que, pour tout $n>n_1$, on a $\lVert{f-f_n}\lVert_\infty<\varepsilon/2$. De plus il existe un entier $n_2$, tel que pour tout entier $p,q>n_2$, on a $\lVert{f_p-f_q}\lVert_\infty<\varepsilon/2$ et il existe donc un compact $K_\varepsilon$ dans $X$ tel que pour tout $\gamma$ dans $\mathcal{G}$ pour lequel $s(\gamma)$ est dans $X\smallsetminus{K_\varepsilon}$, on a $\lvert{f_p(\gamma)}\lvert<\varepsilon/2$, pour tout $p>n_2$. On a alors pour tout r\'eel $\varepsilon$ strictement positif, il existe un entier $n=\max\{n_1,n_2\}$ et un compact $K_\varepsilon$ de $X$ tel que pour tout $\gamma$ de $\mathcal{G}$ avec $s(\gamma)$ dans $X\smallsetminus{K_\varepsilon}$, on a
$$\lvert{f(\gamma)}\lvert=\lvert{f(\gamma)-f_p(\gamma)+f_p(\gamma)}\lvert\leq{\lvert{f(\gamma)-f_p(\gamma)}\lvert+\lvert{f_p(\gamma)}\lvert}<\varepsilon$$
pour $p>n$. L'alg\`ebre $C_0^s(\mathcal{G})$ est ferm\'ee pour la norme $\lVert{.}\lVert_\infty$ et est une sous $C^\ast$-alg\`ebre de $C_b(\mathcal{G})$.
\\
\end{proof}
\begin{Rem}
A partir de maintenant, on \'ecrit $C_0^s$ pour la $C^*$-alg\`ebre $C_0^s(\mathcal{G})$, ceci afin de simplifier les notations.  
\end{Rem}

On construit tout d'abord une approximation de l'unit\'e pour la $C^*$-alg\`ebre $C_0(X)$. L'espace $X$ \'etant consid\'er\'e $\sigma$-compact, on choisit $(K_n)_{n\in\NN}$, une suite croissante de sous espaces compacts de $X$, tels que $K_n\subset{\mathring{K}_{n+1}}$. On consid\`ere $(e_n)_{n\in\NN}$ une suite de fonctions continues \`a support compact dans $X$ telles que pour tout $n\in\NN$
$$
e_n(x) = \left\{
    \begin{array}{ll}
        1 & \mbox{si } x\in{K_n} \\
        0 & \mbox{si } x\notin{\mathring{K}_{n+1}}
    \end{array}
\right.
$$ 
$(e_n)_{n\in\NN}$ est une approximation de l'unit\'e de $C_0(X)$, puisque chaque fonction $e_n$ est positive et dans la boule unit\'e de $C_0(X)$; de plus la suite $(e_n)_{n\in\NN}$ est croissante et on a, pour toute fonction $f$ de $C_0(X)$, 
$$\lim_{n\rightarrow\infty}{\lVert{f-f.e_n}\lVert}=0$$
\\

On utilise l'approximation de l'unit\'e $(e_n)_{n\in\NN}$ construite auparavant, pour d\'efinir une structure de $C_0(X)$-alg\`ebre sur $C_0^s$. Pour cela, on consid\`ere une fonction $f$ de $C_0(X)$ et on pose $s^\ast(f)$ la fonction sur $\mathcal{G}$ d\'efinie, pour tout $\gamma$ dans $\mathcal{G}$, par $s^\ast\big(f\big)(\gamma):=f\big(s(\gamma)\big)$. La fonction $s^\ast(f)$ est clairement continue et born\'ee. Pour tout r\'eel $\varepsilon$ strictement positif, il existe $K_\varepsilon$ sous espace compact de $X$ tel que pour tout $x$ dans $X\smallsetminus{K_\varepsilon}$, on a $\lvert{f(x)}\lvert<\varepsilon$. Ainsi pour tout $\gamma$ dans $\mathcal{G}$ pour lequel $s(\gamma)$ est dans $X\smallsetminus{K_\varepsilon}$, on a $\lvert{s^\ast(f)(\gamma)}\lvert=\lvert{f(s(\gamma))}\lvert<\varepsilon$. La fonction $s^\ast(f)$ est donc dans la $C^\ast$-alg\`ebre $C_0^s$. On obtient une application $s^\ast : C_0(X)\rightarrow{C_0^s}$ dont il est facile de prouver qu'elle est un homomorphisme de $C^*$-alg\`ebres.
\begin{Pro}{\label{InjectifNonDeg}}
L'homomorphisme $s^\ast : C_0(X)\rightarrow{C_0^s}$ est injectif et non-d\'eg\'en\'er\'e.
\end{Pro}
\begin{proof}
On sait que $s^\ast : C_0(X)\rightarrow{C_0^s}$ est un homomorphisme de $C^\ast$-alg\`ebres.
\\
On consid\`ere $f$ une fonction de $C_0(X)$ v\'erifiant $s^\ast(f)=0$, alors pour tout $\gamma\in{\mathcal{G}}$, on a $s^\ast(f)(\gamma)=0=f(s(\gamma))$. Or pour tout $x$ dans $X$, on a $u(x)\in{\mathcal{G}}$, et 
$$f(x)=f(s(u(x)))=s^\ast(f)(u(x))=0$$
La fonction $f$ est nulle sur $X$. Ainsi $s^\ast$ est injectif. 
\\
Pour prouver que l'homomorphisme $s^\ast : C_0(X)\rightarrow{C_0^s}$ est non-d\'eg\'en\'er\'e, il suffit de montrer que la suite $(s^\ast(e_n))_{n\in\NN}$ est une approximation de l'unit\'e de $C_0^s$. Chaque fonction $s^\ast(e_n)$ est positive et se trouve dans la boule unit\'e de $C_0^s$, car $s^\ast$ est un homomorphisme injectif donc une application positive et isom\'etrique. De plus, la suite $(s^\ast(e_n))_{n\in\NN}$ est croissante. 
\\
Soit $f$ une fonction de $C_0^s$ alors pour tout $\varepsilon>0$, il existe $n'$ dans $\NN$ tel que pour tout $n\geq{n'}$ et pour tout ${\gamma}$ hors de $\mathcal{G}_{K_{n}}$, on a $\lvert{f(\gamma)}\lvert<\varepsilon$. On a donc pour tout $n\geq{n'}$ et tout $\gamma$ dans $\mathcal{G}_{K_n}$
\begin{align*}
\lvert{f(\gamma)-f(\gamma)\theta(e_n)(\gamma)}\lvert&=\lvert{f(\gamma)-f(\gamma)e_n(s(\gamma))}\lvert 
\\
& = \lvert{f(\gamma)-f(\gamma)}\lvert
\\
& =0
\end{align*}
et pour tout $n\geq{n'}$ et tout $\gamma$ en dehors de $\mathcal{G}_{K_n}$, on a
\begin{align*}
\lvert{f(\gamma)-f(\gamma)s^\ast(e_n)(\gamma)}\lvert&=\lvert{f(\gamma)-f(\gamma)e_n(s(\gamma))}\lvert 
\\
&\leq\vert{f(\gamma)}\lvert\lvert{1-e_n(s(\gamma))}\lvert
\\
&<{\varepsilon}
\end{align*}
Ainsi pour toute fonction $f$ dans $C_0^s$ et pour tout r\'eel $\varepsilon$ strictement positif, il existe un entier $n'$ tel que pour tout $n\geq{n'}$ et tout pour tout $\gamma$ de $\mathcal{G}$, on a 
$$\lvert{f(\gamma)-f(\gamma)s^\ast(e_n)(\gamma)}\lvert<\varepsilon$$
c'est-\`a-dire que pour toute fonction $f$ de $C_0^s$, on a 
$$\lim_{n\rightarrow\infty}\lVert{f-fs^\ast(e_n)}\lVert=0$$ 
Donc $(s^\ast(e_n))_{n\in\NN}$ est une approximation de l'unit\'e dans $C_0^s$. L'homomorphisme $s^\ast$ est donc non-d\'eg\'en\'er\'e.
\\
\end{proof}

\begin{Rem}
$C_0^s$ \'etant une $C^*$-alg\`ebre commutative, on a une injection de $C_0^s$ dans $\mathcal{Z}\big(\mathcal{M}(C_0^s)\big)$ , le centre de l'alg\`ebre des multiplicateurs.
\\
\end{Rem}

\begin{Pro}
La $C^*$-alg\`ebre commutative $C_0^s$ est munie d'une structure de $C_0(X)$-alg\`ebre par l'homomorphisme $s^\ast : C_0(X)\rightarrow{\mathcal{Z}\big(M(C_0^s)\big)}$.
\end{Pro}
\begin{proof}
Comme $C_0^s$ s'injecte dans le centre $\mathcal{Z}\big(M(C_0^s)\big)$ de l'alg\`ebre des multiplicateurs de $C_0^s$, alors $s^\ast$ d\'efinit un homomorphisme de $C_0(X)$ dans $\mathcal{Z}\big(M(C_0^s)\big)$. De plus, l'image $(s^\ast(e_n))_{n\in\NN}$ de l'approximation de l'unit\'e $(e_n)_{n\in\NN}$ de $C_0(X)$ \'etant une approximation de l'unit\'e de $C_0^s$, on a alors $s^\ast(C_0(X))C_0^s$ est dense dans $C_0^s$; donc $C_0^s$ est une $C_0(X)$-alg\`ebre par l'homomorphisme $s^\ast$.   
\\
\end{proof}

On construit ci-dessous un syst\`eme dynamique de groupo\"{\i}de $(C_0^s,\mathcal{G},\alpha)$ : pour cela, le caract\`ere \'etale du groupo\"{\i}de nous permet de d\'efinir une action $\alpha$ du groupo\"{\i}de $\mathcal{G}$ sur la $C_0(X)$-alg\`ebre $C_0^s$ et on prouve que cette action est continue. On rappelle quelques notations concernant les $C_0(X)$-alg\`ebres que l'on utilise ici :  pour tout $x$ dans $X$, on note $I_x:=\overline{s^\ast(C_x(X))C_0^s}$  id\'eal ferm\'e de $C_0^s$ et $C_0^s(x):=C_0^s/{I_x}$, la fibre de $C_0^s$ sur $x$. 
Pour tout $x$ dans $X$ et toute fonction $f$ de $C_0^s$, on notera $f(x)$ l'image de $f$ dans la fibre $C_0^s(x)$.
\\

On d\'efinit une action du groupo\"{\i}de $\mathcal{G}$ sur la $C_0(X)$-alg\`ebre $C_0^s$ c'est-\`a-dire qu'on construit une famille $(\alpha_\gamma)_{\gamma\in\mathcal{G}}$ telle que pour tout $\gamma$ dans $\mathcal{G}$, l'application 
$$\alpha_\gamma: C_0^s\big(s(\gamma)\big)\rightarrow{C_0^s\big(r(\gamma)\big)}$$
est un isomorphisme v\'erifiant, pour tout $(\gamma,\eta)$ dans $\mathcal{G}^{(2)}$ l'\'egalit\'e $\alpha_{\gamma.\eta}=\alpha_\gamma\circ\alpha_\eta$. 

Soit $\gamma$ un \'el\'ement de $\mathcal{G}$ et $U_\gamma$ un voisinage ouvert (que l'on peut supposer relativement compact) de $\gamma$ dans $\mathcal{G}$ tel que les restrictions des applications but et source \`a $U_\gamma$ sont des hom\'eomorphismes sur leur image respective $r(U_\gamma)$ et $s(U_\gamma)$. On note $x={s(\gamma)}$ et $y=r(\gamma)$ dans $X$ et on pose $\rho_x : s(U_\gamma)\rightarrow{U_\gamma}$ et $\rho_y : r(U_\gamma)\rightarrow{U_\gamma}$ les hom\'eomorphismes r\'eciproques des restrictions sur $U_\gamma$ respectivement des applications source et but. Il existe alors un hom\'eomorphisme $R_\gamma : r(U_{\gamma})\rightarrow{s(U_\gamma)}$ d\'efini, pour tout $z$ dans $r(U_\gamma)$, par $R_\gamma(z):=s\big(\rho_y(z)\big)$. On consid\`ere $K_x$ un voisinage compact de $x$ dans $s(U_\gamma)$ et $\varphi_x$ une fonction continue et \`a support dans $s(U_\gamma)$ v\'erifiant $\chi_{K_x}\leq\varphi_x\leq{1}$. 
\\

Soit $f$ une fonction de $C_0^s$ et $f(x)$ son image dans la fibre $C_0^s(x)$. On a alors \'egalit\'e dans la $C^\ast$-alg\`ebre $C_0^s(x)$ des \'el\'ements $f(x)$ et $\big(s^\ast(\varphi_x)f\big)(x)$ 
\\
Pour tout $\eta$ dans $\mathcal{G}_{r(U_\gamma)}$, on pose 
\begin{align}{\label{AlignAction}}
\tilde{f}(\eta):=\big(s^\ast(\varphi_x)f\big)\big(\eta.\rho_{y}(s(\eta))\big)
\end{align}
qui d\'efinit une fonction continue sur $\mathcal{G}_{r(U_\gamma)}$ et pour tout $\eta$ tel que $s(\eta)$ est en dehors de l'ouvert $r(U_\gamma)$ relativement compact, on a $\tilde{f}(\eta)=0$. Ainsi $\tilde{f}$ est dans la $C_0(X)$-alg\`ebre $C_0^s$.  
\\
Pour toute fonction $f$ dans $C_0^s$, on pose $\alpha_\gamma\big(f(x)\big):=\tilde{f}(y)$. On montre ais\'ement que pour toute fonction $f$ dans $C_0^s$ et toute paire $(\gamma,\eta)$ dans $\mathcal{G}^{(2)}$, on a l'\'egalit\'e
$$\alpha_{\gamma.\eta}(f)=\alpha_\gamma\big(\alpha_\eta(f)\big)$$
On a ainsi d\'efini une action du groupo\"{\i}de $\mathcal{G}$ sur la $C_0(X)$-alg\`ebre $C_0^s$.
\begin{Rem}
Sur la $C_0(X)$-alg\`ebre $C_0(\mathcal{G})$, il existe une action continue du groupo\"{\i}de $\mathcal{G}$, not\'ee $a$, o\`u pour tout $\gamma$ dans $\mathcal{G}$, on a 
$$a_\gamma : C_0(\mathcal{G})(s(\gamma))\cong{C_0(\mathcal{G}_{s(\gamma)})}\longrightarrow{C_0(\mathcal{G})(r(\gamma))\cong{C_0(\mathcal{G}_{r(\gamma)})}}$$
d\'efinie pour toute fonction $f$ dans $C_0(\mathcal{G}_{s(\gamma)})$ par $a_\gamma(f)(\eta)=f(\eta.\gamma)$, pour tout $\eta$ dans $\mathcal{G}_{r(\gamma)}$. On voit alors que la restriction de l'action $\alpha$, d\'efinie ci-dessus, \`a la sous $C_0(X)$-alg\`ebre $C_0(\mathcal{G})$ coincide avec l'action continue $a$ de $\mathcal{G}$ sur $C_0(\mathcal{G})$. 
\end{Rem}

Pour montrer que l'action de $\mathcal{G}$ sur $C_0^s$ est continue, on se place dans le contexte du champ continu de $C^*$-alg\`ebres $\mathcal{C}_0^s$ associ\'e \`a $C_0^s$ et on montre que l'application
$$\mathcal{C}_0^s\ast\mathcal{G}\longrightarrow\mathcal{C}_0^s$$ d\'efinie par l'action de $\mathcal{G}$ sur l'espace topologique $\mathcal{C}^s_0$ est continue. On utilise la proposition C.20 de {\cite{Williams07}}. 
\\
On consid\`ere une suite $(a_i,\gamma_i)_{i\in{I}}$ dans $\mathcal{C}_0^s\ast_{p,s}\mathcal{G}$ convergeant vers un \'el\'ement $(a,\gamma)$ de $\mathcal{C}_0^s\ast_{p,s}\mathcal{G}$. Pour tout $i$ dans $I$, on note $x_i:=p(a_i)=s(\gamma_i)$ et $x:=p(a)=s(\gamma)$. Pour tout $i$ dans $I$, il existe une fonction $f_i$ dans $C_0^s$ telle que $f_i(x_i)=a_i$ et il existe une fonction $f$ dans $C_0^s$ telle que $f(x)=a$. On consid\`ere $U$ un voisinage ouvert de $\gamma$ dans $\mathcal{G}$ hom\'eomorphe \`a $V:=s(U)$, voisinage ouvert de $x$ et on peut supposer qu'il existe un compact $K_x$ d'int\'erieur non vide et inclus dans $V$. On note $\varphi_x : V\rightarrow\RR$ telle que $\chi_{K_x}\leq\varphi_x\leq{1}$. Il existe $i'$ dans $I$ tel que pour tout $i\geq{i'}$, l'\'el\'ement $x_i$ est dans $\mathring{K}_x$. Quitte \`a utiliser la fonction $\varphi_x$, on peut supposer alors qu'il existe que pour tout $i\geq{i'}$, le support de $f_i$ est inclus dans $s^{-1}(V)$. Suivant la construction $\ref{AlignAction}$, on obtient des fonctions $\tilde{f}_i$ et $\tilde{f}$ dans $C_0^s$ telles que, pour tout $i\geq{i'}$, on a $\tilde{f}_i(r(\gamma_i))=\alpha_{\gamma_i}(a_i)$ et $\tilde{f}(r(\gamma))=\alpha_\gamma(a)$. Pour tout $i\geq{i'}$, on note $u_i=\tilde{f}(r(\gamma_i))$ et la suite $(u_i)_{i\geq{i'}}$ converge vers $\alpha_\gamma(a)=\tilde{f}(r(\gamma))$ dans $\mathcal{C}_0^s$. Soit un r\'eel $\varepsilon$ strictement positif, il existe $i_\varepsilon$ dans $I$, tel que pour tout $i\geq{\max\{i',i_\varepsilon\}}$, on a $\lVert{\tilde{f}_i(r(\gamma_i))-\tilde{f}(r(\gamma_i))}\lVert<\varepsilon/2$ et alors $\alpha_{\gamma_i}(a_i)-\tilde{f}(r(\gamma_i))$ converge vers $0_{r(\gamma)}$ dans $\mathcal{C}_0^s$. On a alors
$$
\xymatrix{ \alpha_{\gamma_i}(a_i)=\big(\alpha_{\gamma_i}(a_i)-u_i\big)+u_i \ar[rr] & & {0_{r(\gamma)}+\alpha_\gamma(a)} & {}}
$$
L'action du groupo\"{\i}de $\mathcal{G}$ sur la $C_0(X)$-alg\`ebre $C_0^s$ est continue et on a un syst\`eme dynamique $(C_0^s,\mathcal{G},\alpha)$.

\subsection  {Espace $\beta_X\mathcal{G}$ et propri\'et\'e universelle}\label{Partie3}

Dans cette partie, on utilise la $\mathcal{G}$-alg\`ebre commutative $C_0^s$ construite auparavant pour d\'efinir l'espace topologique $\beta_X\mathcal{G}$. Une fois l'espace $\beta_X\mathcal{G}$ d\'efini, on montre qu'il est localement compact et s\'epar\'e et muni d'une action continue du groupo\"{\i}de $\mathcal{G}$. 
\\
Enfin on prouve que le $\mathcal{G}$-espace $\beta_X\mathcal{G}$ v\'erifie une propri\'et\'e universelle.

\begin{Def}
On d\'efinit l'espace topologique $\beta_X\mathcal{G}$ comme le spectre de la $C^*$-alg\`ebre commutative $C_0^s$. 
\end{Def}

\begin{Rem}{\label{Gelfand}}
La $C^\ast$-alg\`ebre $C_0(\mathcal{G})$ \'etant un id\'eal (bilat\`ere) ferm\'e de la $C^\ast$-alg\`ebre $C_0^s$, l'espace $\mathcal{G}=\widehat{C_0(\mathcal{G})}$ est un sous espace ouvert de $\beta_X\mathcal{G}$.
Par la transform\'ee de Gelfand, il existe isomorphisme $\ast$-isom\'etrique de $C^*$-alg\`ebres entre $C_0^s$ et $C_0(\beta_X\mathcal{G})$ et on a alors un homomorphisme $s^\ast : C_0(X)\rightarrow{C_0(\beta_X\mathcal{G})}$. On note 
\begin{eqnarray*}
\xymatrix{ \tilde{s} : \beta_X\mathcal{G}=\widehat{C_0(\beta_X\mathcal{G})} \ar[rr] & &  {\widehat{C_0(X)}}=X}
\end{eqnarray*}
l'application induite par l'homomorphisme $s^\ast$. L'application $\tilde{s}$ restreinte \`a l'ouvert $\mathcal{G}$ coincide avec l'application source du groupo\"{\i}de $s : \mathcal{G}\rightarrow{X}$.
\end{Rem}

\begin{Pro}
L'application $\tilde{s} : \beta_X\mathcal{G}\rightarrow{X}$ induite par l'homomorphisme $s^\ast : C_0(X)\rightarrow{C_0(\beta_X\mathcal{G})}$ est continue, propre et surjective. 
\end{Pro}

\begin{proof}
D'apr\`es la proposition {\ref{InjectifNonDeg}} , on sait que l'homomorphisme $s^\ast : C_0(X)\rightarrow{C_0^s}$ est injectif et non-d\'eg\'en\'er\'e et d'apr\`es la remarque $\ref{Gelfand}$, $C_0^s$ est isomorphe \`a $C_0(\beta_X\mathcal{G})$. L'homomorphisme de $C^*$-alg\`ebres $s^\ast : C_0(X)\rightarrow{C_0(\beta_X\mathcal{G})}$ est injectif et non d\'eg\'en\'er\'e. En appliquant la proposition $\ref{ProInjectiveNonDeg}$ , l'application $\tilde{s} : \beta_X\mathcal{G}\rightarrow{X}$ est continue, propre et surjective. 

\end{proof}

\begin{Pro}
L'espace topologique $\beta_X\mathcal{G}$ est localement compact et muni d'une action continue du groupo\"{\i}de $\mathcal{G}$.
\end{Pro}

\begin{proof}
L'espace topologique $\beta_X\mathcal{G}$ \'etant d\'efini comme le spectre de la $C^*$-alg\`ebre $C_0^s$, il est donc localement compact.
\\  
On a montr\'e auparavant l'existence d'une application $\tilde{s} : \beta_X\mathcal{G}\rightarrow{X}$ continue, propre, surjective et ouverte. On pose alors
$$\beta_X\mathcal{G}\ast_r\mathcal{G}=\bigsqcup_{x\in{X}}\widehat{C_0^s(x)}\ast_r\mathcal{G}:=\bigg\{(\chi,\gamma)\in{\bigsqcup_{x\in{X}}\widehat{C_0^s(x)}\times\mathcal{G}}\textrm{ : }\tilde{s}(\chi)=r(\gamma)\bigg\}$$
Pour $(\chi,\gamma)$ dans $\beta_X\mathcal{G}\ast_r\mathcal{G}$, la relation $\tilde{s}(\chi)=r(\gamma)$ implique que $\chi$ est dans $\widehat{C_0^s(r(\gamma))}$. 
\\
On pose $\phi : \beta_X\mathcal{G}\ast_r\mathcal{G}\rightarrow{\beta_X\mathcal{G}}$ une application telle que, pour tout couple $(\chi,\gamma)$ dans $\beta_X\mathcal{G}\ast_r\mathcal{G}$, on a $\phi(\chi,\gamma):=\chi.\gamma$ o\`u $\chi.\gamma\in{C_0^s(s(\gamma))}$ est tel que 
$$\forall{f}\in{C_0^s\big(s(\gamma)\big)}\textrm{,}\qquad \chi.\gamma(f):=\chi\big(\alpha_\gamma(f)\big) \qquad$$
On obtient donc une action \`a droite continue du groupo\"{\i}de $\mathcal{G}$ sur l'espace $\beta_X\mathcal{G}$.
\\
\end{proof}

La fin de cette partie consiste \`a d\'emontrer que l'espace $\beta_X\mathcal{G}$ est un espace universel pour la propri\'et\'e suivante :
\begin{Pro}{\label{UnivSpace}}
Soit $Z$ un $\mathcal{G}$-espace localement compact pour lequel l'application $\pi : Z\rightarrow{X}$ est continue, propre, surjective et admet une section continue (qui n'est pas n\'ecessairement $\mathcal{G}$-\'equivariante) $\sigma : X\rightarrow{Z}$, 
alors il existe une application $\tilde\theta : \beta_X\mathcal{G}\rightarrow{Z}$ continue, $\mathcal{G}$-\'equivariante et propre.
\end{Pro}

 Soit $Z$ un $\mathcal{G}$-espace localement compact pour lequel l'application $\pi : Z\rightarrow{X}$ est continue, propre, surjective et admet une section continue (non n\'ecessairement $\mathcal{G}$-\'equivariante) not\'ee $\sigma : X\rightarrow{Z}$. On pose $\theta : \mathcal{G}\rightarrow{Z}$ l'application d\'efinie pour tout $\gamma\in\mathcal{G}$ par
$$\theta(\gamma):=\sigma(r(\gamma)).\gamma$$
Cette application est bien d\'efinie car pour tout $\gamma\in\mathcal{G}$, on a $\sigma(r(\gamma))\in{Z_{r(\gamma)}}$, et en appliquant l'action de $\gamma$ sur $\sigma(r(\gamma))$, on obtient bien $\sigma(r(\gamma)).\gamma$ un \'el\'ement de ${Z_{s(\gamma)}}$.
\\
L'application $\theta$ est continue car compos\'ee d'applications continues et $\mathcal{G}$-\'equivariante car pour toute paire $(\gamma_1,\gamma_2)$ de $\mathcal{G}^{(2)}$, on a 
\begin{align*}
\theta(\gamma_1\gamma_2)&=\sigma(r(\gamma_1\gamma_2)).(\gamma_1\gamma_2)
\\
&=\big(\sigma(r(\gamma_1)).\gamma_1\big).\gamma_2
\\
&=\theta(\gamma_1).\gamma_2
\end{align*}
Pour $\gamma$ dans $\mathcal{G}$, on a $\theta(\gamma)=\sigma\big(r(\gamma)\big).\gamma$ est dans le sous espace $Z_{s(\gamma)}:=\pi^{-1}\big(s(\gamma)\big)$ compact dans $Z$. 

\begin{Lemma}
L'application continue $\theta : \mathcal{G}\rightarrow{Z}$ induit un homomorphisme de $C^*$-alg\`ebres $\theta^* : C_0(Z)\rightarrow{C_0^s}$.
\end{Lemma}

\begin{proof}
L'application $\theta : \mathcal{G}\rightarrow{Z}$ induit un homomorphisme de $C^*$-alg\`ebres $\theta^* : C_0(Z)\rightarrow{C_b(\mathcal{G})}$ d\'efini par 
$$\theta^*(f)(\gamma):=f\circ{\theta}(\gamma)=f\big(\sigma(r(\gamma)).\gamma\big)$$
pour toute fonction $f$ de $C_0(Z)$ et tout $\gamma$ de $\mathcal{G}$. 
\\
Soient une fonction $f$ dans $C_0(Z)$ et $\varepsilon>0$, il existe alors un sous espace compact $K_\varepsilon$ dans $Z$ tel que $\lvert{f(z)}\lvert<\varepsilon$ pour tout $z$ en dehors de $K_\varepsilon$. 
\\
En posant $K_\varepsilon':=\pi(K_\varepsilon)$, on obtient un sous espace compact de $X$ et pour tout $\gamma$ n'appartenant pas \`a l'espace ${\mathcal{G}_{K_\varepsilon'}}$, l'\'el\'ement $\theta(\gamma)=\sigma(r(\gamma)).\gamma$ se trouve en dehors de ${K_\varepsilon}$, par cons\'equent 
$$\big\lvert{\theta^*(f)(\gamma)}\big\lvert=\big\lvert{f\big(\sigma(r(\gamma)).\gamma\big)}\big\lvert<\varepsilon$$
\\
 On a montr\'e que pour tout $\varepsilon>0$, il existe un sous espace $K_\varepsilon'$ compact dans $X$ tel que pour tout $\gamma$ dans $\mathcal{G}$ avec $s(\gamma)\notin{K_\varepsilon'}$, on a $\lvert\theta^*(f)(\gamma)\lvert<\varepsilon$. Donc $\theta^*(f)$ est dans $C_0^s$. On obtient alors un homomorphisme de $C^*$-alg\`ebre, $\theta^* : C_0(Z)\rightarrow{C_0^s}$. 
 \\
\end{proof}

\begin{Lemma}
L'homomorphisme $\theta^* : C_0(Z)\rightarrow{C_0^s}$ est $\mathcal{G}$-\'equivariant.
\end{Lemma}

\begin{proof}
Les $C^*$-alg\`ebres $C_0(Z)$ et $C_0^s$ sont des $\mathcal{G}$-alg\`ebres dont les actions sont donn\'ees par les familles d'isomorphismes $C^*$-alg\`ebres $(\delta_\gamma)_{\gamma\in\mathcal{G}}$ avec $\delta_\gamma : C_0(Z)(s(\gamma))\rightarrow{C_0(Z)(r(\gamma))}$ et $(\rho_\gamma)_{\gamma\in\mathcal{G}}$, avec $\rho_\gamma : C_0^s(s(\gamma))\rightarrow{C_0^s(r(\gamma))}$. 
\\
Pour tout $x$ dans $X$, on pose $\theta^*_x : C_0(Z)(x)\rightarrow{C_0^s(x)}$ l'homomorphisme de $C^*$-alg\`ebres induit par $\theta^* : C_0(Z)\rightarrow{C_0^s}$ sur les fibres en $x$.
\\
On a pour toute fonction $g$ de $C_0(Z)(r(\gamma))$ et tout $\eta$ dans $\mathcal{G}_{r(\gamma)}$
\begin{align*}
\rho_\gamma\circ\theta^*_{r(\gamma)}(g)(\eta)&=\rho_\gamma\big(\theta^*_{r(\gamma)}(g)\big)(\eta)=\theta^*_{r(\gamma)}(g)(\eta.\gamma)=g\big(\theta(\eta.\gamma)\big)
\\
&=g\big(\theta(\eta).\gamma\big)=\delta_\gamma(g)\big(\theta(\eta)\big)
\\
&=\theta^*_{s(\gamma)}\circ\delta_\gamma(g)(\eta)
\end{align*}  
Ainsi pour tout $\gamma$ de $\mathcal{G}$, on a $\rho_\gamma\circ\theta^*_{r(\gamma)}=\theta^*_{s(\gamma)}\circ\delta_\gamma$, par cons\'equent, l'homomorphisme $\theta^* : C_0(Z)\rightarrow{C_0^s}$ est $\mathcal{G}$-\'equivariant.
\\
\end{proof}

%

\begin{proof}[D\'emonstration de la propri\'et\'e \ref{UnivSpace} ]
En \'etudiant la $C^*$-alg\`ebre $C_0^s$, on a montr\'e qu'il existe un $\ast$-isomorphisme isom\'etrique de $C_0^s$ dans $C_0(\beta_X\mathcal{G})$. On obtient un homomorphisme de $C^*$-alg\`ebres $\theta^* : C_0(Z)\rightarrow{C_0(\beta_X\mathcal{G})}$. 
\\
On consid\`ere $(e_n)_{n\in\NN}$ une approximation de l'unit\'e de $C_0(Z)$ et on veut montrer que $(\theta^*(e_n))_{n\in\NN}$ est une approximation de l'unit\'e de $C_0(\beta_X\mathcal{G})$. Soit $f$ une fonction dans $C_0(\beta_X\mathcal{G})$; pour tout $\varepsilon>0$, il existe un sous espace $K$ compact dans $X$, tel que pour tout $\gamma$ dans $\mathcal{G}\setminus\mathcal{G}_{K}$, on a $\lvert{f(\gamma)}\lvert<\varepsilon$. En posant $K':=\overline{\big\{\sigma\big(r(\gamma)\big).\gamma\textrm{ / }s(\gamma)\in{K}\big\}}\subset{\pi^{-1}(K)}$, on obtient un sous espace compact de $Z$. Il existe $n_0$ dans $\NN$ tel que pour tout $n\geq{n_0}$, on a pour tout $z$ dans $\pi^{-1}(K)$  $1\leq{e_n(z)}\leq{1-\varepsilon}$. 
Par cons\'equent, on a pour tout $\gamma$ dans $\mathcal{G}_{K}$, $1\leq{\theta^*(e_n)(\gamma)}\leq{1-\varepsilon}$ ce qui implique 
$$\big\lvert{f(\gamma)-\theta^*(e_n)(\gamma)f(\gamma)}\big\lvert\leq\varepsilon\big\lvert{f(\gamma)}\big\lvert$$
De m\^eme, pour $n\geq{n_0}$ pour $\gamma$ n'appartenant pas \`a $\mathcal{G}_K$, puisque $\lvert{f(\gamma)}\lvert\leq\varepsilon$ et ${\theta^*(e_n)(\gamma)}\in{[0,1]}$, on a 
$$\big\lvert{f(\gamma)-\theta^*(e_n)(\gamma)f(\gamma)}\big\lvert<\varepsilon$$
Ainsi on a montr\'e que pour $f$ fonction de $C_0(\beta_X\mathcal{G})$, pour $\varepsilon>0$, il existe $n_0$ dans $\NN$ tel que pour tout $n\geq{n_0}$, on a 
$$\big\lVert{f-\theta^*{(e_n)}f}\big\lVert_\infty<\varepsilon\big\lVert{f}\big\lVert_\infty$$
On en conclut que $(\theta^*(e_n))_{n\in\NN}$ est une approximation de l'unit\'e de $C_0(\beta_X\mathcal{G})$ et l'homomorphisme $\theta^* : C_0(Z)\rightarrow{C_0(\beta_X\mathcal{G})}$ est non d\'eg\'en\'er\'e.
\\

D'apr\`es la proposition {\ref{ProNonDeg}}, l'homomorphisme non d\'eg\'en\'er\'e $\theta^* : C_0(Z)\rightarrow{C_0(\beta_X\mathcal{G})}$ induit une application $\tilde\theta : \beta_X\mathcal{G}\rightarrow{Z}$. 
\\

L'application $\tilde\theta : \beta_X\mathcal{G}\rightarrow{Z}$ est continue : en effet, on consid\`ere $U$ un ouvert de $Z$ et on pose $V=\tilde\theta^{-1}(U)$. Si $V$ est vide, alors c'est un ouvert de $\beta_X\mathcal{G}$. S'il existe $p_0$ dans $V$, on pose $z_0=\tilde\theta(p_0)$, son image dans $U$. Soit $f$ une fonction de $C_0(Z)$ \`a valeurs dans $[0,1]$, \`a support dans $U$ et pour laquelle $f(z_0)=1$. La fonction $\theta^*(f)$ de $C_0(\beta_X\mathcal{G})$ est \`a valeurs dans $[0,1]$, non nulle car $\theta^*(f)(p_0)=1$ et \`a support dans V. On a alors $\theta^*(f)^{-1}\big(]0,1]\big)$ est un ouvert dans $\beta_X\mathcal{G}$ contenant $p_0$ et contenu dans $V$. Donc $V=\tilde\theta^{-1}(U)$ est un ouvert de $\beta_X\mathcal{G}$. L'application $\tilde{\theta} : \beta_X\mathcal{G}\rightarrow{Z}$ est continue.  
\\

L'application $\tilde\theta : \beta_X\mathcal{G}\rightarrow{Z}$ est $\mathcal{G}$-\'equivariante : en effet pour tout \'el\'ement $(p,\gamma)$ dans $\beta_X\mathcal{G}\ast_{\tilde{s},r}\mathcal{G}$ et toute fonction $f$ dans $C_0(Z)$, on a 
\begin{align*}
f\big(\tilde\theta(p.\gamma)\big)&=\theta^*(f)(p.\gamma)=\theta^*_{r(\gamma)}(f)(p.\gamma)=\rho_\gamma\big(\theta^*_{r(\gamma)}(f)\big)(p)
\\
&=\theta^*_{s(\gamma)}\big(\delta_\gamma(f)\big)(p)=\delta_\gamma(f)\big(\tilde\theta(p)\big)=f\big(\tilde\theta(p).\gamma\big)
\end{align*} 
Pour tout couple $(p,\gamma)$ dans $\beta_X\mathcal{G}$, on a l'\'egalit\'e $\tilde\theta(p.\gamma)=\tilde\theta(p).\gamma$, ce qui prouve que l'application $\tilde\theta$ est $\mathcal{G}$-\'equivariante.
\\

Soit $K$ un sous espace compact de $Z$, montrons que $\tilde\theta^{-1}(K)$ est un sous espace compact de $\beta_X\mathcal{G}$. 
Le sous espace $\tilde\theta^{-1}(K)$ est ferm\'e dans $\beta_X\mathcal{G}$ et $\tilde\theta^{-1}(K)$ est inclus dans ${\tilde{s}^{-1}(\pi(K))}$ car pour tout ${p}$ dans ${\tilde\theta^{-1}(K)}$ il existe ${z}$ dans ${K}$ tel que  : 
$$\tilde\theta(p)=z\quad\textrm{et}\quad\tilde{s}(p)=\pi(\tilde{\theta}(p))$$
donc $p\in{\tilde{s}^{-1}(\pi(K))}$. Or l'application $\pi$ \'etant continue et l'application $\tilde{s}$ \'etant propre, alors $\tilde{s}^{-1}(\pi(K))$ est un sous espace compact de $\beta_X\mathcal{G}$. On en d\'eduit que $\tilde\theta^{-1}(K)$ est un sous espace ferm\'e du compact $\tilde{s}^{-1}(\pi(K))$, donc $\tilde\theta^{-1}(K)$ est compact. L'application $\tilde\theta$ est donc propre.
\\
\end{proof}

%



\subsection  {Espace de mesure}\label{Partie4}

La moyennabilit\'e \`a l'infinie du groupo\"{\i}de $\mathcal{G}$ se caract\'erise par l'existence d'un $\mathcal{G}$-espace localement compact et s\'epar\'e $Y$ tel que $\pi : Y\rightarrow{X}$ une application continue, propre, surjective et ouverte pour lequel l'action est moyennable. Pour d\'emontrer le th\'eor\`eme \ref{TheoEquivEspace} , \`a savoir que la moyennabilit\'e infinie du groupo\"{\i}de $\mathcal{G}$ implique la moyennabilit\'e du groupo\"{\i}de $\beta_X\mathcal{G}\rtimes\mathcal{G}$, l'outil essentiel est la propri\'et\'e universelle de $\beta_X\mathcal{G}$. Le premier r\'eflexe serait d'appliquer cette propri\'et\'e universelle au $\mathcal{G}$-espace $Y$ mais ceci n'est pas possible car l'application $\pi : Y\to{X}$ n'admet en g\'en\'eral pas de section continue.
\\
 L'objectif de cette partie et de la suivante vise \`a remplacer l'espace $Y$ par un $\mathcal{G}$-espace $M$ de mesures de probabilit\'e sur $Y$ judicieusement choisi de sorte que l'espace $M$ poss\`ede les m\^emes caract\'eristiques que l'espace $Y$ et que l'on puisse construire une section continue pour l'application $\tilde{\pi} : M\to{X}$ induite par l'action de $\mathcal{G}$. 
\\

Dans cette partie, on s'emploie tout d'abord \`a d\'efinir l'espace $M$ de mesures sur $Y$ puis on montre que l'espace $M$ est muni d'une action continue du groupo\"{\i}de $\mathcal{G}$ pour lequel l'application $\tilde{\pi} : M\to{X}$ est continue, surjective, ouverte et propre.
\\

Pour tout $x\in{X}$, on note $Y_x=\pi^{-1}(x)$. On note $R(Y)$ l'espace des mesures de Radon sur $Y$.
On pose alors 
$$M=\big\{\mu\in{R(Y)}\textrm{ : }\exists{x}\in{X}\textrm{, }Supp(\mu)\subset{Y_x}\textrm{ et }\mu(Y)=1\big\}$$
 
L'ensemble $M$ est un sous espace de la boule unit\'e de l'espace $R(Y)$. Pour toute mesure $\mu$ dans $M$, il existe un unique $x_\mu$ dans $X$ tel que $Supp(\mu)$ est inclus dans $Y_{x_\mu}$. On note $\tilde{\pi} : M\rightarrow{X}$ l'application canonique qui \`a toute mesure $\mu$ de $M$ associe l'\'el\'ement $\tilde{\pi}(\mu)=x_\mu$ de $X$. 


\begin{Pro}{\label{tildepicanonique}}
L'application canonique $\tilde\pi : M\rightarrow{X}$ est continue, surjective, ouverte et propre.
\end{Pro}

\begin{proof} On montre successivement que l'application est surjective, continue, ouverte et propre.
\\

La surjectivit\'e est assez directe : pour toute mesure $\mu\in{M}$, il existe $x_\mu\in{X}$ tel que $Supp\textrm{ }\mu\subset{Y_{x_\mu}}$. L'application canonique $\tilde\pi : M\rightarrow{X}$ est d\'efinie par $\tilde\pi(\mu)=x_\mu$. Elle est clairement surjective. 
\\

 Pour montrer la continuit\'e de l'application $\tilde\pi : M\rightarrow{X}$, on consid\`ere $U$ un ouvert dans $X$ et on cherche \`a montrer que $\tilde\pi^{-1}(U)$ est un ouvert de $M$; soit $\mu_0$ un \'el\'ement de $\tilde\pi^{-1}(U)$ et $x_0:=\tilde\pi(\mu)$. Soit $U_0$ un ouvert dans $U$ contenant $x_0$; puisque l'application $\pi : Y\rightarrow{X}$ est continue alors $\pi^{-1}(U_0)$ est un ouvert de $Y$; soit $f\in{C_c(\pi^{-1}(U_0))}$ une fonction continue \`a support compact dans $\pi^{-1}(U_0)$ et \`a valeurs dans $[0,1]$ telle que $f\lvert_{Y_{\mu_0}}\equiv{1}$. Alors l'ensemble
$$\big\{\mu\in{M}\textrm{ : }\lvert{\mu(f)-\mu_0(f)}\lvert<{1}\big\}$$
est un ouvert dans $\tilde\pi^{-1}(U)$ contenant $\mu_0$. Ainsi $\tilde\pi^{-1}(U)$ est un ouvert de $M$. L'application $\tilde\pi : M\rightarrow{X}$ est donc continue.
\\

 Supposons maintenant que l'application $\tilde\pi$ n'est pas ouverte c'est-\`a-dire qu'il existe un ouvert $U_0$ de $M$ tel que $\tilde\pi{(U_0)}$ n'est pas ouvert dans $X$. On peut supposer que $U_0$ est un voisinage ouvert de $\mu_0$ de la forme 
$$U_0=\big\{\mu\in{M}\textrm{ : }\lvert{\mu(f_i)-\mu_0(f_i)}\lvert<a\textrm{, }\forall{i}\in\{1,..,n\}\big\}$$ 
avec $a$ un r\'eel strictement positif et $f_i$ est une fonction dans ${C_0(Y)}$ pour tout $i$ dans $\{1,..,n\}$. Comme $\tilde\pi(U_0)$ n'est pas ouvert dans $X$, il existe alors une mesure $\mu'$ dans ${U_0}$ avec $x'=\tilde\pi(\mu')$ tel que pour tout voisinage $V'$ de $x'$ dans $X$, on a $V'\nsubseteq{\tilde\pi(U_0)}$. On peut donc construire une suite g\'en\'eralis\'ee $(x_\lambda)_{\lambda\in\Lambda}$ dans $X$ qui converge vers $x'$ et telle que pour tout $\lambda$ dans $\Lambda$, l'\'el\'ement $x_\lambda$ n'est pas dans ${\tilde\pi(U_0)}$, c'est-\`a-dire pour tout $\mu_\lambda$ dans ${\tilde\pi^{-1}(x_\lambda)}$, il existe $i$ dans $\{1,...,n\}$ tel que $\lvert{\mu_\lambda(f_i)-\mu_0(f_i)}\lvert\geq{a}$.
\\
On consid\`ere $\varepsilon$ un r\'eel strictement positif tel que pour tout $i$ dans $\{1,...,n\}$ on a 
$$\lvert{\mu'(f_i)-\mu_0(f_i)}\lvert<a-2\varepsilon$$
Pour tout $z$ dans ${Y_{x'}}$, on pose 
$$U_{z}=\cap_{i=1}^n\big\{y\in{Y}\textrm{ : }\lvert{f_i(y)-f_i(z)}\lvert<\varepsilon\big\}$$
qui est un ouvert de $z$ dans $Y$. On obtient un recouvrement d'ouverts du compact $Y_{x'}$ duquel on peut extraire un sous recouvrement fini et on a $Y_{x'}\subset{\cup_{j=1}^m}U_{z_j}$. 
\\
L'application $\pi : Y\rightarrow{X}$ \'etant ouverte, l'espace $\cap_{j=1}^m\pi(U_{z_j})$ est un ouvert de $X$ contenant $x'$. Ainsi il existe $\lambda'\in\Lambda$, tel que pour tout $\lambda>\lambda'$, on a $x_\lambda$ appartient \`a l'ouvert $\cap_{j=1}^m\pi(U_{j})$.
\\
On pose $(\phi_j)_{1\leq{j}\leq{m}}$, une partition de l'unit\'e associ\'ee au recouvrement de l'espace $Y_{x'}$ par les ouverts $\{U_{z_j}\}_{j=1}^m$ . 
\\
On consid\`ere $x_\lambda$ dans $\cap_{j=1}^m\pi(U_{z_j})$. Pour tout $j\in\{1,..,m\}$, il existe $y_j$ dans $U_{z_j}\cap{Y_{x_\lambda}}$ et on construit une mesure $\mu_\lambda$ d\'efinie par $\mu_\lambda=\sum_{j=1}^{m}\omega_j\delta_{j}$, o\`u $\delta_{j}$ est la mesure de Dirac en $y_j$ et $\omega_j=\mu'(\phi_j)$. On obtient une mesure \`a support dans $Y_{x_\lambda}$ et comme
$$\mu_\lambda(Y)=\sum_{j=1}^m\omega_j\delta_{j}(M)=\sum_{j=1}^m\mu'(\phi_j)=\mu'(\sum_j\phi_j)=\mu'(1_{Y_{x'}})=1$$  
la mesure $\mu_\lambda$ est dans l'espace $M$. Pour tout $i$ dans $\{1,...,n\}$, on a 
\begin{align*}
\big\lvert{\mu_\lambda(f_i)-\mu'(f_i)}\big\lvert&=\bigg\lvert{\sum_{j=1}^m\omega_j\delta_j(f_i)-\mu\big(\sum_{j=1}^m\phi_jf_i\big)}\bigg\lvert
\\
&=\bigg\lvert{\sum_{j=1}^m\mu'(\phi_j)f_i(z_j)-\sum_{j=1}^m\mu'(\phi_jf_i)}\bigg\lvert
\end{align*}
et, pour tout $1\leq{j}\leq{m}$, on a 
$(f_i(y_j)-\varepsilon)\phi_j\leq\phi_jf_j\leq{(f_i(y_j)+\varepsilon)\phi_j}$.
\\
Par cons\'equent,
\begin{align*}
\sum_{j=1}^m\mu'(\phi_j)\big[f_i(z_j)-(f_i(y_j)+\varepsilon)\big]\leq{\mu_\lambda(f_i)-\mu'(f_i)}\leq{\sum_{j=1}^m\mu'(\phi_j)\big[f_i(z_j)-(f_i(y_j)-\varepsilon)\big]}
\end{align*}
Par d\'efinition des ouverts $U_j$ construits, pour tout $i$ dans $\{1,...,n\}$, on a 
$$-\varepsilon\leq f_i(z_j)-f_i(y_j)\leq\varepsilon$$
On en d\'eduit alors que pour tout $\lambda\geq\lambda'$, pour tout $j$ dans $\{1,...,m\}$ et tout $i$ dans $\{1,...,n\}$,
$$\big\lvert{\mu_\lambda(f_i)-\mu'(f_i)}\big\lvert\leq{2\varepsilon}$$
Ainsi on a pour tout $i$ dans $\{1,...,n\}$, tout r\'eel $\varepsilon$ strictement positif, il existe $\lambda'$ dans $\Lambda$ tel que 
\begin{align*}
\big\lvert{\mu'(f_i)-\mu_0(f_i)}\big\lvert&=\big\lvert{\mu'(f_i)-\mu_\lambda(f_i)+\mu_\lambda(f_i)-\mu_0(f_i)}\big\lvert
\\
&\geq\bigg\lvert{\big\lvert{\mu_\lambda(f_i)-\mu'(f_i)}\big\lvert-\big\lvert{\mu_\lambda(f_i)-\mu_0(f_i)}\big\lvert}\bigg\lvert
\\
&\geq{a-2\varepsilon}
\end{align*}
On obtient alors une contradiction. L'application $\tilde\pi : M\rightarrow{X}$ est donc ouverte.
\\

Montrons maintenant que l'application est propre : l'espace $M$ est un sous espace ferm\'e de $R(Y)$ pour la topologie $\ast\sigma(R(Y),C_0(Y))$. En effet, soit $(\mu_n)_{n\in\NN}$ est une suite de mesure dans $M$ qui converge vers une mesure $\mu$ de $R(Y)$ pour la toplogie $\ast\sigma(R(Y),C_0(Y))$. On note $S$ le support de $\mu$ qui est non vide (car $\mu$ est dans la boule unit\'e de $R(Y)$) et $(x_n)_{n\in\NN}$ la suite dans $X$ o\`u, pour tout entier $n$, on a $x_n=\tilde{\pi}(\mu_n)$. Soit $x$ un \'el\'ement de $\tilde{\pi}(S)$, alors $x$ est dans l'adh\'erence de $\{x_n : n\in\NN\}$. Supposons (on raisonne par l'absurde) qu'il existe deux \'el\'ements $x$ et $x'$ distincts dans $\tilde{\pi}(S)$. On consid\`ere $U$ et $U'$ des voisinages ouverts et disjoints respectivement de $x$ et $x'$. On consid\`ere $f$ (respectivement $f'$) une fonction \`a support compact dans l'ouvert $\pi^{-1}(U)$ (respectivement $\pi^{-1}(U')$) telle que $\mu(f)=1$ (respectivement $\mu(f')=1$). Pour tout r\'eel $\varepsilon$ strictement positif, il existe un entier $n_0$ tel que pour tout $n>n_0$, on a 
$$\lvert{\mu(f)-\mu_n(f)}\lvert<\varepsilon\quad\textrm{et}\quad{\lvert{\mu(f')-\mu_n(f')}\lvert<\varepsilon}$$
Pour tout $n>n_0$, l'espace $Supp(\mu_n)$ est contenu dans $\pi^{-1}(x_n)$. L'espace $\pi^{-1}(x_n)$ est inclus et dans $\pi^{-1}(U)$ et dans $\pi^{-1}(U')$, ce qui est contradictoire. Donc $\tilde{\pi}(S)$ est r\'eduit \`a un singleton $\{x\}$ dans $X$ et $\mu$ est dans $M$. L'espace $M$ est ferm\'e dans la boule unit\'e de $R(Y)$ qui est compacte par le th\'eor\`eme de Banach-Alaoglu. L'espace $M$ est donc compact et l'application $\tilde{\pi} : M\rightarrow{X}$ est propre. 
\\
\end{proof}

\begin{Pro}{\label{EspaceEquivLocComp}}
L'espace topologique $M$ est un $\mathcal{G}$-espace localement compact.
\end{Pro}

\begin{proof}
L'application $\tilde\pi : M\rightarrow{X}$ \'etant propre, l'espace $M$ est alors localement compact. On munit $M$ d'une action du groupo\"{\i}de $\mathcal{G}$ : on consid\`ere l'espace $M\ast\mathcal{G}$ comme \'etant l'ensemble $\big\{(\mu,\gamma)\in{M\times\mathcal{G}}\textrm{ : }\tilde\pi(\mu)=r(\gamma)\big\}$
et on note $\theta : M\ast\mathcal{G}\rightarrow{M}$ l'application qui \`a tout $(\mu,\gamma)$ de $M\ast\mathcal{G}$, associe $\theta((\mu,\gamma))=\mu.\gamma\in{M}$ o\`u $\mu.\gamma$ est d\'efinie pour toute fonction $f$ dans ${C_0(Y)}$ par
\[\mu.\gamma(f)=\int_Yf(y)d\mu.\gamma(y)=\int_{Y_{s(\gamma)}}f(y)d\mu.\gamma(y)=\int_{Y_{r(\gamma)}}f(z\gamma)d\mu(z)\]
On obtient bien une mesure de probabilit\'e et comme $\tilde\pi(\mu.\gamma)=s(\gamma)$, le support de $\mu.\gamma$ est un sous espace de ${Y_{s(\gamma)}}$. Ainsi pour tout \'el\'ement $(\mu,\gamma)$ dans ${M\ast\mathcal{G}}$, la mesure de probabilit\'e $\mu.\gamma$ est dans l'espace ${M}$.
\\
De plus, pour tout $(\gamma,\eta)$ dans ${\mathcal{G}^{(2)}}$ et toute fonction $f$ dans ${C_0(Y)}$, on a 
\begin{align*}
(\mu.\gamma\eta)(f)&=\int_{Y_{s(\eta)}}f(y)d(\mu.\gamma\eta)(y)=\int_{Y_{r(\gamma)}}f(z\gamma\eta)d\mu(z) 
\\
&=\int_{Y_{r(\eta)}}f(z\eta)d\mu.\gamma(z)=\int_{Y_{s(\eta)}}f(y)d(\mu.\gamma).\eta(y) 
\\
&=(\mu.\gamma).\eta(f)
\end{align*}
Enfin, pour tout $\mu\in{M}$ tel que $\tilde\pi(\mu)=x$ et toute fonction $f\in{C_0(Y)}$, on a \[\mu.\tilde\pi(\mu)(f)=\int_{Y_x}f(y)d\mu.\varepsilon(x)(y)=\int_{Y_x}f(y\varepsilon(x))d\mu(y)=\mu(f)\]
Ainsi on a muni l'espace localement compact $M$ d'une action du groupo\"{\i}de $\mathcal{G}$.
\\

L'action est continue si l'application (induite par l'action) $M\ast_{\tilde{\pi},r}\mathcal{G}\rightarrow{M}$ est continue. On consid\`ere une suite $(\mu_i)_{i\in{I}}$ de mesures dans $M$ convergeant vers une mesure $\mu$ dans $M$. Pour tout $i$ dans $I$, on note $x_i:=\tilde{\pi}(\mu_i)$ et $x:=\tilde{\pi}(\mu)$ et on a une suite $(x_i)_{i\in{I}}$ dans $X$ qui converge vers $x$. On consid\`ere une suite $(\gamma_i)_{i\in{I}}$ dans $\mathcal{G}$ telle que pour tout $i$ dans $I$, $\gamma_i$ est dans $\mathcal{G}^{x_i}$ et qui converge vers un \'el\'ement $\gamma$ dans $\mathcal{G}^x$. On montre que la suite $(\mu_i.\gamma_i)_{i\in{I}}$ converge dans $M$ vers $\mu.\gamma$. 
\\
On consid\`ere $U$ un voisinage ouvert de $\gamma$ dans $\mathcal{G}$ tel que les applications but et source sont des hom\'eomorphismes sur leur image respective. On note $\rho_x : r(U)\rightarrow{U}$ l'hom\'eomorphisme inverse de la restriction $r : U\rightarrow{r(U)}$. 
\\
On consid\`ere $V$ un voisinage ouvert de $\mu.\gamma$ dans $M$ que l'on peut supposer de la forme
$$V=\big\{\mu\in{M}\textrm{ : }\lvert{\mu(f_j)-\mu.\gamma(f_j)}\lvert<\varepsilon\textrm{ , }\forall{j}\in\{1,...,m\}\big\}$$ 
o\`u $\varepsilon$ est un r\'eel strictement positif et pour tout $j$ dans $\{1,...,m\}$, $f_j$ est une fonction dans $C_0(Y)$ telle que $Supp(f_{j})$ est inclus dans l'ouvert $\pi^{-1}(s(U))$. Pour tout $j$ dans $\{1,...m\}$, on note $\tilde{f}_j$ la fonction de $C_0(Y)$ obtenue par l'action de $\gamma$ sur $f_j$ telle que le support $Supp\big(\tilde{f}_j\big)$ est inclus dans $\pi^{-1}\big(r(U)\big)$. On note $\tilde{V}$ le voisinage ouvert de $\mu$ dans $M$ d\'efini par 
$$\tilde{V}:=\big\{\mu'\in{M}\textrm{ : }\big\lvert{\mu'\big(\tilde{f}_j\big)-\mu\big(\tilde{f}_j\big)}\big\lvert<\varepsilon\textrm{ , }\forall{j}\in\{1,...,m\}\big\}$$
et la suite $(\mu_i)_{i\in{I}}$ convergeant vers $\mu$, il existe $i_\varepsilon$ dans $I$ tel que pour tout $i\geq{i_\varepsilon}$, l'\'el\'ement $\gamma_i$ est dans $U$ et on a, pour tout $j$ dans $\{1,...,m\}$, 
$$\big\lvert{\mu_i(\tilde{f}_j)-\mu(\tilde{f}_j)}\big\lvert<\varepsilon$$
Or on a pour tout $j$ dans $\{1,...,m\}$
\begin{align*}
\mu\big(\tilde{f}_j\big)&=\int_Y\tilde{f}_j(y)d\mu(y)=\int_{Y_{x}}\tilde{f}_j(y)d\mu(y)=\int_{Y_x}f_j\big(y.\rho_{x}(\pi(y))\big)d\mu(y)
\\
&=\int_{Y_x}f_j(y.\gamma)d\mu(y)=\int_{Y_{s(\gamma)}}f_j(z)d(\mu.\gamma)(z)=\int_Yf_j(z)\d(\mu.\gamma)(z)
\\
&=\big(\mu.\gamma\big)(f_j)
\end{align*}
On obtient de m\^eme pour tout $i\geq{i_\varepsilon}$ dans $I$ et tout $j$ dans $\{1,...,m\}$, l'\'egalit\'e 
$$\mu_i\big(\tilde{f}_j\big)=\mu_i.\gamma_i(f_j)$$
Ainsi pour tout $i\geq{i_\varepsilon}$, on a, pour tout $j$ dans $\{1,...,m\}$
$$\big\lvert{\mu_i.\gamma_i(f_j)-\mu.\gamma(f_j)}\big\lvert=\big\lvert{\mu_i\big(\tilde{f}_j\big)-\mu\big(\tilde{f}_j\big)}\big\lvert<\varepsilon$$
donc $\mu_i.\gamma_i$ est dans l'ouvert $V$ pour tout $i\geq{i_\varepsilon}$. L'action du groupo\"{\i}de $\mathcal{G}$ sur $M$ est continue.
\\
\end{proof}

\subsection  {Existence d'une section continue}\label{Partie5}

Pour d\'emontrer que l'application $\tilde{\pi} : M\to{X}$ continue, surjective, propre et ouverte poss\`ede une section continue (non n\'ecessairement \'equivariante), on utilise le th\'eor\`eme de s\'election de Michael \cite{Michael56} dont on donne un rappel succinct pour la commodit\'e du lecteur. 
\\
 
On consid\`ere $A$ et $B$ deux espaces topologiques (A \'etant suppos\'e s\'epar\'e) et on note $2^B$ l'ensemble des parties non vide de $B$. On consid\`ere $\Omega : A\rightarrow{2^B}$ une application qui sera dite semi-continue inf\'erieurement si, pour tout ouvert $V$ de $B$, l'ensemble $\{a\in{A}\textrm{ : }\Omega(a)\cap{V}\neq\emptyset\}$ est un ouvert de $A$.

\begin{Def}
Soit $A$ et $B$ des espaces topologiques et $\Omega : A\rightarrow{2^B}$ une application. On appelle s\'election pour l'application $\Omega$ une application continue $f : A\rightarrow{B}$ telle que, pour tout $a$ dans $A$, l'\'el\'ement $f(a)$ est dans $\Omega(a)$.
\end{Def}

On consid\`ere \`a partir de maintenant que $B$ est un espace de Banach et on note dans la suite de cette sous section $\mathfrak{F}(B):=\big\{S\subset{B}\textrm{ : }S\textrm{ convexe et ferm\'ee}\big\}$ qui est une sous famille de $2^B$. Michael donne alors une caract\'erisation des applications $\Omega : A\rightarrow\mathfrak{F}(B)$ qui admettent une s\'election par le th\'eor\`eme qui suit :
\begin{Theo}[\cite{Michael56}]{\label{TheoSelecMicha}}
Les assertions suivantes sont \'equivalentes :
\begin{enumerate}
\item[a)] A est un espace paracompact
\item[b)] toute application $\Omega : A\rightarrow\mathfrak{F}(B)$ semi-continue inf\'erieurement admet une s\'election.
\\
\end{enumerate}
\vspace{+1mm}
\end{Theo}
\begin{Rem}
L'int\'er\^et d'introduire l'espace $M$ r\'eside dans le fait que l'application $\tilde{\pi} : M\rightarrow{X}$ est continue, surjective, propre et ouverte, comme le prouve la proposition {\ref{tildepicanonique}}. L'espace des mesures de Radon $R(Y)$ \'etant un espace de Banach, on va, par le th\'eor\`eme de s\'election de Michael \ref{TheoSelecMicha}, en d\'eduire l'existence d'une section continue de $\tilde{\pi}$. En munissant l'espace $M$ d'une structure de $\mathcal{G}$-espace localement compact, on peut ainsi appliquer la propri\'et\'e universelle de l'espace $\beta_X\mathcal{G}$ et d\'efinir une application continue et $\mathcal{G}$-\'equivariante $\tilde{\theta} : \beta_X\mathcal{G}\rightarrow{M}$.
\vspace{+2mm}
\end{Rem}

\begin{Pro}{\label{SectionCont}}
L'application $\tilde{\pi} : M\rightarrow{X}$ continue, surjective, propre et ouverte admet une section continue (non n\'ecessairement $\mathcal{G}$-\'equivarainte).
\end{Pro}
\begin{proof}
Il s'agit de r\'eunir les conditions pour appliquer le th\'eor\`eme de s\'election de Michael ${\ref{TheoSelecMicha}}$. L'espace $M$ est un sous espace de l'espace des mesures de Radon $R(Y)$ qui est un espace de Banach. Pour tout $x$ dans $X$, on note $M_x:=\tilde{\pi}^{-1}\big(\{x\}\big)$ et on a $M=\{M_x\}_{x\in{X}}$ qui est une sous famille de $2^{R(Y)}$. On voit clairement que pour tout $x$ dans $X$, l'espace $M_x$ est ferm\'e et convexe et $M$ est donc une sous famille de $\mathfrak{F}\big(R(Y)\big)$. On pose $\Omega : X\rightarrow\mathfrak{F}\big(R(Y)\big)$ l'application d\'efinie, pour tout $x$ dans $X$, par $\Omega(x)=M_x$. Pour montrer que l'application $\Omega : X\rightarrow\mathfrak{F}\big(R(Y)\big)$ est semi-continue inf\'erieurement, on raisonne par l'absurde : on suppose qu'il existe un ouvert $V$ dans $R(Y)$ pour lequel l'ensemble $\{x\in{X}\textrm{ : }\Omega(x)\cap{V}\neq\emptyset\}$ n'est pas ouvert dans $X$. On peut supposer qu'il existe $a$ un r\'eel strictement positif, $\mu_0$ dans $R(Y)$ et une famille $(f_i)_{i=1}^m$ de fonctions dans $C_c(Y)$ tels que $V$ est de la forme 
$$\big\{\mu\in{R(Y)}\textrm{ : }\lvert{\mu(f_i)-\mu_0(f_i)}\lvert<a\textrm{, }\forall{i}\in{\{1,...,m\}}\big\}$$ 
L'espace $\{x\in{X}\textrm{ : }\Omega(x)\cap{V}\neq\emptyset\}$ \'etant suppos\'e non ouvert dans $X$ (et donc non vide), il existe $x'$ dans $X$ et $\mu'$ dans $M\cap{V}$ avec $\tilde{\pi}(\mu')=x'$, tels que pour tout voisinage ouvert $U'$ de $x'$ dans $X$, $U'$ n'est pas contenu dans $\{x\in{X}\textrm{ : }\Omega(x)\cap{V}\neq\emptyset\}$. 
\\
On consid\`ere $\varepsilon$ un r\'eel dans $]0,a/2[$, v\'erifiant, pour tout $i$ dans $\{1,...,m\}$, l'in\'egalit\'e 
$$\lvert{\mu'(f_i)-\mu_0(f_i)}\lvert<a-2\varepsilon$$
On applique un raisonnement similaire \`a celui de la d\'emonstration de ${\ref{tildepicanonique}}$ et on construit alors
\begin{enumerate}
\item[$\bullet$] une suite g\'en\'eralis\'ee $(x_\lambda)_{\lambda\in\Lambda}$ convergeant dans $X$ vers $x'$ tel que pour tout $\lambda$ dans $\Lambda$, l'intersection de $M_{\mu_\lambda}$ avec $V$ est vide
\item[$\bullet$] un recouvrement ouvert fini $\{U_{z_j}\}_{j=1}^n$ de $Y_{x'}$ pour lequel il existe $\lambda_n$ dans $\Lambda$ tel que, pour tout $\lambda>\lambda_n$, l'\'el\'ement $x_\lambda$ est dans $\cap_{j=1}^n\pi(U_{z_j})$.
\item[$\bullet$] pour $\lambda>\lambda_1$, une mesure $\mu_\lambda$ dans $M_{x_\lambda}$ telle que pour tout $i$ dans $\{1,...,m\}$, on a $\lvert{\mu_\lambda(f_i)-\mu'(f_i)}\lvert\leq2\varepsilon$
\end{enumerate}
On aboutit alors \`a la contradiction $\lvert{\mu'(f_i)-\mu_0(f_i)}\lvert\geq{a-2\varepsilon}$. Donc l'application $\Omega : X\rightarrow\mathfrak{F}\big(R(Y)\big)$ est semi-continue inf\'erieurement. L'espace $X$ \'etant paracompact, l'application $\Omega$ admet une s\'election et il existe donc une section continue de l'application $\tilde{\pi} : M\rightarrow{X}$ que l'on note $\sigma : X\rightarrow{M}$.
\\
\end{proof}

Le $\mathcal{G}$-espace localement compact et s\'epar\'e $M$ satisfaisant les conditions de la propri\'et\'e universelle, on obtient alors la proposition suivante : 
\begin{Pro}{\label{ProSectionCont}}
Il existe une application $\tilde{\theta} : \beta_X\mathcal{G}\rightarrow{M}$ continue, propre et $\mathcal{G}$-\'equivariante.
\end{Pro}
\begin{proof}
D'apr\`es $\ref{EspaceEquivLocComp}$, l'espace $M$ est un $\mathcal{G}$-espace localement compact. D'apr\`es $\ref{tildepicanonique}$ et $\ref{SectionCont}$ l'application $\tilde{\pi} : M\rightarrow{Z}$ est continue, surjective, propre et admet une section continue $\theta : X\rightarrow{M}$. Les conditions de la proposition $\ref{UnivSpace}$ \'etant r\'eunies, il existe une application $\tilde{\theta} : \beta_X\mathcal{G}\rightarrow{M}$ continue, propre et $\mathcal{G}$-\'equivariante.
\\
\end{proof}

\subsection  {Moyennabilit\'e}\label{Partie6}

La moyennabilit\'e \`a l'infini du groupo\"{\i}de $\mathcal{G}$ se caract\'erise par l'existence d'un $\mathcal{G}$-espace localement compact et s\'epar\'e $Y$ tel que $\pi : Y\rightarrow{X}$ une application continue, propre, surjective et ouverte pour lequel l'action est moyennable. On a construit \`a l'aide de mesures de probabilit\'e sur $Y$ le $\mathcal{G}$-espace localement compact $M$ pour lequel l'application $\tilde{\pi} : M\to{X}$ est continue, surjective, propre et admet une section continue (non n\'ecessairement \'equivariante). La proposition $\ref{ProSectionCont}$ donne l'existence d'une application $\tilde{\theta} : \beta_X\mathcal{G}\rightarrow{M}$ continue, propre et $\mathcal{G}$-\'equivariante. 
\\

Dans cette partie, on commence par prouver que la moyennabilit\'e du groupo\"{\i}de $Y\rtimes\mathcal{G}$ implique celle du groupo\"{\i}de $M\rtimes\mathcal{G}$. Une fois la moyennabilit\'e de $M\rtimes\mathcal{G}$ \'etablie, on s'attaque \`a la d\'emonstration du th\'eor\`eme \ref{TheoEquivEspace} : on va \'enoncer deux lemmes \ref{PositivePropre} et \ref{LemmeMoyContPropre} , qui nous permettront de "tir\'e en arri\`ere" la moyennabilit\'e  de $M\rtimes\mathcal{G}$ sur $\beta_X\mathcal{G}\rtimes\mathcal{G}$ en utilisant entre autre l'application $\tilde\theta : \beta_X\mathcal{G}\to{M}$.
\vspace{+2mm}
\\

\begin{Rem}\label{RemEspTopFib}
On consid\`ere $Y'$ un sous espace de $Y$ et $\mathcal{G}'$ un sous espace de $\mathcal{G}$. On note $Y'\ast_{\pi,r}\mathcal{G}'$ le sous espace de $Y\ast_{\pi,r}\mathcal{G}$ d\'efini par $Y'\ast_{\pi,r}\mathcal{G}':=\big\{(y,\gamma)\in{Y'\times\mathcal{G}'}\textrm{ : }\pi(y)=r(\gamma)\big\}$. Si les espaces $Y'$ et $\mathcal{G}'$ sont compacts respectivement dans $Y$ et $\mathcal{G}$, alors l'espace $Y'\ast_{\pi,r}\mathcal{G}'$ est compact dans $Y\ast_{\pi,r}\mathcal{G}$; en effet, l'espace $Y'\times\mathcal{G}'$ est compact dans $Y\times\mathcal{G}$, l'espace $Y\ast_{\pi,r}\mathcal{G}$ est ferm\'e dans $Y\times\mathcal{G}$ et on a $Y'\ast_{\pi,r}\mathcal{G}'=\big(Y'\times\mathcal{G}'\big)\cap\big({Y\ast_{\pi,r}\mathcal{G}}\big)$.
\\
\end{Rem}

On a la proposition suivante : 
\begin{Pro}{\label{ProMoy}}
Si $Y\rtimes\mathcal{G}$ est un groupo\"{\i}de moyennable, alors $M\rtimes\mathcal{G}$ est un groupo\"{\i}de moyennable.
\end{Pro}

\begin{proof}
On suppose que $Y\rtimes\mathcal{G}$ est un groupo\"{\i}de moyennable c'est-\`a-dire qu'il existe une suite $(\phi_i)_{i\in{I}}$ de fonctions continues, de type positif et \`a support compact dans $Y\rtimes{\mathcal{G}}$ telle que 
\begin{itemize}
\item[$\bullet$] $\forall{i}\in{I}\textrm{, }\phi_i^{(0)}\leq{1}$ o\`u $\phi_i^{(0)}$ est la restriction de $\phi_i$ \`a $\big(Y\rtimes\mathcal{G}\big)^{(0)}$
\item[$\bullet$] $\lim_i\phi_i=1$ uniform\'ement sur tout sous espace compact de $Y\rtimes\mathcal{G}$
\\
\end{itemize}

Pour tout $i$ dans $I$, on pose $\psi_i$ la fonction d\'efinie, pour tout $(\mu,\gamma)$ dans $M\rtimes\mathcal{G}$, par 
$$\psi_i(\mu,\gamma):=\int_{Y}\phi_i(y,\gamma)d\mu(y)$$
On obtient ainsi une suite de fonctions $(\psi_i)_{i\in{I}}$ d\'efinie sur $M\rtimes\mathcal{G}$ et on montre successivement que les fonctions $\psi_i$ sont continues, de type positif, \`a support compact et qu'elles v\'erifient un crit\`ere de convergence uniforme sur les compacts ce qui implique la moyennabilit\'e du groupo\"{\i}de $M\rtimes\mathcal{G}$.
\\

On prouve tout d'abord la continuit\'e des fonctions $\psi_i$ : soit $i$ dans ${I}$ et $(\mu_0,\gamma_0)$ un \'el\'ement de ${M\rtimes\mathcal{G}}$. Le groupo\"{\i}de $\mathcal{G}$ \'etant \'etale et localement compact, on consid\`ere $U_0$ un voisinage ouvert de $\gamma_0$ hom\'eomorphe \`a l'ouvert $s(U_0)$ de $X$ et $K_0$ un voisinage compact de $\gamma_0$ inclus dans $U_0$. Il existe une fonction continue $\varphi$ \`a support dans $U_0$ telle que $\varphi\lvert_{K_0}\equiv{1}$. On d\'efinit une fonction $\tilde\phi_i$ sur Y par
$$\tilde\phi_i(y)=\left\{
    \begin{array}{ll}
        \varphi(\gamma)\phi_i(y,\gamma), & \mbox{si } y\in{\pi^{-1}(r(U_0))} \mbox{, } \gamma\in{U_0}\\
        0, & \mbox{si }  y\notin{\pi^{-1}(r(U_0))}    \end{array}
\right.
 $$
qui est continue sur $Y$. On a, pour tout $(\mu,\gamma)$ dans $M\ast_{\tilde{\pi},r}{U_0}$,
\begin{align*}
\big\lvert{\psi_i(\mu,\gamma)-\psi_i(\mu_0,\gamma_0)}\big\lvert&=\bigg\lvert{\int_Y\phi_i(y,\gamma)d\mu(y)-\int_Y\phi_i(y,\gamma_0)d\mu_0(y)}\bigg\lvert \\
&=\bigg\lvert{\int_Y\tilde\phi_i(y)d\mu(y)-\int_Y\tilde\phi_i(y)d\mu_0(y)}\bigg\lvert
\end{align*}
Soit $((\mu_\lambda,\gamma_\lambda))_{\lambda\in\Lambda}$ une suite d'\'el\'ements de $M\rtimes\mathcal{G}$ qui converge vers $(\mu_0,\gamma_0)$. Pour tout r\'eel $\varepsilon$ strictement positif, il existe $V_0$ un voisinage ouvert de $\mu_0$ dans $M$ tel que, pour tout $\mu$ dans ${V_0}$, on a 
$$\bigg\lvert{\int_Y\tilde\phi_i(y)d\mu(y)-\int_Y\tilde\phi_i(y)d\mu_0(y)}\bigg\lvert<\varepsilon$$
Ainsi il existe $\lambda'$ dans $\Lambda$ tel que pour tout $\lambda\geq\lambda'$, l'\'el\'ement $(\mu_\lambda,\gamma_\lambda)$ est dans ${V_0\ast_{\tilde{\pi},r}{U_0}}$ et donc
\begin{align*}
\big\lvert{\psi_i(\mu_\lambda,\gamma_\lambda)-\psi_i(\mu_0,\gamma_0)}\big\lvert&=\bigg\lvert{\int_Y\tilde\phi_i(y)d\mu_\lambda(y)-\int_Y\tilde\phi_i(y)d\mu_0(y)}\bigg\lvert<\varepsilon
\end{align*}
Ainsi pour tout $i$ dans $I$, la fonction $\psi_i$ est continue sur $M\rtimes\mathcal{G}$.
\\

On montre que les fonctions $\psi_i$ sont de type positif : pour tout $i$ dans $I$, la fonction $\psi_i$ est de type positif : soient ${n}$ dans $\NN\textrm{, }\lambda_1,...,\lambda_n$ dans $\CC\textrm{, }\mu$ dans ${M}\textrm{ et }{\gamma_1,...,\gamma_n}$ dans $\mathcal{G}^{\tilde\pi(\mu)}$, on a 

\begin{align*}
A&=\sum_{p,q=1}^n\overline{\lambda_p}\lambda_q\psi_i(\mu\gamma_p,\gamma_p^{-1}\gamma_q)= \sum_{p,q=1}^n\overline{\lambda_p}\lambda_q\int_{Y_{s(\gamma_p)}}\phi_i(y,\gamma_p^{-1}\gamma_q)d\mu.\gamma_p(y)
\\
&=\sum_{p,q=1}^n\int_{Y_{r(\gamma_p)}}\overline{\lambda_p}\lambda_q\phi_i(z.\gamma_p,\gamma_p^{-1}\gamma_q)d\mu(z)
\\
&=\int_{Y_{r(\gamma_p)}}\sum_{p,q=1}^n\overline{\lambda_p}\lambda_q\phi_i(z.\gamma_p,\gamma_p^{-1}\gamma_q)d\mu(z)\geq{0}
\end{align*}
car chaque fonction $\phi_i$ est de type positif donc $\sum_{p,q=1}^n\overline{\lambda_p}\lambda_q\phi_i(z.\gamma_p,\gamma_p^{-1}\gamma_q)\geq0$. Pour tout $i$ de $I$, la fonction $\psi_i$ est de type positif.
\\

On montre maintenant que les fonctions $\psi_i$ sont \`a support compact : soit $i\in{I}$, on pose $K_i:=Supp(\psi_i)$. La fonction $\phi_i$ \'etant \`a support compact, l'espace $Q_i:=Supp(\phi_i)$ est un sous espace compact de $Y\rtimes{\mathcal{G}}$; les projections $p_1 : Y\rtimes\mathcal{G}\rightarrow{Y}$ et $p_2 : Y\rtimes\mathcal{G}\rightarrow\mathcal{G}$ \'etant continues, les espaces $p_1(Q_i)$ et $p_2(Q_i)$ sont compacts respectivement dans Y et $\mathcal{G}$ et $C_i:=\pi(p_1(Q_i))=r(p_2(Q_i))$ est un sous espace compact de $X$; d'apr\`es la remarque $\ref{RemEspTopFib}$, $\tilde\pi^{-1}(C_i)\ast_{\tilde{\pi},r}{p_2(Q_i)}$ est compact dans $M\rtimes\mathcal{G}$. Le support de $\psi_i$ est donc un ferm\'e contenu dans le sous espace compact $\tilde\pi^{-1}(C_i)\rtimes{p_2(Q_i)}$.  Pour tout $i$ dans $I$, la fonction $\psi_i$ est \`a support compact dans $M\rtimes\mathcal{G}$.
\\

On montre le crit\`ere de convergence uniforme : soit $K$ un sous espace compact de $M\rtimes\mathcal{G}$, alors on a pour tout $(\mu,\gamma)$ dans ${K}$
\begin{align*}
\lvert{\psi_i(\mu,\gamma)-1}\lvert&=\bigg\lvert{\int_Y{\phi_i(y,\gamma)}d\mu(y)-1}\bigg\lvert=\bigg\lvert{\int_Y{\big(\phi_i(y,\gamma)-1\big)}d\mu(y)}\bigg\lvert
\\
&\leq\int_Y\big\lvert{{\phi_i(y,\gamma)-1\big\lvert}d\mu(y)}
\end{align*}
On pose $p_1(K)$ l'image de $K$ par la premi\`ere projection sur $M$. L'espace $p_1(K)$ est un sous espace compact de $M$ et $\tilde\pi(p_1(K))$ est un sous espace compact de $X$. Pour $(\mu,\gamma)$ dans ${K}$, il existe $x$ dans ${\tilde\pi(p_1(K))}$ tel que $Supp(\mu)\subset{Y_x}$ et on a 
$$\big\{y\in{Y}\textrm{ : }\exists{(\mu,\gamma)\in{K}}\textrm{, }y\in{Supp(\mu)}\big\}\subset{\pi^{-1}\big(\tilde\pi(p_1(K))\big)}$$
o\`u $\tilde{K}:=\pi^{-1}\big(\tilde\pi(p_1(K))\big)$ est un sous espace compact de $Y$, car l'application $\pi : Y\rightarrow{X}$ est propre. De plus $p_2(K)$ est un sous espace compact de $\mathcal{G}$, donc d'apr\`es la remarque $\ref{RemEspTopFib}$, l'espace $\tilde{K}\ast_{\pi,r}{p_2(K)}$ est un sous espace compact de $Y\rtimes\mathcal{G}$. 
Alors pour tout r\'eel $\varepsilon$ strictement positif, il existe $i_\varepsilon$ dans ${I}$ tel que pour tout $i\geq{i_\varepsilon}$ on a pour tout ${(y,\gamma)}$ dans ${\tilde{K}\ast_{\pi,r}{p_2(K)}}$
$$\lvert{\phi_i(y,\gamma)}-1\lvert<\varepsilon$$
ainsi pour tout $(\mu,\gamma)$ dans ${K}$ et tout $i\geq{i_\varepsilon}$, on a
$$\big\lvert{\psi_i(\mu,\gamma)-1}\big\lvert\leq\int_Y\big\lvert{\phi_i(y,\gamma)-1}\big\lvert{d\mu}(y)<\varepsilon\int_Yd\mu(y)<\varepsilon$$ 
Donc $\lim_i\psi_i=1$ uniform\'ement sur tout compact $K$ de $M\rtimes\mathcal{G}$.
\\

Enfin pour tout $i\in{I}$, on note $\psi_i^{(0)}$ la restriction de $\psi_i$ \`a $(M\rtimes\mathcal{G})^{(0)}$. On a pour tout $x\in{X}$ et tout $\mu\in{M_x}$
\begin{align*} 
\big\lvert{\psi_i^{(0)}}\big\vert&=\bigg\lvert{\int_{Y_x}}\phi_i(y,\varepsilon(x))d\mu(y)\bigg\lvert=\bigg\lvert\int_{Y_x}{\phi_i^{(0)}(y)}d\mu(y)\bigg\lvert
\\
&\leq\int_{Y_x}\lvert{\phi_i^{(0)}(y)}\lvert d\mu(y)\leq\int_{Y_x}1d\mu(y)\leq{1}
\end{align*}

On a montr\'e l'existence d'une suite $(\psi_i)_{i\in{I}}$ de fonctions continues, de type positif et \`a support compact dans $M\rtimes{\mathcal{G}}$ telle que 
\begin{itemize}
\item[$\bullet$] pour tout $i\in{I}$, $\phi_i^{(0)}\leq{1}$ o\`u $\phi_i^{(0)}$ est la restriction de $\phi_i$ \`a $\big(M\rtimes\mathcal{G}\big)^{(0)}$
\item[$\bullet$] $\phi_i$ converge vers un sur tout sous espace compact de $M\rtimes\mathcal{G}$
\end{itemize}
ce qui prouve que le groupo\"{\i}de $M\rtimes\mathcal{G}$ est moyennable.
\\
\end{proof}

\begin{Lemma}{\label{PositivePropre}}
Soient $\mathcal{H}_1$ et $\mathcal{H}_2$ des groupo\"{\i}des topologiques. Soient $f : \mathcal{H}_1\rightarrow\mathcal{H}_2$ un homomorphisme de groupo\"{\i}de continu et $\varphi$ une fonction continue sur $\mathcal{H}_2$ et de type positif. 
\begin{itemize}
\item[$(a)$] la fonction $\varphi\circ{f}$ est continue  sur $\mathcal{H}_1$ et de type positif
\item[$(b)$] si $f$ est propre et $\varphi$ \`a support compact dans $\mathcal{H}_2$, alors $\varphi\circ{f}$ est \`a support compact dans $\mathcal{H}_1$
\end{itemize} 
\end{Lemma}
\begin{proof}
$(a)$ Il est clair que la composition $\varphi\circ{f}$ est continue sur $\mathcal{H}_1$. La fonction $\varphi$ \'etant de type positif, on a pour tout entier $n$, $x$ dans $\mathcal{H}_2^{(0)}$,  $\gamma_1,...,\gamma_n$ dans $\mathcal{H}_2^{x}$ et tout complexe $\lambda_1,...,\lambda_n$
$$\sum_{i,j=1}^n\varphi(\gamma_i^{-1}\gamma_j)\overline{\lambda_i}\lambda_j\geq{0}$$
Ainsi on a pour tout $n$ entier, $x$ dans $\mathcal{H}_1^{(0)}$,  $\gamma_1,...,\gamma_n$ dans $\mathcal{H}_1^{x}$ et tout complexe $\lambda_1,...,\lambda_n$
\begin{align*}
A&=\sum_{i,j=1}^n\varphi\circ{f}(\gamma_i^{-1}\gamma_j)\overline{\lambda_i}\lambda_j=\sum_{i,j=1}^n\varphi\big(f(\gamma_i)^{-1}f(\gamma_j)\big)\overline{\lambda_i}\lambda_j
\\
&=\sum_{i,j=1}^n\varphi(\eta_i^{-1}\eta_j)\overline{\lambda_i}\lambda_j\geq{0}
\end{align*}
donc la fonction $\varphi\circ{f}$ est de type positif.
\\

$(b)$ On suppose que la fonction $\varphi$ est \`a support compact et que $f$ est propre. On note $K_1$ le support de la fonction $\varphi\circ{f}$, qui est ferm\'e de $\mathcal{H}_1$ et $K_2$ le support de $\varphi$, qui est compact dans $\mathcal{H}_2$. Soit $\gamma$ dans $\mathcal{H}_1$ tel que $\lvert{\varphi(f(\gamma))}\lvert>0$, alors $f(\gamma)$ est dans le support de $\varphi$ et $\gamma$ se trouve dans $f^{-1}(K_2)$. Comme $f$ est propre, on a $f^{-1}(K_2)$ est compact et $K_1$ est un sous espace ferm\'e de $f^{-1}(K_2)$, donc lui aussi compact. La fonction $\varphi\circ{f}$ est \`a support compact.
\\
\end{proof}

\begin{Lemma}{\label{LemmeMoyContPropre}}
Soit $f : \mathcal{H}_1\rightarrow\mathcal{H}_2$ un homomorphisme entre groupo\"{\i}des topologiques continu et propre. Si $\mathcal{H}_2$ est moyennable, alors $\mathcal{H}_1$ est moyennable.  
\end{Lemma}
\begin{proof}
Le groupo\"{\i}de $\mathcal{H}_2$ \'etant moyennable il existe une suite $(\varphi_i)_{i\in{I}}$ de fonctions continues sur $\mathcal{H}_2$, de type positif et \`a support compact telle que 
\begin{itemize}
\item[$\bullet$] pour tout $i\in{I}$, $\varphi_i^{(0)}\leq{1}$ o\`u $\phi_i^{(0)}$ est la restriction de $\varphi_i$ \`a $\mathcal{H}_2^{(0)}$
\item[$\bullet$] $\varphi_i$ converge vers un uniform\'ement sur tout sous espace compact de $\mathcal{H}_2$
\end{itemize}
Pour tout $i$ dans $I$, on note $\psi_i=\varphi_i\circ{f}$ la fonction d\'efinie sur $\mathcal{H}_1$. D'apr\`es le lemme $\ref{PositivePropre}$, chaque fonction $\psi_i$ est continue, de type positif et \`a support compact dans $\mathcal{H}_1$. Pour tout $i$ dans $I$ et $x$ dans $\mathcal{H}_1^{(0)}$, on a $\psi_i^{(0)}(x)=\varphi_i^{(0)}(f(x))\leq{1}$. Soit $K$ un sous espace compact de $\mathcal{H}_1$, par continuit\'e de $f$, le sous espace $f(K)$ est compact dans $\mathcal{H}_2$; la suite $(\varphi_i)_{i\in{I}}$ convergeant uniform\'ement vers $1$ sur $f(K)$, pour tout r\'eel $\varepsilon$ strictement positif, il existe $i'$ dans $I$ tel que pour tout $i\geq{i'}$ et tout $k'$ dans $f(K)$, on a $\lvert{\varphi(k')-1}\lvert<\varepsilon$. Ainsi on a pour tout $i\geq{i'}$ et tout $k$ dans $K$
$$\lvert{\psi_i(k)-1}\lvert=\lvert{\varphi_i(f(k))-1}\lvert<\varepsilon$$
Le groupo\"{\i}de $\mathcal{H}_1$ est moyennable.
\end{proof}

On en arrive \`a la d\'emonstration du th\'eor\`eme $\ref{TheoEquivEspace}$

\begin{proof}[D\'emonstration du Th\'eor\`eme $\ref{TheoEquivEspace}$]\label{DemoTheo2.5}

On suppose qu'il existe un $\mathcal{G}$-espace localement compact $Y$ pour lequel l'application $\pi : Y\rightarrow{X}$ est continue, propre, surjective et ouverte et le groupo\"{\i}de $Y\rtimes\mathcal{G}$ est moyennable. On consid\`ere $M$ le sous espace de l'espace des mesures de Radon $R(Y)$ de $Y$ d\'efini par
$$M=\big\{\mu\in{R(Y)}\textrm{ : }\exists{x}\in{X}\textrm{, }Supp(\mu)\subset{Y_x}\textrm{ et }\mu(Y)=1\big\}$$
Suivant la proposition $\ref{EspaceEquivLocComp}$ , l'espace $M$ est localement compact et muni d'une action continue du groupo\"{\i}de $\mathcal{G}$. On note $\tilde{\pi} : M\rightarrow{X}$ l'application canonique et d'apr\`es la proposition $\ref{tildepicanonique}$ , l'application $\tilde{\pi}$ est continue, surjective, propre et ouverte. Le groupo\"{\i}de $Y\rtimes\mathcal{G}$ \'etant moyennable, la proposition $\ref{ProMoy}$ implique que le groupo\"{\i}de $M\rtimes\mathcal{G}$ est moyennable.
D'apr\`es la proposition $\ref{ProSectionCont}$, il existe une application $\tilde{\theta} : \beta_X\mathcal{G}\rightarrow{M}$ continue, propre et $\mathcal{G}$-\'equivariante.
\\
On note $\tilde{\theta}_{\mathcal{G}} : \beta_X\mathcal{G}\rtimes\mathcal{G}\rightarrow{M\rtimes\mathcal{G}}$ l'application d\'efinie, pour tout $(z,\gamma)$ dans $\beta_X\mathcal{G}\rtimes\mathcal{G}$, par 
$$\tilde{\theta}_{\mathcal{G}}(z,\gamma)=\big(\tilde{\theta}(z),\gamma\big)$$
L'application $\tilde{\theta}_{\mathcal{G}}$ est un homomorphisme de groupo\"{\i}des : en effet pour tout $z$ dans $\beta_X\mathcal{G}$ et $(\gamma,\eta)$ dans $\mathcal{G}^{(2)}$ avec $r(\gamma)=\tilde{s}(z)$, on a 
\begin{align*}
\tilde{\theta}_{\mathcal{G}}(z,\gamma)\tilde{\theta}_{\mathcal{G}}(z\gamma,\eta)&=\big(\tilde{\theta}(z),\gamma\big).\big(\tilde{\theta}(z\gamma),\eta\big)
\\
&=\big(\tilde{\theta}(z),\gamma\big).\big(\tilde{\theta}(z)\gamma,\eta\big)
\\
&=\big(\tilde{\theta}(z),\gamma\eta\big)
\\
&=\tilde{\theta}_{\mathcal{G}}(z,\gamma\eta)
\\
&=\tilde{\theta}_{\mathcal{G}}\big((z,\gamma).(z\gamma,\eta)\big)
\end{align*}
L'homomorphisme de groupo\"{\i}de $\tilde{\theta}_{\mathcal{G}}$ est clairement continu et propre, par composition d'applications continues et propres. D'apr\`es le lemme $\ref{LemmeMoyContPropre}$, puisque le groupo\"{\i}de $M\rtimes\mathcal{G}$ est moyennable et l'homomorphisme $\tilde{\theta}_{\mathcal{G}} : \beta_X\mathcal{G}\rtimes\mathcal{G}\rightarrow{M\rtimes\mathcal{G}}$ continu et propre, alors le groupo\"{\i}de $\beta_X\mathcal{G}\rtimes\mathcal{G}$ est moyennable.

\end{proof}

\subsection {Groupo\"{\i}de \'etale exact}\label{Partie7}
On dit qu'un groupo\"{\i}de $\mathcal{G}$ est exact si pour toute suite exacte $\mathcal{G}$-\'equivariante, 
$$0\longrightarrow{I}\longrightarrow{A}\longrightarrow{A/I}\longrightarrow{0}$$
de $\mathcal{G}$-alg\`ebres, la suite
$$0\longrightarrow{I\rtimes_r\mathcal{G}}\longrightarrow{A\rtimes_r\mathcal{G}}\longrightarrow{A/I\rtimes_r\mathcal{G}}\longrightarrow{0}$$ 
est exacte.
\\

On rappelle qu'un groupo\"{\i}de \'etale localement compact $\mathcal{G}\rightrightarrows{X}$ ($X$ est l'espace des unit\'es du groupo\"{\i}de $\mathcal{G}$) est moyennable \`a l'infini s'il existe un $\mathcal{G}$-espace localement compact $Y$, pour lequel l'application $\pi : Y\rightarrow{X}$ est continue, surjective, ouverte et propre, et tel que le groupo\"{\i}de $Y\rtimes\mathcal{G}$ est moyennable.

\begin{Theo}{\label{TheoPrinc3}}
\itshape{Soit $\mathcal{G}$ un groupo\"{\i}de \'etale d'espace des unit\'es $X$. On consid\`ere les assertions suivantes :
\begin{enumerate}
\item[$(1)$] le groupo\"{\i}de $\mathcal{G}$ est moyennable \`a l'infini  
\item[$(2)$] le groupo\"{\i}de $\mathcal{G}$ agit moyennablement sur l'espace $\beta_X\mathcal{G}$
\item[$(3)$] le groupo\"{\i}de $\mathcal{G}$ est exact 
\item[$(4)$] la $C^\ast$-alg\`ebre $C^\ast_r(\mathcal{G})$ est exacte  
\end{enumerate} 
Alors on a $(1)\Rightarrow{(2)}\Rightarrow{(3)}\Rightarrow{(4)}$}
\\
\end{Theo}
\begin{proof}
On d\'emontre successivement chacune des implications dans l'ordre de l'\'enonc\'e : 
\\
(1)$\Rightarrow$(2) : il s'agit de la d\'emonstration \ref{DemoTheo2.5} du th\'eor\`eme $\ref{TheoEquivEspace}$. 
\\
(2)$\Rightarrow$(3) : on consid\`ere une suite exacte $0\rightarrow{I}\rightarrow{A}\rightarrow{A/I}\rightarrow{0}$ de $\mathcal{G}$-alg\`ebres. La suite 
$$0\longrightarrow{\pi^\ast{I}}\longrightarrow\pi^\ast{A}\longrightarrow\pi^\ast{A/I}\longrightarrow{0}$$
est alors une suite exacte de $C_0(Y)$-alg\`ebres munies d'une action du groupo\"{\i}de $Y\rtimes{\mathcal{G}}$. Le groupo\"{\i}de $Y\rtimes{\mathcal{G}}$ \'etant moyennable, le foncteur $B\rightarrow{B\rtimes_r{(Y\rtimes{\mathcal{G}}})}$ est exacte dans la cat\'egorie des $Y\rtimes{\mathcal{G}}$-alg\`ebres. On a la suite exacte suivante :
$$0 \rightarrow \pi^\ast{I}\rtimes_r(Y\rtimes\mathcal{G}) \rightarrow \pi^\ast{A}\rtimes_r(Y\rtimes{\mathcal{G}}) \rightarrow \pi^\ast{A/I}\rtimes_r(Y\rtimes{\mathcal{G}}) \rightarrow 0 $$
L'application $\pi : Y\rightarrow\mathcal{G}^{(0)}$ \'etant continue, surjective et propre, il existe un homomorphisme injectif de $C^\ast$- alg\`ebres $\phi_I : I\rtimes_r\mathcal{G}\hookrightarrow\pi^\ast{I}\rtimes_r{(Y\rtimes\mathcal{G})}$. On d\'efinit de m\^eme des homomorphismes injectifs $\phi_A : A\rtimes_r\mathcal{G}\hookrightarrow\pi^\ast{A}\rtimes_r{(Y\rtimes\mathcal{G})}$ et $\phi_{A/I} : A/I\rtimes_r\mathcal{G}\hookrightarrow\pi^\ast{A/I}\rtimes_r{(Y\rtimes\mathcal{G})}$.
\begin{eqnarray*}
\xymatrix{
0 \ar[r] & I\rtimes_r\mathcal{G} \ar[r]^{i_r} \ar@{^{(}->}[d]^{\phi_I} & A\rtimes_r\mathcal{G} \ar[r]^{q_r} \ar@{^{(}->}[d]^{\phi_A} & A/I\rtimes_r\mathcal{G} \ar[r] \ar@{^{(}->}[d]^{\phi_{A/I}} & 0   
\\ 
0 \ar[r] & \pi^\ast{I}\rtimes_r(Y\rtimes\mathcal{G}) \ar[r]  & \pi^\ast{A}\rtimes_r(Y\rtimes\mathcal{G}) \ar[r]  & \pi^\ast{A/I}\rtimes_r(Y\rtimes\mathcal{G}) \ar[r]  & 0 
}
\end{eqnarray*}
Pour montrer que la premi\`ere ligne est exacte, on consid\`ere $a$ dans $A\rtimes_r\mathcal{G}$ tel que $q_r(a)=0$ et $\phi_A(a)$ est dans $\pi^\ast{A}\rtimes_r(Y\mathcal{G})$. Soit $(u_i)_{i\in{I}}$ une approximation de l'unit\'e dans l'id\'eal $I$, alors $(\phi_I(u_i))_{i\in{I}}$ est une approximation de l'unit\'e dans l'alg\`ebre des multiplicateurs de $\pi^\ast{I}\rtimes_r(Y\rtimes\mathcal{G})$ et on a 
$$\phi_A(a)=\lim_i\phi_I(u_i)\phi_A(a)=\lim_i\phi_A(u_i.a)\in\phi_I(I\rtimes_r\mathcal{G})$$  
Comme $\phi_A$ est injective, l'\'el\'ement $a$ est dans $I\rtimes_r\mathcal{G}$ et la suite 
$$0 \longrightarrow I\rtimes_r\mathcal{G} \longrightarrow A\rtimes_r\mathcal{G} \longrightarrow A/I\rtimes_r\mathcal{G} \longrightarrow 0 $$
est exacte et le foncteur $B\rightarrow{B\rtimes_r\mathcal{G}}$ est exact sur la cat\'egorie des $\mathcal{G}$-alg\`ebres.
\\
(3)$\Rightarrow$(4) : on consid\`ere la suite exacte de $C^\ast$-alg\`ebres suivantes
$$0\longrightarrow{I}\longrightarrow{A}\longrightarrow{A/I}\longrightarrow{0}$$
La $C^\ast$-alg\`ebre $C_0(\mathcal{G}^{(0)})$ \'etant nucl\'eaire, on a alors
$$0\longrightarrow{I\otimes{C_0(\mathcal{G}^{(0)})}}\longrightarrow{A\otimes{C_0(\mathcal{G}^{(0)})}}\longrightarrow{A/I\otimes{C_0(\mathcal{G}^{(0)})}}\longrightarrow{0}$$
est une suite de $C_0(\mathcal{G}^{(0)})$-alg\`ebres telle que la collection d'applications 
$$id_\gamma : C_0(\mathcal{G}^{(0)},I)(s(\gamma))\rightarrow{C_0(\mathcal{G}^{(0)},I)(r(\gamma))}$$
qui \`a tout $i$ dans $C_0(\mathcal{G}^{(0)},I)(s(\gamma))=I$ associe $id_\gamma(i)=i$ dans $I=C_0(\mathcal{G}^{(0)},I)(s(\gamma))$ d\'efinit une action continue de $\mathcal{G}$ sur $I\otimes{C_0(\mathcal{G}^{(0)})}$. On obtient une suite exacte de $\mathcal{G}$-alg\`ebres et le foncteur $\rtimes_r\mathcal{G}$ \'etant exact, par hypoth\`ese, on a 
$$0\longrightarrow{I\otimes{C_0(\mathcal{G}^{(0)})}\rtimes_r\mathcal{G}}\longrightarrow{A\otimes{C_0(\mathcal{G}^{(0)})}\rtimes_r\mathcal{G}}\longrightarrow{A/I\otimes{C_0(\mathcal{G}^{(0)})}\rtimes_r\mathcal{G}}\longrightarrow{0}$$
correspondant \`a la suite 
$$0\longrightarrow{I\otimes{C_r^\ast(\mathcal{G})}}\longrightarrow{A\otimes{C_r^\ast(\mathcal{G})}}\longrightarrow{A/I\otimes{C_r^\ast(\mathcal{G})}}\longrightarrow{0}$$
est exacte.
\\
\end{proof}


\section{{$C^\ast$-alg\`ebre $C_r^\ast({\mathcal{G}})$ et approximation de la repr\'esentation r\'eguli\`ere}}

Il est l\'egitime de se demander si l'exactitude de $C^*$-alg\`ebre r\'eduite $C_r^*(\mathcal{G})$ implique \`a son tour la moyennabilit\'e \`a l'infini du groupo\"{\i}de $\mathcal{G}$ ou encore la moyennabilit\'e du groupo\"{\i}de $\beta_X\mathcal{G}\rtimes\mathcal{G}$. 
\\

Dans  \cite{KirchbergWassermann99}, Kirchberg et Wassermann donne l'\'equivalence entre l'exactitude d'un groupe discret et l'exactitude de sa $C^*$-alg\`ebre r\'eduite de ce groupe. Dans \cite{Ozawa00}
, Ozawa prouve l'\'equivalence entre l'exactitude de la $C^*$-alg\`ebre r\'eduite et la moyennabilit\'e de l'action par translation \`a gauche du groupe sur son compactifi\'e de Stone-Cech.
\\

Anantharaman-Delaroche donne une r\'eponse affirmative \`a cette question dans \cite{Anantharaman2} dans le cadre plus large des groupes localement compacts et s\'epar\'es en utilisant une condition suffisante sur le groupe localement compact et s\'epar\'e qu'elle appelle {\itshape{propri\'et\'e $(W)$}} et qui est une condition plus faible que la moyennabilit\'e int\'erieure : 
\begin{Def}
On dit qu'un groupe localement compact et s\'epar\'e poss\`ede la {\itshape{propri\'et\'e $(W)$}} si pour tout compact  $K$ de $G$ et tout r\'eel $\varepsilon$ strictement positif, il existe une fonction $f$ continue, born\'ee, de type positif et \`a support proprement support\'e sur $G\times{G}$ telle que  $\big\lvert{f(g,g)-1}\big\lvert\leq\varepsilon$ pour tout \'el\'ement $g$ dans $K$.
\end{Def}
et obtient le r\'esultat suivant 
\begin{Theo}[\cite{Anantharaman2},Th\'eor\`eme 7.3]
Soit $G$ un groupe localement compact et s\'epar\'e. Si le groupe $G$ poss\`ede la {\itshape{propri\'et\'e $(W)$}} et si la $C^*$-alg\`ebre $C^*_r(G)$ est exacte, alors le groupe $G$ est moyennable \`a l'infini.
\end{Theo}
Il est facile de prouver que les groupes discrets poss\`edent la  {\itshape{propri\'et\'e $(W)$}} en consid\'erant la fonction caract\'eristique sur la diagonale $\Delta_G$ de l'espace $G\times{G}$ ce qui permet de retrouver le r\'esultat  de Ozawa \cite{Ozawa00} : 
$$C^*_r(G)\textrm{ exact}\iff{ G}\textrm{ moyennable \`a l'infini}\iff{\beta{G}\rtimes{G}\textrm{ moyennable}}$$ 

Dans cette section, on consid\`ere un groupo\"{\i}de \'etale localement compact et $\sigma$-compact s\'epar\'e $\mathcal{G}\rightrightarrows{X}$ tel que la $C^*$-alg\`ebre r\'eduite $C^*_r(\mathcal{G})$ est exacte et on note $\nu$ le syst\`eme de Haar pour $\mathcal{G}$ constitu\'e des mesures de comptage des $s$-fibres de $\mathcal{G}$ et $\mathcal{L}\big(L^2(\mathcal{G},\nu)\big)$ la $C_0(X)$-alg\`ebre des op\'erateurs $C_0(X)$-lin\'eaires qui admettent un adjoint, sur le $C_0(X)$-module de Hilbert $L^2(\mathcal{G},\nu)$. L'id\'ee est de travailler sur des $C_0(X)$-alg\`ebres continues et exactes afin d'utiliser un r\'esultat de \cite{Blanchard97} de plongement $C_0(X)$-lin\'eaire de telles  $C_0(X)$-alg\`ebres dans la $C_0(X)$-alg\`ebre nucl\'eaire $C_0(X)\otimes\mathcal{O}_2$, o\`u $\mathcal{O}_2$ est l'alg\`ebre de Cuntz. Comme la $C^*$-alg\`ebre $C_r^*(\mathcal{G})$ n'est pas munie d'une structure de $C_0(X)$-alg\`ebre, on construit une $C_0(X)$-alg\`ebre, not\'ee $\mathcal{A}$, qui est la plus petite sous-$C^*$-alg\`ebre de $\mathcal{L}\big(L^2(\mathcal{G},\nu)\big)$ contenant l'image de $C^*_r(\mathcal{G})$ dans $\mathcal{L}\big(L^2(\mathcal{G},\nu)\big)$ par la repr\'esentation r\'eguli\`ere $\lambda$ et l'image de $C_0(X)$ par la structure de $C_0(X)$-alg\`ebre de $\mathcal{L}\big(L^2(\mathcal{G},\nu)\big)$. On montre alors que l'exactitude de la $C^*$-alg\`ebre $C^*_r(\mathcal{G})$ implique l'exactitude de la $C_0(X)$-alg\`ebre $\mathcal{A}$. Afin d'utiliser le th\'eor\`eme de \cite{Blanchard97}, on suppose que la $C_0(X)$-alg\`ebre $\mathcal{A}$ est continue et on se restreint au cas o\`u l'espace des unit\'es $X$ du groupo\"{\i}de est compact. On travaille donc sur une classe plus faible de groupo\"{\i}des \'etales localement compacts et $\sigma$-compacts qui satisfont ces deux conditions. D\`es lors que l'on a utilis\'e le th\'eor\`eme de plongement $C(X)$-lin\'eaire des $C(X)$-alg\`ebres continues et exactes dans la $C(X)$-alg\`ebre nucl\'eaire $C(X)\otimes\mathcal{O}_2$, on va pouvoir utiliser les propri\'et\'es de $C(X)$-nucl\'earit\'e pour obtenir une approximation par des factorisations \`a travers les $C(X)$-alg\`ebres $C(X)\otimes{}M_n(\CC)$ de l'application identit\'e sur $\mathcal{A}$ et donc de la repr\'esentation r\'eguli\`ere de $C^*_r(\mathcal{G})$. C'est par le biais de cette approximation par des factorisations de la repr\'esentation r\'eguli\`ere de la $C^*$-alg\`ebre exacte $C^*_r(\mathcal{G})$ que nous pouvons montrer la moyennabilit\'e de l'action de $\mathcal{G}$ sur $\beta_X\mathcal{G}$. On obtient alors l'implication $(4)\Rightarrow{(2)}$ du th\'eor\`eme \ref{TheoPrinc3} pour la classe des groupo\"{\i}des \'etales localement compacts et $\sigma$-compacts ayant la propri\'et\'e $(W)$, dont l'espace des unit\'es est compact et tels que la $C(X)$-alg\`ebre induite $\mathcal{A}$ est continue.

\subsection  {Approximation de l'unit\'e et op\'erateurs de rangs finis}
Dans cette partie, on utilise les r\'esultats de l'article d'Arveson \cite{Arveson77} sur les $C^*$-alg\`ebres que l'on adapte au cas des $C_0(X)$-alg\`ebres, dont les \'el\'ements sont consid\'er\'es comme des op\'erateurs sur un $C_0(X)$-module de Hilbert. 
\\

Pour $H$ un espace de Hilbert s\'eparable, on rappelle qu'un op\'erateur quasidiagonal $T$ est un op\'erateur born\'e sur l'espace de Hilbert s\'eparable $H$ pour lequel il existe une suite $(F_n)_{n\in\NN}$ de projections de rang fini, qui converge fortement vers l'identit\'e et v\'erifiant 
$$\lim_{n\rightarrow{\infty}}\lVert{F_nT-TF_n}\lVert=0$$
Cette d\'efinition \'equivaut \`a l'existence d'une suite $(E_n)_{n\in\NN}$ de projections de rang fini deux \`a deux orthogonales telle que $\sum_{n\in\NN}E_n=1$ et $T$ est une perturbation compacte de l'op\'erateur $\sum_{n\in\NN}E_nTE_n$.
\\
Dans l'article \cite{Arveson77}, l'auteur construit, dans le cadre plus g\'en\'eral des op\'erateurs born\'es $T : H\rightarrow{H}$ (non n\'ecessairement quasidiagonaux), une suite $(F_n)_{n\in\NN}$ d'op\'erateurs positifs de rang fini, qui converge fortement vers l'op\'erateur identit\'e et v\'erifiant la relation, pour tout $T$ dans $L(H)$,
$$\lim_{n\in\NN}\lVert{F_nT-TF_n}\lVert=0$$
Comme dans le cas des op\'erateurs quasidiagonaux, l'existence d'une telle suite $(F_n)_{n\in\NN}$ permet de construire  une suite $(E_n)_{n\in\NN}$ d'op\'erateurs positifs de rang fini dans $L(H)$ tel que $\sum_{n\in\NN}E_n^2=1$ et pour laquelle, tout op\'erateur $T$ est une perturbation compacte de $\sum_{n\in\NN}E_nTE_n$.
\vspace{+2mm}
\\

Le principal objectif de cette  partie consiste \`a montrer que pour toute $C_0(X)$-alg\`ebre $\mathcal{A}$, consid\'er\'ee comme sous $C_0(X)$-alg\`ebre de la $C_0(X)$-alg\`ebre des op\'erateurs d'un $C_0(X)$-module de Hilbert $\mathcal{H}$,  et tout id\'eal $\mathcal{K}$ de $\mathcal{A}$ d'espace nul trivial, il existe une suite $(E_n)_{n\in{\NN}}$ d'\'el\'ements positifs dans $\overline{\mathcal{K}}$ v\'erifiant 
\begin{itemize}
\vspace{+1mm}
\item[(a)] $E_n\in\overline{\mathcal{K}}\quad\textrm{et}\quad\sum_nE_n^2=Id_{\mathcal{H}}$
\vspace{+1mm}
\item[(b)] $A-\sum_nE_nAE_n$ est dans $\mathcal{K}$, pour tout $A$ dans $\mathcal{A}$
\vspace{+1mm}
\end{itemize}
On montre aussi que pour toute partie finie $\mathcal{F}$ de $\mathcal{A}$, il est possible de choisir les \'el\'ements $E_n$ de sorte que, pour tout \'el\'ement $A$ de $\mathcal{F}$, la norme $\lVert{A-\sum_nE_nAE_n}\lVert$ soit aussi petite qu'on le souhaite. 
\vspace{+3mm}
\\

On consid\`ere $\mathcal{A}$ une $C_0(X)$-alg\`ebre et $\mathcal{K}$ un id\'eal bilat\`ere, non n\'ecessairement ferm\'e, dans $\mathcal{A}$. 
\begin{Def}
Une approximation de l'unit\'e pour $\mathcal{K}$ est une suite croissante $(e_\lambda)_{\lambda\in\Lambda}$ d'\'el\'ements positifs de la boule unit\'e de l'id\'eal $\mathcal{K}$ tels que $\lim_\lambda\lVert{e_{\lambda}k-k}\lVert=0$ pour tout $k$ dans $\mathcal{K}$.
De plus, si $(e_\lambda)_{\lambda\in\Lambda}$ satisfait la relation $\lim_{\lambda}\lVert{e_\lambda{a}-ae_\lambda}\lVert=0$, pour tout $a$ dans $\mathcal{A}$, alors la suite est quasicentrale.
\end{Def}
Dans le cadre des $C_0(X)$-alg\`ebres, on obtient un th\'eor\`eme sur l'existence d'approximation de l'unit\'e quasicentrale, analogue au cas des $C^*$-alg\`ebres.
\begin{Theo}{\label{TheoArv}}
Tout id\'eal $\mathcal{K}$ dans une $C_0(X)$-alg\`ebre $\mathcal{A}$ poss\`ede une approximation de l'unit\'e quasicentrale. Si de plus $\mathcal{A}$ est s\'eparable, cette unit\'e approch\'ee quasicentrale peut \^etre index\'ee sur $\NN$ et on a $u_n\leq{u_{n+1}}$, pour tout $n$ dans $\NN$.
\end{Theo}
\begin{proof}
L'id\'eal $\mathcal{K}$ de la $C(X)$-alg\`ebre $\mathcal{A}$ est aussi un id\'eal pour la structure de $C^*$-alg\`ebre sous jacente de $\mathcal{A}$; on applique alors la d\'emonstration de \cite{Arveson77}.
\end{proof}

\begin{Rem}
La preuve montre que l'on peut constuire, \`a partir d'une unit\'e approch\'ee $(e_\lambda)_{\lambda\in\Lambda}$ quelconque pour $\mathcal{K}$, une approximation de l'unit\'e quasicentrale, dont les \'el\'ements sont des combinaisons lin\'eaires convexes des \'el\'ements de $(e_\lambda)_{\lambda\in\Lambda}$.
\vspace{+4mm}
\end{Rem}



Soit $B$ une $C^*$-alg\`ebre et $\mathcal{H}$ un module hilbertien sur $B$.

\begin{Def}
Soit $\mathcal{H}$ un $B$-module hilbertien s\'eparable et $T$ un op\'erateur $B$-lin\'eaire et poss\'edant un adjoint sur $\mathcal{H}$; on dit que $T$ est un op\'erateur positif si pour tout \'el\'ement $\xi$ dans $\mathcal{H}$ on a $\langle{T\xi,\xi}\rangle$ est positif dans $B$.
\end{Def}

\begin{Notation} On adoptera les notations suivantes en ce qui concerne les op\'erateurs de rang fini : 
\begin{itemize}
\item[$\bullet$] On appelle op\'erateur de rang $1$, tout endomorphisme sur le $B$-module hilbertien $\mathcal{H}$ qui s'\'ecrit sous la forme $\theta_{x,y}$, avec $x$ et $y$ dans $\mathcal{H}$, o\`u pour tout $\xi$ dans $\mathcal{H}$, on a $\theta_{x,y}(\xi)=x\langle{{y,\xi}}\rangle$.
\item[$\bullet$] On appelle op\'erateur de rang fini sur le $B$-module hilbertien $\mathcal{H}$ est un endomorphisme de $B$-module $T : \mathcal{H}\rightarrow\mathcal{H}$, qui est une combinaison lin\'eaire d'op\'erateurs de rang un.
\item[$\bullet$] L'ensemble des op\'erateurs de rang fini sur un $B$-module hilbertien $\mathcal{H}$ est un id\'eal dans $\mathcal{L}(\mathcal{H})$, not\'e $\mathcal{K}_0$, dont la fermeture, not\'ee $\mathcal{K}$, est l'ensemble des op\'erateurs compacts de $\mathcal{L}(\mathcal{H})$.
\end{itemize} 
\vspace{+3mm}
\end{Notation}
En utilisant le th\'eor\`eme \ref{TheoArv} , on obtient le corollaire suivant :
\begin{Cor}
Soit $\mathcal{H}$ un $C(X)$-module de Hilbert s\'eparable.
\begin{enumerate}
\item[(i)] Il existe une suite croissante $(F_\lambda)_{\lambda\in\Lambda}$ d'op\'erateurs positifs de rang fini qui converge fortement vers l'op\'erateur identit\'e de $\mathcal{H}$ et v\'erifiant la relation $\lim_{\lambda}\lVert{F_{\lambda}T-TF_\lambda}\lVert=0$, pour tout op\'erateur $C(X)$-lin\'eaire $T$ dans $\mathcal{L}\big(\mathcal{H}\big)$.
\item[(ii)] Pour $\mathcal{A}$ une $C(X)$-alg\`ebre s\'eparable d'op\'erateurs sur un champ continu d'espaces de Hilbert s\'eparables $\mathcal{H}$, il existe une suite $(F_n)_{n\in\NN}$ croissante d'op\'erateurs positifs et de rang fini qui converge fortement vers l'identit\'e et v\'erifiant la relation $\lim_n\lVert{F_nA-AF_n}\lVert=0$, pour tout $A$ dans $\mathcal{A}$.
\end{enumerate}
\end{Cor}

\begin{proof}
Quitte \`a remplacer $\mathcal{A}$ par la $C(X)$-alg\`ebre des perturbations compactes d'op\'erateurs de $\mathcal{A}$, on consid\`ere l'id\'eal $\mathcal{K}$ des op\'erateurs de rang fini dans la $C(X)$-alg\`ebre s\'eparable $\mathcal{A}$ et on applique le th\'eor\`eme $\ref{TheoArv}$ pour obtenir la suite voulue.
\\
\end{proof}
Dans le th\'eor\`eme qui suit, on consid\`ere la $C(X)$-alg\`ebre $\mathcal{A}$ comme une sous $C(X)$-alg\`ebre d'op\'erateurs de l'ensemble des op\'erateurs born\'es, $C(X)$-lin\'eaires et poss\'edant un adjoint sur un $C(X)$-module hilbertien $\mathcal{H}_{C(X)}$. Soit $\mathcal{I}$ un id\'eal de $\mathcal{A}$, on appelle espace nul de $\mathcal{I}$, l'espace d\'efini par $\big\{\xi\in{\mathcal{H}_{C(X)}}\textrm{ : }T(\xi)=0\textrm{, }\forall{T}\in\mathcal{I}\big\}$; on dira que l'espace nul de l'id\'eal $\mathcal{I}$ est trivial, s'il est r\'eduit au vecteur nul.
\vspace{+1mm}
\begin{Theo}{\label{TheoCool}}
Soit $\mathcal{K}$ un id\'eal dans $\mathcal{A}$ dont l'espace nul est trivial. Il existe une suite $(E_n)_{n\in\mathbb{N}}$ d'op\'erateurs positifs dans la fermeture $\overline{\mathcal{K}}$ (pour la topologie de la norme de $\mathcal{K}$) telle que pour tout $n$ dans $\NN$, on a $E_n^2$ est dans l'id\'eal ${\mathcal{K}}\textrm{, }\sum_{n\geq{1}}{E_n^2}=1$ et $A-\sum{E_nAE_n}$ est dans l'id\'eal ${\mathcal{K}}\textrm{, }$pour tout ${A}$ dans $\mathcal{A}$.
\\
De plus, en fixant $\varepsilon\geq{0}$ et $\mathcal{F}\subset\mathcal{A}$ un sous ensemble fini, on peut choisir la suite $(E_n)_{n\in\mathbb{N}}$ de sorte que $\|A-\sum{E_nAE_n}\lVert<{\varepsilon}$, pour tout $A$ dans $\mathcal{F}$
\end{Theo}
\vspace{+2mm}
\begin{Rem}{\label{RemCool}}
Avant de s'attaquer \`a la d\'emonstration du th\'eor\`eme, on peut faire les remarques suivantes :
\vspace{+1mm}

$\mathbf{a)}$ Si une telle suite $(E_n)_{n\in\mathbb{N}}$ existe, alors $\sum{E_nAE_n}$ est bien d\'efinie par l'hypoth\`ese $\sum{E_n^2}=1$ qui implique la convergence forte de la somme.
\\
En effet, on consid\`ere le $C(X)$-module hilbertien $\mathcal{H}':=\bigoplus_{n\in\NN}\mathcal{H}_{C(X)}$ et on note $i'$ la repr\'esentation de $\mathcal{A}$ dans l'ensemble des op\'erateurs $\mathcal{L}\big(\mathcal{H}'\big)$ d\'efinie par
 $i'(A)={\oplus_{n\in\NN}{A}}$.
On peut alors d\'efinir une isom\'etrie, en posant
\begin{align*}
V :  \textrm{ } \mathcal{H}_{C(X)} & \longrightarrow{\oplus_{n\in\NN}\mathcal{H}_{C(X)}}=\mathcal{H}'\\
  \xi & \longrightarrow{\oplus_{n\in\NN}{E_n(\xi)}}
\end{align*}
On note $P_n : \mathcal{H'}\rightarrow\mathcal{H'}$ la projection sur les $n$ premi\`eres composantes de $\mathcal{H'}=\oplus_{n\in\NN}\mathcal{H}_{C(X)}$ et $V_n:=P_nV$; alors pour tout ${A}$ dans $\mathcal{A}$, on a la convergence dans la topologie forte
$$
\xymatrix{
\sum_{k=1}^n{E_{k}AE_{k}}=V^*P_{n}i'(A)P_{n}V \ar[rr] & & {V^*i'(A)V=\sum_{k=1}^{+\infty}}E_kAE_k
}
$$

$\mathbf{b)}$ L'application $l(A)=\sum E_nAE_n$ pouvant s'\'ecrire sous la forme $l(A)=V^*i'(A)V$, elle est compl\`etement positive et son image est dans la fermeture forte $\overline{\mathcal{A}}^{forte}$ de  $\mathcal{A}$. Puisque l'\'el\'ement $A-l(A)$ est dans ${{\mathcal{K}}}$ qui est un id\'eal de ${\mathcal{A}}$, alors l'application $l$ est \`a image dans $\mathcal{A}$.
\vspace{+2mm}
\end{Rem}

Pour d\'emontrer  le th\'eor\`eme \ref{TheoCool}, on utilise le lemme suivant d\'emontr\'e dans \cite{Arveson77} : 
\begin{Lemma}{\label{LemCool}}
Soit $\varepsilon$ un r\'eel strictement positif et $f$ une fonction continue sur l'intervalle $[0,1]$ \`a valeurs complexes et qui soit nulle en $0$.
Alors il existe un r\'eel $\delta$ strictement positif tel que pour tout \'el\'ement $a$ et tout \'el\'ement positif $e$ dans la boule unit\'e de $\mathcal{A}$, on a
$\|ae-ea\|\leqslant{\delta}$ implique ${\|af(e)-f(e)a\|\leqslant{\varepsilon}}$.
\end{Lemma}

\begin{proof}[D\'emonstration du Th\'eor\`eme $\ref{TheoCool}$]
On va utiliser l'existence de l'approximation de l'unit\'e dans $\mathcal{K}$ pour construire la suite $(E_n)_{n\in\mathbb{N}}$
\\

a) Soit $\varepsilon$ strictement positif, alors, d'apr\`es le lemme $\ref{LemCool}$, il existe $(\delta_n)_{n\in\mathbb{N}}$ suite d\'ecroissante de r\'eels positifs tendant vers $0$ telle que pour tout $a$ et $E$ dans la boule unit\'e de $\mathcal{A}$, avec $E$ positif, on a
$$\|Ea-aE\|\leq{\delta_n}\Rightarrow\|E^{1/2}a-aE^{1/2}\|\leq{\varepsilon/2^{n+1}}$$
\\
On consid\`ere une suite croissante $(\mathcal{F}_n)_{n\in\NN}$ de sous ensembles finis de $\mathcal{A}$ telle que la r\'eunion $\bigcup\mathcal{F}_n$ est dense dans la boule unit\'e de $\mathcal{A}$. (On supposera de plus que $\mathcal{F}_1$ contient un nombre fini d'\'el\'ements de norme 1)
\\
On a vu qu'il existe $(u_n)_{n\in\mathbb{N}}$ une approximation de l'unit\'e quasicentrale dans $\mathcal{K}$ et (quitte \`a consid\'erer une sous suite), on peut supposer que pour tout ${A}$ dans $\mathcal{F}_{n+1}$,
$$\|u_nA-Au_n\|\leq{\delta_{n}/2}$$
On a, en posant $u_0=0$
$$\left\{
    \begin{array}{ll}
        \forall{A\in\mathcal{F}_1}\textrm{,}\qquad\|u_1A-Au_1\|<\delta_1
        \\
        \forall{A}\in\mathcal{F}_n\textrm{,}\qquad\|(u_n-u_{n-1})A-A(u_n-u_{n-1})\|<\delta_{n-1}
    \end{array}
\right.
$$

b/ Par calcul fonctionnel continu pour la fonction $f(x)=x^{1/2}$, on pose, pour tout $n$ dans $\NN$, $E_n=(u_n-u_{n-1})^{1/2}$. Alors pour $\varepsilon\geq{0}$ et une suite d\'ecroissante de r\'eels postifs tendant vers z\'ero $(\delta_n)_{n\in\mathbb{N}}$ appropri\'ee, on a pour tout ${A}$ dans $\mathcal{F}_n$
$$\|E_nA-AE_n\|\leq{\varepsilon/2^n}$$

L'\'el\'ement $E_n^2=u_n-u_{n-1}$ est dans $\mathcal{K}$; comme $(u_n)_{n\in\mathbb{N}}$ est une unit\'e approch\'ee pour $\mathcal{K}$ (dont l'espace nul est trivial), la suite $(u_n)_{n\in\mathbb{N}}$ converge fortement vers $1$, ce qui implique l'\'egalit\'e
$$\sum_{n=1}^{+\infty}E_n^2=1$$

c/ Il reste \`a montrer que pour tout \'el\'ement $A$ dans $\mathcal{A}$, l'image $l(A)$ est une perturbation compacte de l'op\'erateur $A$.
Or pour tout ${A}$ dans $\bigcup\mathcal{F}_n$, on a par construction $\sum\|E_nA-AE_n\|<+\infty$.
D'apr\`es la remarque $\ref{RemCool}$ , combin\'ee avec l'\'egalit\'e suivante  $$A-l(A)=A-\sum_{n\geq{1}}{E_nAE_n}=\sum_{n\geq{1}}(AE_n^2-E_nAE_n)=\sum_{n\geq{1}}(AE_n-E_nA)E_n$$
on a alors $A-l(A)\in{\mathcal{K}}$, pour tout $A$ dans $\bigcup\mathcal{F}_n$.
\\
Or pour tout ${A}$ dans $\mathcal{A}$, il existe une suite $(A_k)_{k\in\NN}$ d'\'el\'ements de $\bigcup\mathcal{F}_n$, qui converge en norme vers $A$. D'apr\`es la remarque $\ref{RemCool}$ , $A-l(A)$ est bien d\'efinie et la suite $\big(A_k-l(A_k)\big)_{k\in\NN}$ converge en norme vers $A-l(A)$. Comme $A_k-l(A_k)$ est compact pour tout $k$ entier, on en d\'eduit que pour tout $A$ dans $\mathcal{A}$, l'\'el\'ement $A-l(A)$ est dans ${\mathcal{K}}$.
\\
\end{proof}
\vspace{+2mm}
\begin{Def}
Soit $\mathcal{A}$ une $C(X)$-alg\`ebre, consid\'er\'ee comme sous-$C(X)$-alg\`ebre de la $C(X)$-alg\`ebre des op\'erateurs born\'es, $C(X)$-lin\'eaires admettant un adjoint sur un $C(X)$-module de Hilbert s\'eparable. On suppose l'existence d'une suite $(E_n)_{n\in\mathbb{N}}$ d'op\'erateurs positifs, de rang fini satisfaisant $\sum_nE_n^2=1$. On pose $\omega : \mathcal{A}\rightarrow\mathcal{A}$ l'application $C(X)$-lin\'eaire positive qui, \`a tout $a$ de $\mathcal{A}$, associe l'op\'erateur $\sum_{n\in\NN}E_naE_n$
L'application $\omega : \mathcal{A}\rightarrow\mathcal{A}$ est dite application localisante pour $\mathcal{A}$ si pour tout ${a}$ dans $\mathcal{A}$, l'op\'erateur $a-\omega(a)$ est compact.
\end{Def}

\subsection  {Repr\'esentation r\'eguli\`ere et op\'erateurs de rang fini}

On a vu dans la section $\ref{Rappel}$, la construction de $C^*$-alg\`ebres associ\'ees \`a un groupo\"{\i}de localement compact et s\'epar\'e muni d'un syst\`eme de Haar.  Pour simplifier les notations, on note dans la suite $\mathcal{L}(L^2(\mathcal{G},\nu))$  pour la $C_0(X)$-alg\`ebre des op\'erateurs $\mathcal{L}_{C_0(X)}\big(L^2(\mathcal{G},\nu)\big)$.  
\vspace{+1mm}
\\

La repr\'esentation r\'eguli\`ere est le $\ast$-homomorphisme injectif de $C^*$-alg\`ebres $\lambda : C_r^*(\mathcal{G})\rightarrow{\mathcal{L}\big(L^2(\mathcal{G},\nu)\big)}$ et, d'apr\`es la proposition $\ref{ProAlgMod}$, la $C^*$-alg\`ebre $\mathcal{L}\big(L^2(\mathcal{G},\nu)\big)$ est munie d'une structure de $C_0(X)$-alg\`ebre. L'image par la repr\'esentation $\lambda$ de $C_r^*(\mathcal{G})$ est une sous-$C^*$-alg\`ebre de $\mathcal{L}\big(L^2(\mathcal{G},\nu)\big)$, qui n'est pas n\'ecessairement stable par la structure de $C_0(X)$-module de $\mathcal{L}\big(L^2(\mathcal{G},\nu)\big)$.
\vspace{+2mm}
\\
Dans cette partie on d\'efinit une sous-$C_0(X)$-alg\`ebre de $\mathcal{L}(L^2(\mathcal{G},\nu))$, que l'on note $\mathcal{A}$, \`a partir de la $C^*$-alg\`ebre r\'eduite $C_r^*(\mathcal{G})$. On montre, via la proposition \ref{ProConsExa}, que cette construction conserve la propri\'et\'e d'exactitude. L'id\'ee est de remplacer la sous-$C^*$-alg\`ebre $C_r^*(\mathcal{G})$ exacte par la sous-$C_0(X)$-alg\`ebre $\mathcal{A}$ de mani\`ere \`a pouvoir utiliser un th\'eor\`eme de \cite{Blanchard97} qui donne l'existence d'un plongement $C_0(X)$-lin\'eaire de toute $C_0(X)$-alg\`ebre continue et exacte dans la $C_0(X)$-alg\`ebre $C_0(X)\otimes\mathcal{O}_2$. On termine cette partie en appliquant le th\'eor\`eme \ref{TheoCool} \`a la $C_0(X)$-alg\`ebre $\mathcal{A}$ et l'id\'eal des op\'erateurs compacts : la suite d'op\'erateurs positifs compacts que l'on obtient, joue un r\^ole central dans l'approximation par factorisation de l'application identit\'e de $\mathcal{A}$. 
\vspace{+2mm}

\begin{Def}{\label{DefAlgGra}}
On note $\mathcal{A}=C^*\big(\lambda(C^*_r(\mathcal{G})),\phi(C_0(X))\big)$, la plus petite sous $C^*$-alg\`ebre d'op\'erateurs de $L^2(\mathcal{G},\nu)$ contenant $\phi\big(C_0(X)\big)$ et $\lambda\big(C^*_r(\mathcal{G})\big)$. 
\end{Def}
Pour tout $a$ dans $C^*_r(\mathcal{G})$, l'op\'erateur $\lambda(a)$ sur $L^2(\mathcal{G},\nu)$ est $C(X)$-lin\'eaire. La $C^*$-alg\`ebre $\mathcal{A}$ est engendr\'ee par les \'el\'ements $\sum_{i\in{I}}{\phi(f_i).\lambda(a_i)}$ o\`u la somme est finie et $a_i$ est un \'el\'ement de $C^*_r(\mathcal{G})$ et la fonction $f_i$ est dans $C(X)$.

\begin{Pro}
La $C^*$-alg\`ebre $\mathcal{A}$ est une sous-$C_0(X)$-alg\`ebre de la $C_0(X)$-alg\`ebre $\mathcal{L}\big(L^2(\mathcal{G},\nu)\big)$. 
\end{Pro}

\begin{Rem}
Nous consid\'erons \`a partir de maintenant la $C_0(X)$-alg\`ebre $\mathcal{A}$ comme une sous-$C_0(X)$-alg\`ebre de la $C_0(X)$-alg\`ebre des op\'erateurs $\mathcal{L}\big(L^2(\mathcal{G},\nu)\big)$.
\end{Rem}

\begin{Def}
Soit $\mathcal{B}$ une $C_0(X)$-alg\`ebre. On dit que $\mathcal{B}$ une $C_0(X)$-alg\`ebre continue si le champ semi-continu sup\'erieurement de $C^*$-alg\`ebres sur $X$ est continu, c'est-\`a-dire si pour tout $b$ dans $\mathcal{B}$, la fonction $x\rightarrow{\lVert{b_x}\lVert}$ est continue.
\end{Def}

\begin{Notation}
On dit que le groupo\"{\i}de $\mathcal{G}\rightrightarrows{X}$ v\'erifie la condition $\mathbf{(Cont)}$ si la $C_0(X)$-alg\`ebre associ\'ee $\mathcal{A}$ introduite dans la d\'efinition $\ref{DefAlgGra}$ est continue. On suppose \`a partir de maintenant que le groupo\"{\i}de \'etale $\mathcal{G}\rightrightarrows{X}$ v\'erifie la condition $\mathbf{(Cont)}$ c'est-\`a-dire que la $C_0(X)$-alg\`ebre $\mathcal{A}$ est continue. 
\vspace{+4mm}
\end{Notation}

La proposition qui suit permet de voir que la propri\'et\'e d'exactitude de la $C^*$-alg\`ebre r\'eduite $C^*_r(\mathcal{G})$ est transf\'er\'ee \`a la $C^*$-alg\`ebre $\mathcal{A}$ :
\begin{Pro}{\label{ProConsExa}}
Si la $C^*$-alg\`ebre $C^*_r(\mathcal{G})$ est exacte, alors la $C(X)$-alg\`ebre $\mathcal{A}$ est exacte.
\end{Pro}
\begin{proof}
On note $A$ la $C^*$-alg\`ebre $C^*_r(\mathcal{G})$ et on suppose que $A$ est exacte; la $C^*$-alg\`ebre $C(X)$ \'etant nucl\'eaire, $A\otimes{C(X)}$ est exacte. On d\'efinit un homomorphisme $\pi : A\otimes{C(X)}\rightarrow\mathcal{A}$ d\'efini pour les \'el\'ements de la forme
$\sum_i{a_i\otimes{f_i}}$ par 
$$\pi\big(\sum_i{a_i\otimes{f_i}}\big)=\sum_i{\phi(f_i).\lambda(a_i)}$$ 
Cet homomorphisme est surjectif et on a un isomorphisme de $C^*$-alg\`ebres de $A\otimes{C(X)}/I$ dans $\mathcal{A}$, o\`u $I:=\ker({\pi})$. D'apr\`es un r\'esultat de Kirchberg (voir une version (IV.3.4.19) de [Blac]) affirmant que tout quotient de $C^*$-alg\`ebre s\'eparable exacte est exact, la $C(X)$-alg\`ebre $\mathcal{A}$ est exacte.
\\
\end{proof}

On rappelle ci-dessous le th\'eor\`eme de stabilisation de Kasparov pour les modules de Hilbert : 
\begin{Def}
 A toute $C^*$-alg\`ebre $B$, on associe le module de Hilbert standard sur $B$, que  l'on note $\mathcal{H}_{B}$ d\'efini par 
$$\mathcal{H}_{B}=\bigg\{(b_k)\in\prod_{k\in\NN}B\textrm{ : }\sum_{k\in\NN}b_k^*b_k\textrm{ converge dans }B\bigg\}$$
\end{Def}
\begin{Theo}[\textbf{Stabilisation de Kasparov}]{\label{TheoStabKas}}
Pour tout $B$-module de Hilbert $\mathcal{F}$, d\'enombrablement engendr\'e, il existe un isomorphisme de $B$-module de Hilbert entre $\mathcal{F}\oplus{\mathcal{H}_B}$ et $\mathcal{H}_B$.
\end{Theo}
\vspace{+2mm}
On trouve une d\'emonstartion du th\'eor\`eme \ref{TheoStabKas} dans \cite{Kasparov80} ou encore dans \cite{Lance95} 
: il s'agit d'une variante de l'orthogonalisation de Gram-Schmidt. A l'aide du th\'eor\`eme de stabilisation de Kasparov \ref{TheoStabKas} , on prouve la proposition suivante
\begin{Pro}
Il existe un homomorphisme de $C(X)$-alg\`ebres injectif de la $C(X)$-alg\`ebre $\mathcal{L}\big(L^2(\mathcal{G},\nu)\otimes{l^2(\NN)}\big)$ dans $\mathcal{L}\big(\mathcal{H}_{C(X)}\big)$
\end{Pro}
\begin{proof}
Comme $l^2(\NN)$ est un espace de Hilbert s\'eparable donc d\'enombrablement engendr\'e $(e_n)_{n\in\NN}$ et
$L^2(\mathcal{G},\nu)$ est un $C(X)$-module de Hilbert d\'enombrablement engendr\'e, alors le $C(X)$-module de Hilbert $L^2(\mathcal{G},\nu)\otimes{l^2(\NN)}$ est d\'enombrablement engendr\'e. En appliquant le th\'eor\`eme $\ref{TheoStabKas}$ de stabilisation de Kasparov au $C(X)$-module de Hilbert $L^2(\mathcal{G},\nu)\otimes{l^2(\NN)}$ d\'enombrablement engendr\'e, il existe un isomorphisme de $C(X)$-module hilbertien entre $L^2(\mathcal{G},\nu)\otimes{l^2(\NN)}\oplus\mathcal{H}_{C(X)}$ et $\mathcal{H}_{C(X)}$.  L'homomorphisme de $C(X)$-module de Hilbert 
$$i : L^2(\mathcal{G},\nu)\otimes{l^2(\NN)}\hookrightarrow{L^2(\mathcal{G},\nu)\otimes{l^2(\NN)}\oplus{\mathcal{H}_{C(X)}}}$$ \'etant injectif, il existe un plongement $C(X)$-lin\'eaire $\mathcal{L}(L^2(\mathcal{G},\nu)\otimes{l^2(\NN)})$ dans $\mathcal{L}(\mathcal{H}_{C(X)})$.
\end{proof}


On termine cette partie en appliquant le th\'eor\`eme \ref{TheoCool} \`a la $C(X)$-alg\`ebre d'op\'erateurs $\mathcal{A}$. On note $\mathcal{K}_0$ l'id\'eal des op\'erateurs de rang fini sur le $C(X)$-module de Hilbert $\mathcal{L}(\mathcal{H}_{C(X)})$, d\'ecrit auparavant. On suppose, quitte \`a remplacer $\mathcal{A}$ par la $C(X)$-alg\`ebre des perturbations compactes des \'el\'ements de $\mathcal{A}$, que l'ensemble des op\'erateurs compacts $\mathcal{K}$ est dans $\mathcal{A}$.
\\
On consid\`ere $(e_n)_{n\in\NN}$, la suite dans $\mathcal{L}\big(C(X)\otimes{l^2(\NN)}\big)$, telle que pour tout entier $n$, on a $e_n=1\otimes{p_n}$, avec $p_n$ la projection sur les $n$ premi\`eres composantes de $l^2(\NN)$. On obtient une approximation de l'unit\'e dans l'id\'eal $\mathcal{K}_0$, constitu\'ee de projections.  
On va ainsi pouvoir construire une approximation de l'unit\'e quasicentrale $(u_n)$ dans l'id\'eal $\mathcal{K}_0$ de $\mathcal{A}$ c'est-\`a-dire une suite croissante $(u_n)_{n\in\NN}$ d'\'el\'ements positifs de la boule unit\'e de $\mathcal{K}$, telle que pour tout $a$ dans $\mathcal{A}$, $\lim_{n\rightarrow{\infty}}{\big\lVert{au_n-u_na}\big\lVert}=0$ c'est-\`a-dire $$\lim_{n\rightarrow{\infty}}\sup_{x\in{X}}\big\lVert{a(x)u_n(x)-u_n(x)a(x)}\big\lVert=0$$

\begin{Lemma}
L'id\'eal $\mathcal{K}_0$ des op\'erateurs de rang fini dans la $C(X)$-alg\`ebre $\mathcal{A}$ poss\`ede une approximation de l'unit\'e quasicentrale $(u_n)_{n\in\NN}$, obtenue \`a partir de combinaisons lin\'eaires convexes d'\'el\'ements de l'approximation de l'unit\'e $(e_n)_{n\in\NN}$ de $\mathcal{K}_0$. 
\end{Lemma}

\begin{proof}
La preuve est directe en utilisant le th\'eor\`eme \ref{TheoArv}.
\end{proof}

\begin{Lemma}
L'espace nul de l'id\'eal $\mathcal{K}_0$ est trivial.
\end{Lemma}

\begin{Pro}\label{ProSuiAlgRem}
Il existe une suite d'op\'erateurs $(E_n)_{n\in\NN}$ positifs dans l'ensemble des op\'erateurs de rang fini $\mathcal{K}_0$, telle que $\sum_{n\in\NN}E_n^2=1$ et pour tout $A$ dans $\mathcal{A}$, on a $\sum_{n\in\NN}E_nAE_n$ est une pertubation compacte de $A$. De plus, pour tout r\'eel strictement positif $\varepsilon$ et tout toute partie finie (ou m\^eme compacte) $\mathcal{F}$ dans $\mathcal{A}$, on peut choisir $(E_n)_{n\in\NN}$ de sorte que $\lVert{A-\sum_{n\in\NN}E_nAE_n}\lVert<\varepsilon$, pour tout $A$ dans $\mathcal{F}$.
\end{Pro}
\begin{proof}
L'ensemble $\mathcal{K}_0$ des op\'erateurs de rang fini sur le $C(X)$-module de Hilbert $\mathcal{H}_{C(X)}$ est un id\'eal bilat\`ere dans la $C(X)$-alg\`ebre $\mathcal{A}$, dont l'espace nul est r\'eduit au vecteur nul de $\mathcal{H}_{C(X)}$, et on a construit une unit\'e approch\'ee dans $\mathcal{K}_0$ quasicentrale pour $\mathcal{A}$. D'apr\`es le th\'eor\`eme \ref{TheoCool} il existe une suite $(E_n)_{n\in\NN}$ d'op\'erateurs positifs dans l'ensemble des op\'erateurs compacts $\mathcal{K}$, (qui est la fermeture pour la norme de l'id\'eal $\mathcal{K}_0$) v\'erifiant pour tout entier $n$, $E_n^2$ est un op\'erateur de rang fini, $\sum_nE_n^2=1$ et pour tout op\'erateur $A$ de $\mathcal{A}$, l'op\'erateur $A-\sum_nE_nAE_n$ est compact. De plus, pour tout entier $n$, l'op\'erateur $E_n$ est de rang fini car $E_n^2$ l'est.
\\
\end{proof}
La suite d'op\'erateurs positifs compacts $(E_n)_{n\in\NN}$ donn\'ee par la proposition \ref{ProSuiAlgRem} nous permet de construire dans les prochaines \'etapes une application localisante dans $\mathcal{A}$ et des projections dans $\mathcal{L}(\mathcal{H}_{C(X)})$.
\vspace{+1mm}

\subsection  {$C(X)$-nucl\'earit\'e et exactitude}

On consid\`ere \`a partir de maintenant et ceci jusqu'\`a la fin de cette section que l'espace des unit\'es $X$ du groupo\"{\i}de \'etudi\'e est compact. Dans cette partie, on s'int\'eresse aux relations entre l'exactitude et la nucl\'earit\'e dans la cat\'egorie des $C(X)$-alg\`ebres continues.  
\\
On rappelle tout d'abord les notions de nucl\'earit\'e pour les $C(X)$-alg\`ebres. On d\'efinit l'alg\`ebre de Cuntz $\mathcal{O}_2$ puis on montre que la $C(X)$-alg\`ebre continue $C(X)\otimes\mathcal{O}_2$ est $C(X)$-nucl\'eaire. On utilise alors un r\'esultat de Kirchberg qui apparait dans l'appendice de \cite{Blanchard97} pour d\'efinir un homomorphisme injectif de $C(X)$-alg\`ebres $\pi : \mathcal{A}\to{C(X)\otimes\mathcal{O}_2}$. On se sert alors de la $C(X)$-nucl\'earit\'e de $C(X)\otimes\mathcal{O}_2$ pour obtenir les premi\`eres approximations par factorisations de l'homomorphisme identit\'e $id_\mathcal{A}$. 
\\

Dans la th\'eorie des $C^*$-alg\`ebres, la nucl\'earit\'e d'une $C^*$-alg\`ebre $B$ peut \^etre caract\'eris\'ee par l'existence d'approximations par les applications compl\`etement positives de rang fini de l'homomorphisme identit\'e de $B$. On dit que l'homomorphisme identit\'e est nucl\'eaire. Cette d\'efinition s'\'etend aux applications lin\'eaires compl\`etement positives.
\\   
On d\'efinit une notion analogue de $C(X)$-nucl\'earit\'e pour les applications $C(X)$-lin\'eaires compl\`etement positives  entre $C(X)$-alg\`ebres.
\begin{Def}
a) Soient $A_1$ et $A_2$ deux $C(X)$-alg\`ebres. Une application $C(X)$-lin\'eaire compl\`etement positive $\theta : A_1\rightarrow{A_2}$ est dite $C(X)$-nucl\'eaire si et seulement si pour tout sous ensemble compact $K$ dans $A_1$ et tout r\'eel strictement positif $\varepsilon$, il existe un entier $n$ et des applications $C(X)$-lin\'eaires, compl\`etement positives et contractantes $\phi_n : A_1\rightarrow{C(X)\otimes{M_n(\CC)}}$ et $\psi_n : C(X)\otimes{M_n(\CC)}\rightarrow{A_2}$ telles que pour tout $a$ dans $K$, on ait 
$$\big\lVert{\theta(a)-\psi_n\circ\phi_n(a)}\big\lVert<\varepsilon$$
b) Une $C(X)$-alg\`ebre $A$ est dite $C(X)$-nucl\'eaire si l'homomorphisme identit\'e est $C(X)$-nucl\'eaire.
\end{Def}
\vspace{+2mm}
\begin{Def}
On appelle $C^*$-alg\`ebre de Cuntz unif\`ere la $C^*$-alg\`ebre  engendr\'ee par deux isom\'etries $s_1\textrm{, }s_2$ satisfaisant la relation
$1 = s_1s^*_1 + s_2s^*_2$. On note $\mathcal{O}_2$ cette $C^*$-alg\`ebre.
\end{Def}
\vspace{+2mm}
\begin{Pro}
La $C(X)$-alg\`ebre $C(X)\otimes{\mathcal{O}_2}$ est $C(X)$-nucl\'eaire. 
\end{Pro}


\begin{proof}
La $C^*$-alg\`ebre $\mathcal{O}_2$ \'etant nucl\'eaire, pour tout sous espace compact $K$ de $\mathcal{O}_2$ et tout r\'eel strictement positif $\varepsilon$, il existe un entier $n$ et des applications lin\'eaires continues, compl\`etement positives et contractantes, not\'ees $\varphi : \mathcal{O}_2\rightarrow{M_n(\CC)}$ et $\psi : M_n(\CC)\rightarrow\mathcal{O}_2$ v\'erifiant pour tout $o$ dans $K$, 
$$\lVert{o-\psi\circ\varphi(o)}\lVert<\varepsilon$$
L'alg\`ebre $\mathcal{O}_2\otimes{C(X)}$ est une $C(X)$-alg\`ebre continue (c'est un champ continu trivial de $C^*$-alg\`ebres sur $X$ de fibre constante $\mathcal{O}_2$). On consid\`ere $\mathcal{K}$ un sous espace compact de $\mathcal{O}_2\otimes{C(X)}$ et $\varepsilon$ un r\'eel strictement positif. Pour tout $f$ dans $\mathcal{K}$, on pose $U_f$ le voisinage ouvert d\'efini par $U_f=\big\{g\in\mathcal{O}_2\otimes{C(X)}\textrm{ : }\lVert{f-g}\lVert<\varepsilon/3\big\}$. On obtient alors un recrouvrement ouvert de $\mathcal{K}$. Par compacit\'e, il existe une famille finie $\{f_1,...,f_p\}$ d'\'el\'ements de $\mathcal{K}$, tels que $\mathcal{K}\subset{\cup_{j=1}^pU_{f_j}}$. Pour tout $f$ dans $\mathcal{K}$, il existe $j$ dans $\{1,...,p\}$ tel que pour tout $x$ dans $X$, on a 
$$\lVert{f(x)-f_j(x)}\lVert<\varepsilon/3$$ 
On peut identifier la $C(X)$-alg\`ebre $\mathcal{O}_2\otimes{C(X)}$ avec la $C(X)$-alg\`ebre $C(X,\mathcal{O}_2)$ et pour tout $j$ dans $\{1,...,p\}$, on a $\overline{f_j(K)}$ est un sous espace compact de $\mathcal{O}_2$. La r\'eunion finie $K=\cup_{j=1}^p\overline{f_j(K)}$ est un sous espace compact de $\mathcal{O}_2$, et il existe donc $n$ un entier et des contractions lin\'eaires continues et compl\`etement positives, not\'ees $\tilde\varphi : \mathcal{O}_2\rightarrow{M_n(\CC)}$ et $\tilde\psi : M_n(\CC)\rightarrow\mathcal{O}_2$, v\'erifiant $\lVert{f_j(x)-\tilde\psi\circ\tilde\varphi(f_j(x))}\lVert<\varepsilon/3$, pour tout $x$ dans $X$ et tout $j$ dans $\{1,...,p\}$. On pose $\varphi : \mathcal{O}_2\otimes{C(X)}\rightarrow{M_n\otimes{C(X)}}$ application $C(X)$-lin\'eaire compl\`etement positive et contractante d\'efinie par $\varphi:=\tilde\varphi\otimes{id_{C(X)}}$ et $\psi : M_n\otimes{C(X)}\rightarrow\mathcal{O}_2\otimes{C(X)}$ application $C(X)$-lin\'eaire compl\`etement positive et contractante d\'efinie par $\psi:=\tilde\psi\otimes{id_{C(X)}}$ et on a pour tout  $j$ dans $\{1,...,p\}$, 
\begin{align*}
\big\lVert{f_j-\psi\circ\phi(f_j)}\big\lVert&=\sup_{x\in{X}}\big\lVert{f_j(x)-(\tilde\psi\otimes{id_{C(X)}})\circ(\tilde\varphi\otimes{id_{C(X)}})(f_j)(x)}\big\lVert
\\
&=\sup_{x\in{X}}\big\lVert{f_j(x)\tilde\psi\circ\tilde\varphi(f_j(x))}\big\lVert<\frac{\varepsilon}{3}
\end{align*}
Pour tout $f$ dans $\mathcal{K}$, il existe $j$ dans $\{1,...,p\}$ tel que pour tout $x$ dans $X$, on a $\lVert{f(x)-f_j(x)}\lVert<\varepsilon/3$ et on a 
\begin{align*}
\big\lVert{f-\psi\circ\varphi(f)}\big\lVert&=\big\lVert{f-f_j+f_j-\psi\circ\varphi(f_j)+\psi\circ\varphi(f_j)-\psi\circ\varphi(f)}\big\lVert
\\
&\leq\big\lVert{f-f_j}\big\lVert+\big\lVert{f_j-\psi\circ\varphi(f_j)}\big\lVert+\big\lVert{\psi\circ\varphi(f_j)-\psi\circ\varphi(f)}\big\lVert
\\
&<\frac{\varepsilon}{3}+\frac{\varepsilon}{3}+\big\lVert{\psi\circ\varphi(f_j-f)}\big\lVert<\frac{\varepsilon}{3}+\frac{\varepsilon}{3}+\big\lVert{\psi\circ\varphi}\big\lVert\big\lVert{f_j-f}\big\lVert
\\
&<\frac{\varepsilon}{3}+\frac{\varepsilon}{3}+\frac{\varepsilon}{3}<\varepsilon
\end{align*}
On a donc prouv\'e que pour tout r\'eel strictement positif et tout sous ensemble compact $\mathcal{K}$ de $\mathcal{O}_2\otimes{C(X)}$, il existe un entier $n$ et deux applications $\varphi : \mathcal{O}_2\otimes{C(X)}\rightarrow{M_n\otimes{C(X)}}$ et $\psi : M_n\otimes{C(X)}\rightarrow\mathcal{O}_2\otimes{C(X)}$, $C(X)$-lin\'eaires continues, compl\`etement positives et contractantes telles que, pour tout $f$ dans $\mathcal{K}$, on a 
$$\lVert{f-\psi\circ\varphi(f)}\lVert<\varepsilon$$
Donc la $C(X)$-alg\`ebre $\mathcal{O}_2\otimes{C(X)}$ est $C(X)$-nucl\'eaire.
\\
\end{proof}
\vspace{+2mm}

Dans [Blan], l'auteur donne une caract\'erisation des $C(X)$-alg\`ebres continues et nucl\'eaires, \'etendue au cas des $C(X)$-alg\`ebres continues et exactes par Kirchberg (voire appendice de [Blan]). Il d\'emontre le th\'eor\`eme suivant :  
\begin{Theo}[\cite{Blanchard97}]{\label{TheoBlan}}
Soit $\mathcal{B}$ une $C(X)$-alg\`ebre continue et s\'eparable. La $C^*$-alg\`ebre $\mathcal{B}$ est exacte si et seulement si il existe un plongement $C(X)$-lin\'eaire de $\mathcal{B}$ dans la $C(X)$-alg\`ebre $C(X)\otimes\mathcal{O}_2$.
\end{Theo}
On applique alors le th\'eor\`eme $\ref{TheoBlan}$ \`a la $C(X)$-alg\`ebre continue et exacte $\mathcal{A}$ et on obtient le corollaire suivant :
\begin{Cor}{\label{TheoBlanCor}}
Il existe un homomorphisme injectif de $C(X)$-alg\`ebres $\pi : \mathcal{A}\to{C(X)\otimes\mathcal{O}_2}$.
\end{Cor}
\begin{proof}
La $C(X)$-alg\`ebre $\mathcal{A}$ \'etant continue et exacte, d'apr\`e le th\'eor\`eme $\ref{TheoBlan}$, il existe un homomorphisme injectif de $C(X)$-alg\`ebres $\pi : \mathcal{A}\to{C(X)\otimes\mathcal{O}_2}$.
\\
\end{proof}
La proposition qui suit permet de plonger la $C(X)$-alg\`ebre continue $C(X)\otimes\mathcal{O}_2$ dans $\mathcal{L}\big(\mathcal{H}_{C(X)}\big)$ :
\begin{Pro}
Il existe un homomorphisme injectif de $C(X)$-alg\`ebres de $\mathcal{O}_2\otimes{C(X)}$ dans $\mathcal{L}\big(\mathcal{H}_{C(X)}\big)$.
\end{Pro}
\begin{proof}
La $C^*$-alg\`ebre de Cuntz $\mathcal{O}_2$ \'etant s\'eparable, il existe, par le th\'eor\`eme de Gelfand-Naimark, un homomorphisme injectif $i : \mathcal{O}_2\hookrightarrow{L(l^2(\NN))}$. On a alors 
$$i\otimes{id_{C(X)}} : \mathcal{O}_2\otimes{C(X)}\hookrightarrow{L(l^2(\NN))\otimes{C(X)}}\hookrightarrow{\mathcal{L}\big(l^2(\NN)\otimes{C(X)}\big)}$$ 
$C(X)$ \'etant une $C^*$-alg\`ebre $\sigma$-unitale, on a $\mathcal{L}\big(l^2(\NN)\otimes{C(X)}\big)\cong\mathcal{L}(\mathcal{H}_{C(X)})$.\\
Il existe donc un homomorphisme injectif de $C(X)$-alg\`ebres de $\mathcal{O}_2\otimes{C(X)}$ dans $\mathcal{L}(\mathcal{H}_{C(X)})$.
\\
\end{proof}

\begin{Rem}
Le probl\`eme que nous devons \'etudier, \`a savoir l'approximation de la repr\'esentation r\'eguli\`ere de $C_r^*(\mathcal{G})$, peut se r\'esumer grossi\`erement par le diagramme suivant : 
\begin{eqnarray*}
\xymatrix{\mathcal{A} \ar@{^{(}->}[rr] \ar@{^{(}->}[dd]^\pi  && \mathcal{L}\big(L^2(\mathcal{G},\nu)\otimes{l^2(\NN)}\big) \ar@{^{(}->}[r] & \mathcal{L}\big(\mathcal{H}_{C(X)})  
\\ 
 & & & 
\\
C(X)\otimes\mathcal{O}_2 \ar[rr]^{id} \ar@{.>}[rd]  & & C(X)\otimes\mathcal{O}_2 \ar@{^{(}->}[r] & \mathcal{L}\big(\mathcal{H}_{C(X)}\big)  \ar@/_2pc/@{<.>}[uu]_{?} 
\\
 & C(X)\otimes{M_n(\CC)} \ar@{.>}[ru] } & &
\end{eqnarray*}

\end{Rem}

\subsection  {Etude de $E_n$ : projection et application localisante}
On utilise la suite $(E_n)_{n\in\NN}$ d'op\'erateurs positifs compacts obtenus dans la proposition \ref{ProSuiAlgRem} pour construire des projections $P_{E_n}$ dans $\mathcal{L}(\mathcal{H}_{C(X)})$ et une application localisante $l : \mathcal{A}\to\mathcal{A}$ d\'efinie par $l(A)=\sum_nE_nAE_n$. L'int\'er\^et d'introduire ces objets est la d\'emonstration de la proposition \ref{ProAppComPos} . 
\vspace{+1mm}
\\

On rappelle un th\'eor\`eme classique pour les op\'erateurs sur un module hilbertien, dont l'image est ferm\'ee (voir \cite{Lance95} ou \cite{ManTro05}) 
\begin{Pro}{\label{ProLance}}
Soit $B$ une $C^*$-alg\`ebre, $\mathcal{E}$ et $\mathcal{F}$ des $B$-modules hilbertiens.  Soit $t\in\mathcal{L}(\mathcal{E},\mathcal{F})$ \`a image ferm\'ee, alors on a 
\begin{enumerate}
\item[a)] Le noyau $\ker(t)$ est un sous module de $\mathcal{E}$ qui admet un sous module suppl\'ementaire.
\item[b)] L'image $Im(t)$ est un sous module de $\mathcal{F}$ qui admet un sous module suppl\'ementaire.
\item[c)] $t^\ast\in\mathcal{L}(\mathcal{F},\mathcal{E})$ a une image ferm\'ee.
\end{enumerate} 
\end{Pro}

\begin{Pro}
Pour tout entier $n$ dans $\NN$, l'image de l'op\'erateur $E_n$ est ferm\'ee dans $\mathcal{H}_{C(X)}$ et admet un sous module suppl\'ementaire.
\end{Pro}

\begin{proof}
Pour tout entier $n$, l'op\'erateur $E_n$ est positif donc $E_n^\ast{E_n}=E_n^2$. Par d\'efinition, on a $E_n^2=u_{n}-u_{n-1}$ et le spectre $\sigma(E_n^2)$ est un ensemble fini, donc $\sigma(E_n^2)\smallsetminus\{0\}$ est ferm\'ee. L'op\'erateur $E_n$ a une image ferm\'ee dans $\mathcal{H}_{C(X)}$.
\\
En appliquant la proposition ${\ref{ProLance}}$ \`a l'op\'erateur $E_n$, on en d\'eduit que l'image $Im(E_n)$ admet un sous module suppl\'ementaire dans $\mathcal{H}_{C(X)}$.
\\
\end{proof}

\begin{Rem}
On note $P_{E_n} : \mathcal{H}_{C(X)}\rightarrow{\mathcal{H}_{C(X)}}$ la projection associ\'ee au sous module $Im(E_n)$. 
\end{Rem}

\begin{Lemma}
Pour tout $n$ dans $\NN$, il existe un entier $N_n$ tel que l'image de l'op\'erateur $E_n$ soit incluse dans $C(X)\otimes\CC^{N_n}$.
\end{Lemma}
\begin{proof}
Pour tout entier $n$, on a, par construction, $E_n^2=u_{n+1}-u_n$. Or chaque $u_n$ est une combinaison lin\'eaire convexe d'\'el\'ements de la suite $(e_n)_{n\in\NN}$. Les op\'erateurs $E_n^2$ peuvent alors s'\'ecrire sous la forme $E_n^2=\sum_{n\in\NN}\lambda_ne_n$, o\`u $\lambda_n$ est non nul pour un nombre fini de $n$ et v\'erifient $\sum_n\lambda_n=0$. Puisque chaque op\'erateur $E_n^2$ est diagonal et de rang fini, il existe un entier positif $N_n$ tel que l'image de $E_n$ soit inclus dans ${C(X)}\otimes{\CC^{N_n}}$.
\\
\end{proof}

\begin{Rem}
On pose $\check{\mathcal{H}}:=\oplus_{n\in\NN}Im(E_n)$. L'espace $\check{\mathcal{H}}$ est muni d'une structure de $C(X)$-module hilbertien s\'eparable (voir \cite{Lance95} pour la construction g\'en\'erale) : on d\'efinit un \'el\'ement de $\check{\mathcal{H}}$ comme une suite $(x_n)_{n\in\NN}$, o\`u $x_n$ est dans $E_n$, et v\'erifiant $\sum_n\langle{x_n,x_n}\rangle$ convergence dans $C(X)$ et la structure de $C(X)$-module est \'evidente (voir \cite{Lance95} pour la construction g\'en\'erale). 
\end{Rem}

\begin{Rem}
On a par construction de $\check{\mathcal{H}}$ un isomorphisme de $C(X)$-module hilbertien entre $\check{\mathcal{H}}$ et $\mathcal{H}_{C(X)}$
\end{Rem}

\begin{Lemma}
Pour tout $n$ dans $\NN$, il existe $N_n$ entier, tel que $P_{E_n}(\mathcal{H}_{C(X)})$ soit isomorphe au $C(X)$-module de Hilbert $C(X)\otimes{\CC}^{N_n}$.
\end{Lemma}

%

\begin{Lemma}
Soit $V : \mathcal{H}_{C(X)}\rightarrow\check{\mathcal{H}}$ l'application qui \`a tout vecteur $\xi$ de $\mathcal{H}_{C(X)}$ associe le vecteur $V(\xi)=\oplus_{n\in\NN}E_n(\xi)$. L'application $V : \mathcal{H}_{C(X)}\rightarrow\check{\mathcal{H}}$ est une isom\'etrie. 
\end{Lemma}
\begin{proof}
Pour tout $\xi$ dans $\mathcal{H}_{C(X)}$, on a 
\begin{align*}
\lVert{V(\xi)}\lVert^2&=\lVert{\oplus_{n\in\NN}}E_n(\xi)\lVert^2=\sum_{n\in\NN}\big\langle{E_n(\xi),E_n(\xi)}\big\rangle
=\sum_{n\in\NN}\big\langle{E_n^2(\xi),\xi}\big\rangle
\\
&=\bigg\langle{\sum_{n\in\NN}E_n^2(\xi),\xi}\bigg\rangle
=\big\langle{\xi,\xi}\big\rangle=\lVert{\xi}\lVert^2
\end{align*}
Pour tout $\eta=\oplus_{n\in\NN}\eta_n$ de $\check{\mathcal{H}}$
et tout $\xi$ dans $\mathcal{H}_{C(X)}$, on a 
\begin{align*}
\langle{V\xi,\eta}\rangle&=\big\langle{\oplus_{n\in\NN}E_n\xi,\oplus_{n\in\NN}}\eta_n\big\rangle=\sum_{n\in\NN}\big\langle{E_n\xi,\eta_n}\big\rangle
=\sum_{n\in\NN}\big\langle{\xi,E^\ast_n\eta_n}\big\rangle
\\
&=\bigg\langle{\xi,\sum_{n\in\NN}E_n\eta_n}\bigg\rangle=\langle{\xi,V^\ast\eta}\rangle
\end{align*}
Son adjoint $V^\ast$ est d\'efini pour tout $\eta=\oplus_{n\in\NN}\eta_n$ de $\check{\mathcal{H}}$ par 
$V^*(\eta)=\sum_{n\in\NN}E_n(\eta_n)$.
\\
\end{proof}

L'application $l : \mathcal{A}\rightarrow\mathcal{A}$ d\'efinie pour tout $A$ dans $\mathcal{A}$ par $l(A)=\sum_{n\in\NN}E_nAE_n$ est telle que $A-l(A)$ est un op\'erateur compact, pour tout $A$ dans $\mathcal{A}$.

\begin{Def}{\label{DefDelta}}
On pose, pour tout $n$ entier, l'application $\delta_n : \mathcal{A}\rightarrow{P_{E_n}(\mathcal{H}_{C(X)})}$ d\'efinie, pour tout $A$ dans $\mathcal{A}$, par $\delta_n(A)=P_{E_n}AP_{E_n}$. On dit que $\delta_n(A)$ est la compression de l'op\'erateur $A$ sur le sous module $P_{E_n}\big(\mathcal{H}_{C(X)}\big)$.
\\
On pose $\delta=\oplus_{n\in\NN}\delta_n$ l'application sur $\mathcal{A}$ \`a valeurs dans les op\'erateurs $\mathcal{L}(\check{\mathcal{H}})$.
\end{Def}

\begin{Pro}{\label{ProAppComPos}}
Il existe un $C(X)$-module hilbertien s\'eparable $\check{\mathcal{H}}$, une isom\'etrie $V : \mathcal{H}_{C(X)}\rightarrow{\check{\mathcal{H}}}$ et une application $C(X)$-lin\'eaire, compl\`etement positive $\delta : \mathcal{A}\rightarrow\mathcal{L}(\check{\mathcal{H}})$ tels que pour tout ${A}$ dans ${\mathcal{A}}$, l'op\'erateur $A$ est une perturbation compacte de ${V^*\delta(A)V}$.
\end{Pro}
\begin{proof}
Le $C(X)$-module hilbertien s\'eparable $\check{\mathcal{H}}$, l'isom\'etrie $V$ et l'application $C(X)$-lin\'eaire compl\`etement positive ont \'et\'e construits auparavant. 
De plus, on a pour tout entier ${n}$ et tout ${A}$ dans $\mathcal{A}$, la relation $E_n\delta_n(A)E_n=E_nAE_n$, par d\'efinition de $\delta_n(A)$. Ainsi on a  
\begin{align*}
V^*\delta(A)V=V^*\big(\oplus_{n\in{\NN}}AE_n\big)=\sum_{n\in{\NN}}E_nAE_n=l(A)
\end{align*}
L'application $l$ \'etant localisante, alors pour tout $A$ dans $\mathcal{A}$ l'op\'erateur $A$ est une perturbation compacte de $V^*\delta(A)V$.
\\
\end{proof}

\subsection  {Champ de formes vectorielles}
Dans cette partie, on s'int\'eresse aux applications $C(X)$-lin\'eaires et compl\`etement positives $\delta_i : \mathcal{A}\to{M_{n_i}(C(X))}$, d\'efinies pour tout entier $i$ dans \ref{DefDelta} : on associe \`a $\delta_i$ une application $\delta'_i : M_{n_i}(\mathcal{A})\to{C(X)}$ \'egalement $C(X)$-lin\'eaire et compl\`etement positive qui est un $C(X)$-\'etat et on montre que $\delta'_i$ peut \^etre approch\'ee par des $C(X)$-\'etats vectoriels d\'efinis eux-m\^eme \`a l'aide de la $C(X)$-repr\'esentation fid\`ele $\pi_{n_i} : M_{n_i}(\mathcal{A})\to\mathcal{L}(\mathcal{H}_{C(X)}^{n_i})$ induite par $\pi : \mathcal{A}\to{\mathcal{L}(\mathcal{H}_{C(X)})}$ obtenue au corollaire $\ref{TheoBlanCor}$. On utilisera les r\'esultats d'approximation d'un \'etat d'une $C^*$-alg\`ebre par des sommes \'etats vectoriels que l'on peut retrouver dans le livre \cite{Dixmier69}.
\\

On consid\`ere la $C(X)$-repr\'esentation fid\`ele $\pi : \mathcal{A}\hookrightarrow\mathcal{L}({\mathcal{H}}_{C(X)})$, obtenue \`a partir du plongement $C(X)$-lin\'eaire $\pi : \mathcal{A}\rightarrow{C(X)\otimes\mathcal{O}_2}$. Pour tout $i$ dans $\NN$, on note $\pi_{n_i}:=\pi\otimes{id}: \mathcal{A}\otimes{M_{n_i}}\rightarrow{\mathcal{L}\big(\mathcal{H}_{C(X)}\big)\otimes{M_{n_i}}}$.

\begin{Lemma}
Pour tout $i$ dans $\NN$, l'application $\pi_{n_i} : M_{n_i}(\mathcal{A})\rightarrow\mathcal{L}\big(\mathcal{H}^{n_i}_{C(X)}\big)$ est un champ de repr\'esentations fid\`eles.
\end{Lemma}
\begin{proof}
Soit $i$ un entier. L'application $\pi_{n_i} : M_{n_i}(\mathcal{A})\rightarrow\mathcal{L}\big(\mathcal{H}^{n_i}_{C(X)}\big)$ est \'evidemment un homomorphisme de $C(X)$-alg\`ebres. Il s'agit de montrer pour tout $x$ dans $X$, que l'homomorphisme $\pi_{n_i,x} : M_{n_i}(\mathcal{A}_x)\rightarrow{\mathcal{L}\big((\mathcal{H}_{C(X)}^{n_i})_x\big)}$ est une repr\'esentation fid\`ele de la fibre en $x$. On a, pour tout $x$ dans $X$ et tout $A$ dans $M_{n_i}(\mathcal{A})$,
\begin{align*}
\lVert{\pi_{n_i,x}(a_x)}\lVert&=\inf{\big\{\lVert{\pi_{n_i}(a)-f.\pi_{n_i}(a)}\lVert\textrm{, }f\in{C_x(X)}}\big\}
\\
&=\inf{\big\{\lVert{\pi_{n_i}(a-f.a)}\lVert\textrm{, }f\in{C_x(X)}\big\}}
\\
&=\inf{\big\{\lVert{a-f.a}\lVert\textrm{, }f\in{C_x(X)}}\big\}
\\
&=\lVert{a_x}\lVert
\end{align*} 
On a montr\'e, pour tout $x$ dans $X$, que $\lVert{\pi_{n_i,x}}\lVert=1$, c'est-\`a-dire que $\pi_{n_i,x}$ est une repr\'esentation fid\`ele de la fibre $M_{n_i}(\mathcal{A}_x)$ dans la $C^*$-alg\`ebre des op\'erateurs de l'espace de Hilbert $\mathcal{L}\big((\mathcal{H}^{n_i}_{C(X)})_x\big)$. Ainsi $\pi_{n_i}$ est un champ de repr\'esentations fid\`eles.
\\
\end{proof}



On rappelle un r\'esultat classique pour les \'etats, dans le cadre g\'en\'eral des $C^*$-alg\`ebres dont les d\'etails se trouvent dans \cite{Dixmier69}. 
\\
Soit $B$ une $C^*$-alg\`ebre et $H$ un espace de Hilbert s\'eparable. On consid\`ere une repr\'esentation fid\`ele de $B$ dans l'alg\`ebre des op\'erateurs de $H$, que l'on note $\pi : B\hookrightarrow{L({H})}$. 
Pour tout vecteur $\xi$ de $H$, on note $\omega_\xi$ la forme lin\'eaire continue positive sur $B$, d\'efinie par la repr\'esentation $\pi$ de $B$ et le vecteur $\xi$, c'est-\`a-dire l'application lin\'eaire continue positive d\'efinie pour tout $b$ dans $B$, par 
$$\omega_\xi(b)=\langle{\xi,\pi(b)\xi}\rangle$$
Une telle forme lin\'eaire sera appel\'ee forme vectorielle. Si la repr\'esentation $\pi$ est non d\'eg\'en\'er\'ee et si le vecteur $\xi$ de $H$ est unitaire, alors la forme lin\'eaire $\omega_\xi$ est un \'etat et on parle alors d'\'etat vectoriel sur $B$.

\begin{Theo}{\label{TheoLimFaible}}\cite{Dixmier69}
Soient $H$ un espace de Hilbert, $K$ l'ensemble des op\'erateurs compacts sur $H$, $B$ une  sous-$C^*$-alg\`ebre d'op\'erateurs de $L({H})$ et $\phi$ est un \'etat sur $B$ nul sur $B\cap{{K}}$. Alors $\phi$ est limite $\ast$-faible d'\'etats vectoriels de $B$.
\end{Theo}


\begin{Def}\cite{Blanchard96}
Soit $\mathcal{B}$ une $C(X)$-alg\`ebre. On appelle champ continu d'\'etats (ou $C(X)$-\'etat) une application $C(X)$-lin\'eaire positive $\varphi$ de $\mathcal{B}$ dans $C(X)$ telle que, pour tout $x$ dans $X$, l'application $\varphi_x = e_x\circ\varphi$ soit un \'etat sur $\mathcal{B}_x$.
\end{Def}

\begin{Def}
Soit $\mathcal{B}$ une $C(X)$-alg\`ebre continue et $\pi : \mathcal{B}\rightarrow\mathcal{L}(\mathcal{H})$ une $C(X)$-repr\'esentation fid\`ele dans la $C(X)$-alg\`ebre des op\'erateurs sur le $C(X)$-module hilbertien $\mathcal{H}$. Un $C(X)$-\'etat vectoriel $\varphi : \mathcal{B}\rightarrow{C(X)}$ est un $C(X)$-\'etat pour lequel il existe un vecteur $\xi$ dans $\mathcal{H}$, tel que $\xi_x$ est unitaire dans $\mathcal{H}_x$ pour tout $x$ dans $X$ et v\'erifiant $\varphi(b)=\langle\xi,\pi(b)\xi\rangle$ pour tout $b$ dans $\mathcal{B}$.
\end{Def}





\begin{Pro}
Pour tout $i$ dans ${\NN}$, l'application $\delta_i : \mathcal{A}\rightarrow{M_{n_i}(C(X))}$ est $C(X)$-lin\'eaire et compl\`etement positive.
\end{Pro}

\begin{proof}
Soit $i$ un entier. L'application $\delta_i : \mathcal{A}\rightarrow{M_{n_i}}\otimes{C(X)}$ est $C(X)$-lin\'eaire car, pour tout $A$ dans $\mathcal{A}$, on a $\delta_i(A)=P_{E_i}AP_{E_i}$, o\`u $P_{E_i}$ et $A$ (vu comme op\'erateurs sur $\mathcal{H}_{C(X)}$) sont $C(X)$-lin\'eaires.
\\
Pour montrer que $\delta_i$ est compl\`etement positive, on montre que pour tout entier $n$, tout $(A_1,...,A_n)$ dans $\mathcal{A}$ et tout $(m_1,...,m_n)$ dans $M_{n_i}\otimes{C(X)}$, on a 
$$S_n=\sum_{p,q=1}^nm_p^\ast\delta_i(A_p^\ast{A_q})m_q\geq{0}$$
Or pour tout entier $n$, tout $(A_1,...,A_n)$ dans $\mathcal{A}$ et tout $(m_1,...,m_n)$ dans $M_{n_i}\otimes{C(X)}$, on a 
\begin{align*}
S_n&=\sum_{p,q=1}^nm_p^\ast\delta_i(A_p^\ast{A_q})m_q=\sum_{p,q=1}^nm_p^\ast\big({P}_{E_i}A_p^\ast{A_q}P_{E_i}\big)m_q=\sum_{p,q=1}^n\big(A_p{P}_{E_i}m_p\big)^\ast{A_q}P_{E_i}m_q
\\
&=\sum_{p=1}^n\sum_{q=1}^n\big(A_p{P}_{E_i}m_p\big)^\ast{A_q}P_{E_i}m_q=\sum_{p=1}^n\bigg(\big(A_pP_{E_i}m_p\big)^\ast\big({\sum_{q=1}^n}A_qP_{E_i}m_q\big)\bigg)
\\
&=\bigg(\sum_{p=1}^n(A_pP_{E_i}m_p)\bigg)^\ast\bigg(\sum_{q=1}^n(A_qP_{E_i}m_q)\bigg)\geq{0}
\end{align*}
\vspace{+2mm}
\end{proof}

\begin{Lemma}{\label{LemmeCorrespondence}}
Pour tout entier $n$ et toute $C(X)$-alg\`ebre $A$, il existe une correspondance entre les applications $C(X)$-lin\'eaires compl\`etement positives de ${A}$ dans $M_{n}(\CC)\otimes{C(X)}$ et les applications $C(X)$-lin\'eaires compl\`etement positives de $M_{n}(\CC)\otimes{A}$ dans $C(X)$.
\end{Lemma}
\begin{proof}
On trouve un preuve dans \cite{Blackadar06}  ($\mathrm{II}.6.9.11$).
\end{proof}
\begin{Rem}
On associe \`a $\delta_i : \mathcal{A}\rightarrow{M_{n_i}}\otimes{C(X)}$, par la correspondance donn\'ee dans le lemme $\ref{LemmeCorrespondence}$ , $\delta_i' : M_{n_i}(\mathcal{A})\rightarrow{C(X)}$ l'application $C(X)$-lin\'eaire compl\`etement positive. 
On va montrer que $\delta_i' : M_{n_i}(\mathcal{A})\rightarrow{C(X)}$ est un $C(X)$-\'etat et de mani\`ere analogue au th\'eor\`eme \ref{TheoLimFaible} sur les \'etats, on prouve, pour tout $i$ dans $\NN$, que le $C(X)$-\'etat $\delta_i' : M_{n_i}(\mathcal{A})\rightarrow{C(X)}$ est limite $\ast$-faible de $C(X)$-\'etats vectoriels.
\end{Rem}

\begin{Pro}
L'application $\delta_i'$ est un un champ continu de formes lin\'eaires continues et positives. 
\end{Pro}

\begin{proof}
L'application $\delta'_i$ est une application $C(X)$-lin\'eaire (compl\`etement) positive de $M_{n_i}(\mathcal{A})$ dans $C(X)$. De plus pour tout $x$ dans ${X}$, l'application $\delta'_{i,x}:=e_x\circ{\delta'_i}$ est une forme lin\'eaire continue et positive sur $M_{n_i}(\mathcal{A})_x\cong{M_{n_i}(\mathcal{A}_x)}$.
\\
\end{proof}

Pour prouver que $\delta_i'$ est limite $\ast$-faible de $C(X)$-\'etats, on va proc\'eder en deux \'etapes : 
\begin{itemize}
\item[\textbf{1)}] on fait une \'etude locale de $\delta_i$ en chaque point $x$ de $X$ et on utilise les r\'esultats connus de formes lin\'eaires et continues pour les $C^*$-alg\`ebres 
\item[\textbf{2)}] on utilise une partition de l'unit\'e pour revenir dans le cadre global des champs continus de formes lin\'eaires
\vspace{+2mm}
\\
\end{itemize}

\textbf{Etape 1)} Pour $i$ dans $\NN$ et $x$ dans ${X}$ fix\'es, on consid\`ere la forme lin\'eaire continue et positive d\'efinie par $\delta'_{i,x} : M_{n_i}(\mathcal{A}_x)\rightarrow{\CC}$, obtenue par la composition $e_x\circ\delta'_i$.
\\
\begin{Pro}{\label{ProChap3}}
Pour tout r\'eel $\varepsilon$ strictement positif et toute partie finie $\mathcal{F}$ de ${M_{n_i}(\mathcal{A})}$, il existe ${\xi(i,x)}$ dans $({\mathcal{H}_{C(X)}})^{n_i}_x$ tel que pour tout ${A}$ dans $\mathcal{F}$, on ait
$$\big|\delta_{i,x}'(A_x)-\omega_{\xi(i,x)}(\pi_{n_i,x}(A_x))\big|<\varepsilon$$
\end{Pro}
\begin{proof}


En utilisant la repr\'esentation fid\`ele $\pi_{n_i,x}$, on consid\`ere $M_{n_i}(\mathcal{A}_x)$ comme sous-$C^*$-alg\`ebre de $\mathcal{L}\big((\mathcal{H}_{C(X)}^{n_i})_x\big)$. Quitte \`a tensoriser par $l^2(\NN)$, on peut supposer que $\mathcal{A}_x\cap{\mathcal{K}\big((\mathcal{H}_{C(X)}^{n_i})_x\big)}$ est r\'eduit \`a l'\'el\'ement nul. L'application $\delta_{i,x}' : M_{n_i}(\mathcal{A}_x)\rightarrow{\CC}$ est une forme lin\'eaire continue et positive. 
\\
En appliquant le th\'eor\`eme $\ref{TheoLimFaible}$ , on en conclut que $\delta'_{i,x}$ est limite $\ast$-faible de formes vectorielles. Ainsi pour tout $\varepsilon>0$, toute partie finie $\mathcal{F}$ de $M_{n_i}(\mathcal{A})$, il existe un vecteur $\xi(i,x)$ dans $\big(\mathcal{H}^{n_i}_{C(X)}\big)_x$ tel que, pour tout $A$ dans $\mathcal{F}$,
 $$\big|\delta_{i,x}'(A_x)-\omega_{\xi(i,x)}(\pi_{n_i,x}(A_x))\big|<\varepsilon$$
\end{proof}

\begin{Rem}
Puisque $\mathcal{H}_{C(X)}^{n_i}$ est un $C(X)$-module de Hilbert, il existe suffisamment de sections au sens o\`u, pour tout $\eta\in{\big({\mathcal{H}^{n_i}_{C(X)}}\big)_x}$, il existe $\tilde\eta\in{\big({\mathcal{H}^{n_i}_{C(X)}}\big)}$ tel que $\eta=\tilde\eta_x$.
\\
Il existe un vecteur $\xi^{(x)}_i\in{\big({\mathcal{H}^{n_i}_{C(X)}}\big)}$ tel que $\xi_{i,x}^{(x)}=\xi({i},x)$.
\end{Rem}
\begin{Lemma}{\label{LemChap3}}
Pour tout r\'eel $\varepsilon$ strictement positif et toute partie finie $\mathcal{F}$ de ${M_{n_i}(\mathcal{A})}$, il existe un voisinage ouvert $U_{i,x}$ de $x$ dans ${X}$, tel que pour tout $y$ dans $U_{i,x}$ et tout $A$ dans $\mathcal{F}$, on a les in\'egalit\'es suivantes 
\begin{align*}
\left\{\begin{array}{rl}
&\big|\delta'_{i,x}(A_x)-\delta'_{i,y}(A_y)\big|<\varepsilon
\\
&\big|\omega_{\xi^{(x)}_{i,x}}(A_x)-\omega_{\xi^{(x)}_{i,y}}(A_y)\big|<\varepsilon  \end{array} \right.
\end{align*}
\end{Lemma}
\begin{proof}
En effet, la premi\`ere in\'egalit\'e r\'esulte du fait que $\delta'_i$ est un $C(X)$-\'etat c'est-\`a-dire un champ continu d'\'etats sur $X$. La seconde in\'egalit\'e s'obtient en consid\'erant la fonction continue $f:=\omega_{\xi_i^{(x)}}(A)=\langle\xi^{(x)}_i;\pi_i(A)\xi^{(x)}_i\rangle$ de ${C(X)}$. On a alors pour tout ${\varepsilon>0}$ un voisinage ouvert $V_x$ de $x$ tel que pour tout ${y}\in{V_x}$ on a
\begin{eqnarray*}
|f(x)-f(y)| &=& \big|\big\langle\xi^{(x)}_i,\pi_{n_i}(A)\xi^{(x)}_i\big\rangle(x)-\big\langle\xi^{(x)}_i,\pi_{n_i}(A)\xi^{(x)}_{i}\big\rangle(y)\big| \\
&=& \big|\big\langle\xi^{(x)}_{i,x},\pi_{n_i,x}(A_x)\xi^{(x)}_{i,x}\big\rangle-\big\langle\xi^{(x)}_{i,y},\pi_{n_i,y}(A_y)\xi^{(x)}_{i,y}\big\rangle\big|<\varepsilon
\end{eqnarray*}
\end{proof}

\begin{Pro}
Pour tout r\'eel $\varepsilon$ strictement positif et toute partie finie $\mathcal{F}$ de ${M_{n_i}(\mathcal{A})}$, il existe un voisinage ouvert $U_{i,x}$ de $x$ dans ${X}$, tel que pour tout $y$ dans $U_{i,x}$ et tout $A$ dans $\mathcal{F}$, on a 
$$\big\lvert{\delta_{i,y}'(A_y)-\omega_{\xi_{i,y}^{(x)}}(A_y)}\big\lvert<\varepsilon$$
\end{Pro}
\begin{proof} 
On utilise la propri\'et\'e \ref{ProChap3} et le lemme \ref{LemChap3} pour le r\'eel $\varepsilon/3$ et la famille finie $\mathcal{F}$. On a pour tout $A$ dans $\mathcal{F}$ et pour tout $y\in{U_{i,x}}$
\begin{eqnarray*}
D&=&\lvert{\delta'_{i,y}(A)-\omega_{\xi^{(x)}_{i,y}}(A_y)}\lvert 
\\
&=& \big\lvert\delta'_{i,y}(A_y)-\delta'_{i,x}(A_x)+\delta'_{i,x}(A_x)-\omega_{\xi^{(x)}_{i,x}}(A_x)+\omega_{\xi^{(x)}_{i,x}}(A_x)-\omega_{\xi^{(x)}_{i,y}}(A_y)\big\lvert 
\\
&<& \big\lvert\delta'_{i,y}(A_y)-\delta'_{i,x}(A_x)\big\lvert+\big\lvert\delta'_{i,x}(A_x)-\omega_{\xi^{(x)}_{i,x}}(A_x)\big\lvert+\big\lvert\omega_{\xi^{(x)}_{i,y}}(A_x)-\omega_{\xi^{(x)}_{i,y}}(A_y)\big\lvert<3\times\frac{\varepsilon}{3} 
\\
&<& \varepsilon
\end{eqnarray*}
\end{proof}

\begin{Rem}
Pour tout r\'eel $\varepsilon$ strictement positif et toute partie finie $\mathcal{F}$ de ${M_{n_i}(\mathcal{A})}$, il existe un vecteur ${\xi^{(x)}_i}$ dans $({\mathcal{H}^{n_i}_{C(X)}})$ et un voisinage ouvert de $x$, not\'e ${U_{i,x}}$, tel que $\xi^{(x)}_{i,y}$ est unitaire dans $\big({\mathcal{H}^{n_i}_{C(X)}}\big)_y$ pour tout $y\in{U_{i,x}}$ et v\'erifiant pour tout ${y}$ dans ${U_{i,x}}$ et tout ${A}$ dans ${\mathcal{F}}$
$$\big\lvert\delta_{i,y}'(A)-\omega_{\xi^x_i(y)}(A)\big\lvert<\varepsilon$$
\end{Rem}
$$$$

\textbf{Etape 2)} On cherche maintenant \`a revenir aux champs continus de formes lin\'eaires et continues sur l'espace $X$. 
On fixe $i$ un entier dans $\NN$, un r\'eel $\varepsilon$ strictement positif et une partie finie $\mathcal{F}$ dans ${M_{n_i}(\mathcal{A})}$.
\\
Pour l'op\'erateur $\delta'_{i} : M_{n_i}(A)\rightarrow{C(X)}$, on peut, pour tout $x$ dans ${X}$, r\'eit\'erer la d\'emarche pr\'ec\'edente assurant l'existence d'un voisinage ouvert $U_{i,x}$ de $x$ dans ${X}$ et d'un vecteur $\xi^{(x)}_{i}$ dans ${({\mathcal{H}_{C(X)}})^{n_i}}$ tel que pour tout $y$ dans $U_{i,x}$, le vecteur $\xi^{(x)}_{i,y}$est unitaire dans $\big({\mathcal{H}^{n_i}_{C(X)}}\big)_y$. On consid\`ere $\tilde{\mathcal{U}}_i=\{U_{i,x}\}_{x\in{X}}$ le recouvrement ouvert de $X$. 
\\
\begin{Rem}
Comme $X$ est un espace topologique suppos\'e compact, il existe un ensemble fini $F_i$ dans $X$ tel que $\mathcal{U}_i=\{U_{i,\alpha}\}_{\alpha\in{F_i}}$ soit un recouvrement ouvert fini de $X$. On note $\{\xi^\alpha_{i}\}_{\alpha\in{F_i}}$ les vecteurs de $({\mathcal{H}_{C(X)}})^{n_i}$. 
\\
On pose $\{\phi_{i,\alpha}\}_{\alpha\in{F_i}}$ une partition de l'unit\'e associ\'ee \`a ce recouvrement d'ouverts, o\`u $\phi_{i,\alpha} : X\rightarrow{[0,1]}$ fonctions continues \`a support compact dans $U_{i,\alpha}$ v\'erifiant l'\'egalit\'e $\sum_{\alpha\in{F}}\phi_{i,\alpha}(x)=1$, pour tout $x$ dans $X$.
\end{Rem}

\begin{Lemma}
Il existe une application $C(X)$-lin\'eaire $\omega_i : M_{n_i}(\mathcal{A})\rightarrow{C(X)}$ telle que, pour tout $y$ dans $X$ et tout $A$ dans $\mathcal{F}$, on a $\big|\delta'_{i,y}(A)-\omega_{i}(A)(y)\big|<\varepsilon$
\end{Lemma}

\begin{proof}
On pose $\omega_i:=\sum_{\alpha\in{F_i}}{\phi_{i,\alpha}\omega_{\xi^\alpha_{i}}}$, application ${C(X)}$-lin\'eaire de $M_{n_i}(\mathcal{A})$ dans $C(X)$ et on a
\begin{align*}
\big\lvert\delta'_{i}(A)(y)-\omega_{i}(A)(y)\big\lvert&=\bigg\lvert\sum_{\alpha\in{F_i}}\phi_{i,\alpha}(y)\delta'_{i,y}(A)-\sum_{\alpha\in{F_i}}\phi_{i,\alpha}(y)\omega_{\xi^\alpha_{i}}(A)(y)\bigg\lvert 
\\  
&<\sum_{\alpha\in{F_i}}\phi_{i,\alpha}(y)\big\lvert\delta'_{i,y}(A)-\omega_{\xi^\alpha_{i}}(A)(y)\big\lvert<\varepsilon
\end{align*}
\end{proof}
On consid\`ere $F_i$ de la forme $\{1,...,m_i\}$ et on note $\xi_i:=\big(\sqrt{\phi_{i,1}}\xi_{i}^1;...;\sqrt{\phi_{i,m_i}}\xi_{i}^{m_i}\big)$,  \'el\'ement de $({\mathcal{H}_{C(X)}^{n_i}})^{m_i}$. En effet, on a
$$\big\langle\xi_i,\xi_i\big\rangle=\sum_{\alpha=1}^{m_i}\big\langle\sqrt{\phi_{i,\alpha}}\xi_{i}^\alpha,\sqrt{\phi_{i,\alpha}}\xi_{i}^\alpha\big\rangle=\sum_{\alpha=1}^{m_i}\phi_{i,\alpha}\big\langle\xi_{i}^\alpha,\xi_{i}^\alpha\big\rangle\in{C(X)}$$
\begin{Lemma}
L'application $\omega_i$ est un champ continu de formes lin\'eaires positives vectorielles.
\end{Lemma}
\begin{proof}
Si on pose $\omega_{\xi_{i}}:=\sum_{\alpha=1}^{m_i}\phi_{i,\alpha}\big\langle\xi_{i}^\alpha,\pi_{n_i}(A)\xi_{i}^\alpha\big\rangle$, on obtient un $C(X)$-\'etat vectoriel tel que
$$\omega_{\xi_{i}}=\omega_i$$
\end{proof}
\begin{Rem}
Pour tout entier ${i}$ dans $\NN$, tout r\'eel $\varepsilon$ strictement positif et toute partie finie $\mathcal{F}$ de ${M_{n_i}(A)}$, on a obtenu une approximation de chacun des $C(X)$-\'etats $\delta_i'$ par un $C(X)$-\'etat vectoriel $\omega_{\xi_i}$, o\`u $\xi_i$ est dans $({\mathcal{H}^{n_i}_{C(X)}})^{m_i}$.
\end{Rem}

On a montr\'e que pour tout entier ${i}$, l'application $\delta'_i : M_{n_i}(\mathcal{A})\rightarrow{C(X)}$ est limite $\ast$-faible de $C(X)$-\'etats c'est-\`a-dire que pour tout r\'eel ${\varepsilon}$ strictement positif et toute partie finie ${\mathcal{F}}$ de ${M_{n_i}(\mathcal{A})}$ il existe un vecteur $\xi_i$ dans ${\big({\mathcal{H}^{n_i}_{C(X)}}\big)^{m_i}}$ tel que pour tout ${A}$ dans ${\mathcal{F}}$
$$\big\|\delta'_i(A)-\omega_{\xi_i}(A)\big\|<\varepsilon$$

\subsection  {Approximation par factorisation}{\label{ApproxiFacto}}

Pour tout entier $i$, on a construit le vecteur $\xi_i=\big(\sqrt{\phi_{i,1}}\xi_{i}^1;...;\sqrt{\phi_{i,m_i}}\xi_{i}^{m_i}\big)$ dans le $C(X)$-module hilbertien $\oplus_{k=1}^{m_i}\oplus_{j=1}^{n_i}\mathcal{H}_{C(X)}$. On adoptera les notations suivantes, en supposant que $\xi_i$ s'\'ecrive sous la forme 
  
$$  \xi_i= \left (
   \begin{array}{cccc} 
   \begin{pmatrix}
   \xi_i^{1,1} \\ \vdots \\ \xi_i^{n_i,1}
   \end{pmatrix},
& \begin{pmatrix}
   \xi_i^{1,2} \\ \vdots \\ \xi_i^{n_i,2}
   \end{pmatrix},
& \cdots 
&\begin{pmatrix}
   \xi_i^{1,m_i} \\ \vdots \\ \xi_i^{n_i,m_i}
   \end{pmatrix} 
   \end{array}
   \right )$$
o\`u chaque $\xi_i^{p,q}$ est un \'el\'ement de $\mathcal{H}_{C(X)}$.
\\
On note, pour $k$ entier, $\xi_i^{(k)}$ l'\'el\'ement de $\big(\mathcal{H}_{C(X)}\big)^{m_i}$, d\'efini par 
 $\xi_i^{(j)}=(\xi_i^{j,1}, \xi_i^{j,2},..., \xi_i^{j,m_i})$.
   


\begin{Def}
On note, pour tout $i$ entier, $V_i : C(X)^{n_i}\rightarrow{\big({\mathcal{H}^{n_i}_{C(X)}}\big)^{m_i}}$, l'application lin\'eaire d\'efinie par $V_i(f_1,...,f_{n_i})=\oplus_{k=1}^{m_i}\oplus_{j=1}^{n_i}f_j\xi_i^{j,k}=\oplus_{j=1}^{n_i}f_j\xi_i^{(j)}$
\end{Def}
\begin{Pro}
L'application $V_i : C(X)^{n_i}\rightarrow{\big({\mathcal{H}^{n_i}_{C(X)}}\big)^{m_i}}$ est une application $C(X)$-lin\'eaire, born\'ee et poss\'edant un adjoint.
\end{Pro}
\begin{proof}
L'application $V_i$ est $C(X)$-lin\'eaire. Pour montrer qu'elle est born\'ee on consid\`ere $(f_1,...,f_{n_i})$ un \'el\'ement de $C(X)^{n_i}$. On a alors
\begin{align*}
\big\lVert{V_i(f_1,...,f_{n_i})}\big\lVert^2&=\sup_{x\in{X}}\big\lvert{\langle{\oplus_{j=1}^{n_i}\xi_i^{(j)}f_j,\oplus_{j=1}^{n_i}\xi_i^{(j)}f_j}\rangle}(x)\big\lvert
\\
&=\sup_{x\in{X}}\bigg\lvert{\sum_{j=1}^{n_i}\lvert{f_j}\lvert^2(x)}\langle{\xi_i^{(j)},\xi_i^{(j)}}\rangle(x)\bigg\lvert=\sup_{x\in{X}}{\sum_{j=1}^{n_i}\lvert{f_j(x)}\lvert^2}\lvert{g_j(x)}\lvert^2
\\
&\leq{\big\lVert{(f_1,...,f_{n_i})}\big\lVert^2}{\big\lVert{(g_1,...,g_{n_i})}\big\lVert^2}\leq{\big\lVert{(f_1,...,f_{n_i})}\big\lVert^2}
\end{align*}
L'application $C(X)$-lin\'eaire $V_i$ est donc born\'ee.
\\
On montre que $V_i^*: {\big({\mathcal{H}_{C(X)}^{n_i}}\big)^{m_i}}\rightarrow{C(X)^{n_i}}$ est d\'efinie, pour tout $\eta$ dans $\big({\mathcal{H}_X^{n_i}}\big)^{m_i}$, par $$V_i^*(\eta)=\big(\langle{\xi_i^{(j)},\eta^{(j)}}\rangle\big)_{1\leqslant{j}\leqslant{n_i}}$$
En effet, pour tout $(f_1,...,f_{n_i})$ dans $C(X)^{n_i}$ et tout $\eta$ dans $(\mathcal{H}_X^{n_i})^{m_i}$, on a 
\begin{align*}
\big\langle{V_i(f_1,...,f_{n_i}),\eta}\big\rangle&=\bigg(\sum_{k=1}^{m_i}\sum_{q=1}^{n_i}\big\langle{\xi_i^{j,k}f_j,\eta^{j,k}f_j}\big\rangle\bigg)=\sum_{j=1}^{n_i}f_j\sum_{k=1}^m\langle{\xi_i^{j,k},\eta^{j,k}}\rangle
\\
&=\sum_{j=1}^{n_i}f_j\langle{\xi_i^{(j)},\eta^{(j)}}\rangle=\bigg\langle{(f_1,...,f_{n_i}),\big(\langle{\xi_i^{(j)},\eta^{(j)}}\rangle\big)_{1\leqslant{j}\leqslant{n_i}}}\bigg\rangle
\\
&=\big\langle{(f_1,...,f_{n_i}),V_i^*(\eta)}\big\rangle
\end{align*}
L'application $C(X)$-lin\'eaire born\'ee $V_i$ admet un op\'erateur adjoint. 
\\
De plus pour $(f_1,...,f_{n_i})$ dans ${C(X)^{n_i}}$, on a 
\begin{align*}
V_i^*V_i^{}(f_1,...,f_{n_i})&=V_i^*\bigg(\oplus_{j=1}^{n_i}f_j\xi_i^{(j)}\bigg)=\big(\big\langle\oplus_{k=1}^{m}\xi_i^{j,k},f_j\xi_i^{(j)}\big\rangle\big)_{1\leqslant{j}\leqslant{n_i}}
\\
&=\big(f_j\big\langle\xi_i^{(j)},\xi_i^{(j)}\big\rangle\big)_{1\leqslant{j}\leqslant{n_i}}=\bigg(f_j\sum_{k=1}^m\langle{\xi_i^{j,k},\xi_i^{j,k}}\rangle\bigg)_{1\leqslant{j}\leqslant{n_i}}
\\
&=(g_1f_1,...,g_{n_i}f_{n_i})
\end{align*}
o\`u chaque $g_j$ est une fonction continue positive non nulle sur $X$ d\'efinie, pour tout entier $j$, par $g_j=\langle{\xi_i^{(j)},\xi_i^{(j)}}\rangle$. 
\\
L'application $V_i : C(X)^{n_i}\rightarrow{\big({\mathcal{H}^{n_i}_{C(X)}}\big)^{m_i}}$ est dans $\mathcal{L}_{C(X)}\big(C(X)^{n_i},\big({\mathcal{H}^{n_i}_{C(X)}}\big)^{m_i}\big)$
\end{proof}
\begin{Rem}
Par le lemme \ref{LemmeCorrespondence} donnant la correspondance entre les applications $C(X)$-lin\'eaires compl\`etement positives $f: \mathcal{A}\rightarrow{M_{n_i}(C(X))}$ et les applications $C(X)$-lin\'eaires compl\`etement positives $f : M_{n_i}(\mathcal{A})\rightarrow{C(X)}$, on associe \`a l'application $\omega_{\xi_i} : M_{n_i}(\mathcal{A})\rightarrow{C(X)}$, l'application C(X)-lin\'eaire et compl\`etement positive $\tilde\omega_{\xi_i} : \mathcal{A}\rightarrow{M_{n_i}(C(X))}$.
\end{Rem}
\begin{Lemma}
L'application $C(X)$-lin\'eaire et compl\`etement positive $\tilde\omega_{\xi_i} : \mathcal{A}\rightarrow{M_{n_i}(C(X))}$ est d\'efinie pour tout $A$ dans $\mathcal{A}$ par $$\tilde\omega_{\xi_i}(A)=\bigg(\big\langle{\oplus_{k=1}^{m_i}\xi_i^{p,k},\oplus_{k=1}^{m_i}\big(\pi(A)\xi_i^{q,k}\big)}\big\rangle\bigg)_{1\leqslant{p,q}\leqslant{n_i}}$$
\end{Lemma}
\begin{proof}
L'application \'etant $C(X)$-lin\'eaire et compl\`etement positive sur $M_{n_i}(\mathcal{A})$, on peut, pour toute matrice $A=(A_{p,q})$ dans $M_{n_i}(\mathcal{A})$, \'ecrire $\omega_{\xi_i}(A)$ sous la forme $\sum_{p,q=1}^{n_i}\omega_{p,q}(A_{p,q})$. L'application $C(X)$-lin\'eaire et compl\`etement positive $\tilde\omega_{\xi_i}$ est d\'efini, pour tout $A$ dans $\mathcal{A}$, par la formule
$$\tilde\omega_{\xi_i}(A)=\sum_{p=1}^{n_i}\sum_{q=1}^{n_i}e_{p,q}\otimes\omega_{p,q}(A)$$
Pour toute matrice $A=(A_{p,q})$ dans $M_{n_i}(\mathcal{A})$, on a 
\begin{align*}
\omega_{\xi_i}(A)&=\big\langle{\xi_i,\big(\oplus_{k=1}^m\pi_{n_i}(A)\big)\xi_i}\big\rangle
\\
&=\big\langle{\oplus_{k=1}^{m}\oplus_{j=1}^{n_i}\xi_i^{j,k},\oplus_{k=1}^{m}\big(\pi_{n_i}(A)\oplus_{j=1}^{n_i}\xi_i^{j,k}\big)}\big\rangle
\\
&=\sum_{k=1}^m\big\langle{\oplus_{j=1}^{n_i}\xi_i^{j,k},\pi_{n_i}\big((A_{p,q})\big)\oplus_{j=1}^{n_i}\xi_i^{j,k}}\big\rangle
\\
&=\sum_{k=1}^m\big\langle{\oplus_{j=1}^{n_i}\xi_i^{j,k},\oplus_{p=1}^{n_i}\big(\sum_{q=1}^{n_i}\pi(A_{p,q})\xi_i^{q,k}\big)}\big\rangle=\sum_{k=1}^m\sum_{p=1}^{n_i}\big\langle{\xi_i^{p,k},\sum_{q=1}^{n_i}\pi(A_{p,q})\xi_i^{q,k}}\big\rangle
\\
&=\sum_{p=1}^{n_i}\sum_{q=1}^{n_i}\big\langle{\oplus_{k=1}^m\xi_i^{p,k},\oplus_{k=1}^{m}\pi(A_{p,q})\xi_i^{q,k}}\big\rangle=\sum_{p=1}^{n_i}\sum_{q=1}^{n_i}\omega_{p,q}(A_{p,q})
\end{align*}
o\`u $\omega_{p,q}(A_{p,q})=\big\langle{\oplus_{k=1}^{m}\xi_i^{p,k},\oplus_{k=1}^m\pi(A_{p,q})\xi_i^{q,k}}\big\rangle$. L'application $C(X)$-lin\'eaire et compl\`etement positive $\tilde\omega_{\xi_i}$ est donc d\'efinie par
$$\tilde\omega_{\xi_i}(A)=\sum_{p=1}^{n_i}\sum_{q=1}^{n_i}e_{p,q}\otimes\big\langle{\oplus_{k=1}^{m}\xi_i^{p,k},\oplus_{k=1}^m\pi(A)\xi_i^{q,k}}\big\rangle$$

\end{proof}

\begin{Pro}
Pour toute partie finie $\mathcal{F}\subset{\mathcal{A}}$ il existe un vecteur $\xi_i$ dans $\big(\mathcal{H}^{n_i}_{C(X)}\big)^{m}$ tel que pour tout ${A}$ dans ${\mathcal{F}}$
$$\delta_i(A)-\big(\big\langle\xi^{(p)}_{i},(\pi(A)\otimes{id_{m_i}})(\xi^{(q)}_i)\big\rangle\big)_{1\leqslant{p,q}\leqslant{n_i}}$$
est un op\'erateur compact dans $M_{n_i}\otimes{C(X)}$; de plus pour un r\'eel strictement positif $\varepsilon$, on peut choisir $\xi_i$ de sorte que, pour tout $A$ dans $\mathcal{F}$, la norme  de cet op\'erateur compact soit inf\'erieure \`a $\frac{\varepsilon}{2^{n_i}}$.
\end{Pro}
\begin{proof}
On pose $\tilde{\mathcal{F}}=\{A\otimes{e_{p,q}}\textrm{, }A\in\mathcal{F}\textrm{, }1\leq p,q\leq{n_i}\}$ est une partie finie de $M_{n_i}(\mathcal{A})$. Par le travail pr\'ec\'edent, il existe un \'el\'ement $\xi_i\in{(\mathcal{H}^{n_i}_{C(X)})^{m}}$ tel que pour tout $\tilde{A}$ dans $\tilde{\mathcal{F}}$, on a $$\lVert{\delta_i'(\tilde{A})-\omega_{\xi_i}(\tilde{A})}\lVert<\frac{\varepsilon}{n_i^22^{i}}$$
et pour tout $\tilde{A}$ dans $M_{n_i}(\mathcal{A})$, on a l'\'egalit\'e
\begin{align*}
\delta_i'(\tilde{A})-\omega_{\xi_i}(\tilde{A})&=\sum_{p,q=1}^{n_i}\delta_{i,p,q}(\tilde{A}_{p,q})-\sum_{p,q=1}^{n_i}\big\langle{\xi_i^{(p)},(\pi(\tilde{A}_{p,q})\otimes{id_{m_i}})(\xi_i^{(q)})}\big\rangle
\\
&=\sum_{p,q=1}^{n_i}\bigg(\delta_{i,p,q}(\tilde{A}_{p,q})-\big\langle{\xi_i^{(p)},(\pi(\tilde{A}_{p,q})\otimes{id_{m_i}})(\xi_i^{(q)})}\big\rangle\bigg)
\end{align*} 

Pour tout $A$ dans $\mathcal{A}$, 
$$\delta_i(A)-\tilde{\omega}_{\xi_i}(A)=\sum_{p,q=1}^{n_i}e_{p,q}\otimes\big(\delta_{i,p,q}(A)-\langle{\xi_i^{(p)},(\pi({A})\otimes{id_{m_i}})(\xi_i^{(q)})}\rangle\big)$$ donc on a 
$${\big\lVert{\delta_i(A)-\tilde{\omega}_{\xi_i}(A)}\big\lVert}\leq\sum_{p,q=1}^{n_i}\big\lVert{\delta_{i,p,q}(A)-\langle{\xi_i^{(p)},(\pi({A})\otimes{id_{m_i}})(\xi_i^{(q)})}\rangle}\big\lVert$$

Pour tout $A$ dans $\mathcal{F}$ et tout $p,q$ dans $\{1,...,n_i\}$, la matrice $A\otimes{e_{p,q}}$ est dans $\tilde{\mathcal{F}}$, on a alors 
\begin{align*}
M&=\big\lVert{\delta_{i,p,q}(A)-\big\langle{\xi_i^{(p)},(\pi(A)\otimes{id_{m_i}})(\xi_i^{(q)})}\big\rangle}\big\lVert
\\
&=\bigg\lVert{\delta'_i(A\otimes{e_{p,q}})-\sum_{p,q}\langle{\xi_i^{(p)},(\pi(A\otimes{e_{p,q}})\otimes{id_{m_i}})(\xi_i^{(q)})}\rangle}\bigg\lVert
\\
&<\frac{\varepsilon}{n_i^22^{i}}
\end{align*}
Donc pour tout $\varepsilon>0$ et toute partie finie $\mathcal{F}$ dans $\mathcal{A}$, on a pour tout $A$ dans $\mathcal{F}$, 
$$\big\lVert{\delta_i(A)-\tilde{\omega}_{\xi_i}(A)}\big\lVert<\frac{\varepsilon}{2^{i}}$$

\end{proof}

\begin{Rem}
Pour tout $\xi_i\in{\big({\mathcal{H}^{n_i}_{C(X)}}\big)^{m}}$ et tout ${A}$ dans ${\mathcal{A}}$, on a 
$$\big(\big\langle\xi_i^{(j)},(\pi(A)\otimes{id_{m_i}})(\xi_i^{(j)})\big\rangle\big)_{1\leqslant{j}\leqslant{n_i}}=V_i^*\big((\pi(A)\otimes{id_{n_i}})\otimes{id_{m_i}}\big)V_i$$
\end{Rem}
\begin{Cor}
Pour tout r\'eel strictement positif $\varepsilon$ et toute partie finie $\mathcal{F}$ incluse dans ${\mathcal{A}}$, il existe ${V_i}$ dans  $\mathcal{L}_{C(X)}\big(C(X)^{n_i},{({\mathcal{H}^{n_i}_{C(X)}})^{m_i}}\big)$, telle que pour tout  ${A}$ dans $\mathcal{F}$
$$\big\lVert\delta_i(A)-V_i^*\big((\pi(A)\otimes{id_{n_i}})\otimes{id_{m_i}}\big)V_i^{}\big\lVert<\frac{\varepsilon}{2^{i}}$$
\end{Cor}
\paragraph{Approximation :} 
on va maintenant pouvoir donner une approximation de l'application identit\'e de $\mathcal{A}$ ce qui induit alors une approximation de la repr\'esentation r\'eguli\`ere du groupo\"{\i}de $\mathcal{G}$ localement compact $\sigma$-compact et \'etale. On note dans la suite $\mathcal{H}^\infty_{C(X)}:=\oplus_{i\in\NN}\big(\mathcal{H}^{n_i}_{C(X)}\big)^{m_i}$. On consid\`ere 
$\delta=\oplus_{i\in\NN}\delta_i : \mathcal{A}\rightarrow{\oplus_{i\in\NN}}M_{n_i}\otimes{C(X)}$
et on pose  
$$W : \oplus_{i\in\NN}C(X)^{n_i}\longrightarrow\oplus_{i\in\NN}\big(\mathcal{H}^{n_i}_{C(X)}\big)^{m_i}$$
d\'efini par $W:=\oplus_{i\in\NN}V_i$.
On note $\pi_{\infty}$ la repr\'esentation fid\`ele d\'efinie par
\begin{eqnarray*}
\pi_\infty : &\mathcal{A}& \longrightarrow{\mathcal{L}\big(\mathcal{H}^{\infty}_{C(X)}\big)} \\
&A& \longrightarrow{\oplus_{i\in\NN}\big((\pi(A)\otimes{id_{n_i}})\otimes{id_{m_i}}\big)}
\end{eqnarray*}

\begin{Pro}
Pour tout $\varepsilon>0$ et toute partie finie $\mathcal{F}\subset{\mathcal{A}}$, on a
\begin{enumerate}
\item[a)] il existe une application ${W} : \mathcal{H}^\infty_{C(X)}\rightarrow{\mathcal{H}^\infty_{C(X)}}$ qui est $C(X)$-lin\'eaire et adjointable tel que $\delta(A)-W^*\big(\pi_\infty(A)\big)W$ est compact pour tout ${A}$ dans $\mathcal{F}$ et v\'erifie
$$\big\lVert\delta(A)-W^*\big(\pi_\infty(A)\big)W\big\lVert<{\varepsilon}$$
\item[b)] il existe une application $\tilde{W} : \mathcal{H}^\infty_{C(X)}\rightarrow{\mathcal{H}^\infty_{C(X)}}$ $C(X)$-lin\'eaire et adjointable tel que pour tout ${A}$ dans $\mathcal{F}$
$A-\tilde{W}^*\big(\pi_\infty(A)\big)\tilde{W}$ est compact et v\'erifie
$$\big\lVert{A}-\tilde{W}^*\big(\pi_\infty(A)\big)\tilde{W}\big\lVert<{\varepsilon}$$
\end{enumerate}
\vspace{+2mm}
\end{Pro}

\begin{proof}
Soit $\varepsilon>0$ et $\mathcal{F}$ une partie finie de $\mathcal{A}$.
\vspace{+2mm}
\\
a) Pour tout $i$ dans $\NN$, il existe un op\'erateur born\'e $V_i : C(X)^{n_i}\rightarrow{\big(\mathcal{H}^{n_i}_{C(X)}\big)^m}$, tel que pour tout $A$ dans $\mathcal{F}$, on a $\lVert{\delta_i(A)-V_i^\ast{\oplus\pi(A)}V_i}\lVert<\frac{\varepsilon}{2^{i}}$.
\\
On pose $W : \oplus_{i\in\NN}C(X)^{n_i}\rightarrow{\oplus_{i\in\NN}\big(\mathcal{H}^{n_i}_{C(X)}\big)^m}$ d\'efini, pour tout $(f_i)_{i\in\NN}$ dans $\oplus_{i\in\NN}C(X)^{n_i}$ par $W((f_i)_{i\in\NN})=\oplus_{i\in\NN}V_i(f_i)$. $W$ est un op\'erateur d\'efini sur $\oplus_{i\in\NN}C(X)^{n_i}$, qui admet un adjoint $W^\ast=\oplus_{i\in\NN}V_i^\ast$, donc continu par le th\'eor\`eme du graphe ferm\'e. On a alors 
\begin{align*}
\lVert{\delta(A)-W^\ast\pi_\infty(A)W}\lVert&=\lVert{\oplus_{i\in\NN}\delta_i(A)-\oplus_{i\in\NN}V_i^\ast\oplus\pi(A)V_i}\lVert
\\
&\leq{\sum_{i\in\NN}\Vert{\delta_i(A)-V_i^\ast\oplus\pi(A)V_i}\lVert}
<{\sum_{i\in\NN}\frac{\varepsilon}{2^{i}}}
\\
&<\varepsilon
\vspace{+2mm}
\end{align*}
b) Il existe $V : \rightarrow{ }$ tel que ${A}-V^*\delta(A)V$ est compact et $\lVert{A}-V^*\delta(A)V\lVert<\varepsilon$ donc en notant $\tilde{W}=VW$, on a  $A-\tilde{W}^*\big(\pi_\infty(A)\big)\tilde{W}$ est compact et v\'erifie
$$\|{A}-\tilde{W}^*\big(\pi_\infty(A)\big)\tilde{W}\|\leq{\varepsilon}$$
\end{proof}

\begin{Theo}{\label{TFF}}
Pour toute famille finie $\mathcal{F}$ dans $\mathcal{A}$ et tout r\'eel $\varepsilon$ strictement positif, il existe un entier $n$ et des applications $\varphi : \mathcal{A}\rightarrow{M_n\otimes{C(X)}}$ et $\psi : M_n\otimes{C(X)}\rightarrow{\mathcal{L}\big(\mathcal{H}_{C(X)}^\infty\big)}$, chacune $C(X)$-lin\'eaire, compl\`etement positive et contractante, v\'erifiant pour tout $A$ dans $\mathcal{F}$
$$\big\lVert{A-\psi\circ\varphi(A)}\big\lVert<\varepsilon$$
\end{Theo}

\begin{proof}
On consid\`ere une famille finie $\mathcal{F}$ dans $\mathcal{A}$ et $\varepsilon$ un r\'eel strictement positif. Par la proposition pr\'ec\'edente, il existe un op\'erateur $W$ dans $\mathcal{L}(\mathcal{H}_{C(X)}^{\infty})$, tel que pour tout $A$ dans $\mathcal{F}$, 
$$\lVert{A-W^\ast\pi_\infty(A)W}\lVert<\frac{\varepsilon}{2}$$
Par la $C(X)$-nucl\'earit\'e de la $C(X)$-alg\`ebre $C(X)\otimes\mathcal{O}_2$, on a montr\'e, pour l'ensemble fini $\mathcal{F}$ de $\mathcal{A}$ et $\varepsilon$ le r\'eel strictement positif choisis, l'existence d'un entier $n$ et d'applications $\tilde\varphi : \mathcal{A}\rightarrow{C(X)\otimes{M_n(\CC)}}$ et $\tilde\psi : C(X)\otimes{M_n(\CC)}\rightarrow{\mathcal{L}\big(\mathcal{H}_{C(X)}\big)}$, $C(X)$-lin\'eaires compl\`etement positives et contractantes v\'erifiant, pour tout $A$ dans $\mathcal{F}$,
$$\big\lVert{\pi(A)-\tilde\psi\circ\tilde\varphi(A)}\big\lVert<\frac{\varepsilon}{2\lVert{W}\lVert^2}$$
On pose $\varphi:=\tilde\varphi$ une application de $\mathcal{A}$ dans $M_n\otimes{C(X)}$ et $\psi:=W^\ast(\tilde\psi\otimes{id_{l^2(\NN)}})W$ une application de $M_n\otimes{C(X)}$ dans $\mathcal{L}(\mathcal{H}_{C(X)}^\infty)$. Ce sont des applications $C(X)$-lin\'eaires compl\`etement positives et contractantes et pour tout $A$ dans $\mathcal{F}$, on a 
\begin{align*}
\big\lVert{A-\psi\circ\varphi(A)}\big\lVert&=\big\lVert{A-W^\ast\pi_\infty(A)W+W^\ast\pi_\infty(A)W-\psi\circ\varphi(A)}\big\lVert
\\
&\leq{\big\lVert{A-W^\ast\pi_\infty(A)W}\big\lVert+\big\lVert{W^\ast\pi_\infty(A)W-\psi\circ\varphi(A)}\big\lVert
}
\\
&<\frac{\varepsilon}{2}+\big\lVert{W^\ast\pi_\infty(A)W-W^\ast(\tilde\psi\otimes{id_{l^2(\NN)}})(\tilde\varphi(A))W}\big\lVert
\\
&<\frac{\varepsilon}{2}+\big\lVert{W^\ast\big(\pi_\infty(A)-(\tilde\psi\otimes{id_{l^2(\NN)}})(\tilde\varphi(A))\big)W}\big\lVert
\\
&<\frac{\varepsilon}{2}+\big\lVert{W}\big\lVert^2\big\lVert{\pi_\infty(A)-(\tilde\psi\otimes{id_{l^2(\NN)}})(\tilde\varphi(A))}\big\lVert
\\
&<\frac{\varepsilon}{2}+\big\lVert{W}\big\lVert^2\big\lVert{\big(\pi(A)-\tilde\psi\circ\tilde\varphi(A))\big)\otimes{id_{l^2(\NN)}}}\big\lVert
\\
&<\frac{\varepsilon}{2}+\big\lVert{W}\big\lVert^2\big\lVert{\pi(A)-\tilde\psi\circ\tilde\varphi(A))}\big\lVert
\\
&<{\frac{\varepsilon}{2}+\big\lVert{W}\big\lVert^2\frac{\varepsilon}{2\lVert{W}\lVert^2}}
\\
&<\varepsilon
\end{align*}
On a montr\'e que pour toute famille finie $\mathcal{F}$ dans $\mathcal{A}$ et tout $\varepsilon$ r\'eel strictement positif, il existe un entier $n$ et des applications $\varphi : \mathcal{A}\rightarrow{M_n\otimes{C(X)}}$ et $\psi : M_n\otimes{C(X)}\rightarrow{\mathcal{L}\big(\mathcal{H}_{C(X)}^\infty\big)}$, chacune $C(X)$-lin\'eaire, compl\`etement positive et contractante, v\'erifiant pour tout $A$ dans $\mathcal{F}$
$$\big\lVert{A-\psi\circ\varphi(A)}\big\lVert<\varepsilon$$
\end{proof}
\begin{Rem}
On peut \'etendre le r\'esultat du th\'eor\`eme $\ref{TFF}$  \`a des familles $\mathcal{F}$ compactes (et non plus seulement finie) de $\mathcal{A}$. En effet en consid\'erant $\mathcal{F}$ une famille compacte dans $\mathcal{A}$ et $\varepsilon$ un r\'eel strictement positif, on peut pour chaque \'el\'ement $a$ de $\mathcal{F}$ consid\'erer un voisinage ouvert $U_a$ dans $\mathcal{A}$ de sorte que pour tout $a'$ dans $U_a$, on a $\lVert{a-a'}\lVert<\varepsilon/3$. On obtient ainsi un recouvrement ouvert de $\mathcal{A}$ duquel on peut extraire un sous recouvrement fini $\mathcal{U}=\{U_{a_i}\}_{i=1}^n$. On consid\`ere $\mathcal{F}_1$ le sous ensemble fini de $\mathcal{F}$ constitu\'e des $m$ \'el\'ements $a_i$, pour $i$ variant de $1$ \`a $m$. D'apr\`es le th\'eor\`eme $\ref{TFF}$ , il existe un entier $n$ et des applications $\varphi : \mathcal{A}\rightarrow{M_n\otimes{C(X)}}$ et $\psi : M_n\otimes{C(X)}\rightarrow{\mathcal{L}\big(\mathcal{H}_{C(X)}^\infty\big)}$, chacune $C(X)$-lin\'eaire, compl\`etement positive et contractante, v\'erifiant pour tout $i$ dans $\{1,...,m\}$,
$$\big\lVert{a_i-\psi\circ\varphi(a_i)}\big\lVert<\varepsilon/3$$
Pour tout $a$ dans $\mathcal{F}$, il existe $a_i$ dans $\mathcal{F}'$, tel que $\Vert{a-a_i}\lVert<\varepsilon/3$, et on a 
\begin{align*}
\big\lVert{a-\psi\circ\varphi(a)}\big\lVert&=\big\lVert{a-a_i+a_i-\psi\circ\varphi(a_i)+\psi\circ\varphi(a_i)-\psi\circ\varphi(a)}\big\lVert
\\
&\leq{\big\lVert{a-a_i}\big\lVert+\big\lVert{a_i-\psi\circ\varphi(a_i)}\big\lVert+\big\lVert{\psi\circ\varphi(a_i)-\psi\circ\varphi(a)}\big\lVert}
\\
&<\varepsilon/3+\varepsilon/3+\varepsilon/3
\\
&<\varepsilon
\end{align*}
\end{Rem}
\begin{Def}
Soit $\mathcal{G}$ un groupo\"{\i}de s\'epar\'e, localement compact, $B$ une $C_0(X)$-alg\`ebre et $\phi : C^\ast_r(\mathcal{G})\rightarrow{B}$ une application $C_0(X)$-lin\'eaire compl\`etement positive. L'application $\phi$ est dite \`a support compact s'il existe un sous espace compact $K$ de $\mathcal{G}$ tel que pour toute fonction $f$ dans $C_c(\mathcal{G})$ v\'erifiant $Supp(f)\cap{K}=\emptyset$, on a $\phi(f)=0$.   
\end{Def}

\begin{Pro}
Soit	$\varphi : \mathcal{A}\rightarrow{M_n\otimes{C(X)}}$ une application $C(X)$-lin\'eaire et compl\`etement positive, alors pour tout r\'eel $\varepsilon$ strictement positif et tout famille finie $\mathcal{F}$ dans $\mathcal{A}$, il existe une application $C(X)$-lin\'eaire, compl\`etement positive et \`a support compact $\varphi' : \mathcal{A}\rightarrow{M_n\otimes{C(X)}}$ telle que pour tout $a$ dans $\mathcal{F}$, 
$$\lVert{\varphi(a)-\varphi'(a)}\lVert<\varepsilon$$
\end{Pro}

\begin{proof}
L'application $\varphi :  \mathcal{A}\rightarrow{M_n\otimes{C(X)}}$ \'etant $C(X)$-lin\'eaire et compl\`etement positive, alors par le th\'eor\`eme de Stinespring pour les $C(X)$-modules de Hilbert, il existe un $C(X)$-module $F_\varphi$, un $\ast$-homomorphisme $\pi_\varphi :  \mathcal{A}\rightarrow{\mathcal{L}(F_\varphi)}$ et un op\'erateur $v_\varphi$  dans $\mathcal{L}(M_n\otimes{C(X)},F_\varphi)$ tel que pour tout $a$ dans $\mathcal{A}$, on a $\varphi(a)=v_\varphi^\ast\pi_\varphi(a)v_\varphi$ c'est-\`a-dire qu'il existe des vecteurs $e_1,...,e_n$ dans $F_\varphi$ tels que 
$$\varphi(a)=v_\varphi^\ast\pi_\varphi(a)v_\varphi=\big[\langle{e_i,\pi_\varphi(a)e_j}\rangle\big]$$
On a un repr\'esentation fid\`ele de $\lambda : \mathcal{A}$ dans $L^2(\mathcal{G},\nu)$ et pour tout $\eta>0$, quitte \`a prendre un multiple de la repr\'esentation fid\`ele $\lambda$, il existe des vecteurs $\xi_1,...,\xi_n$ dans $L^2(\mathcal{G},\nu)$ tels que pour tout $a$ dans $F$
$$\big\lvert\langle{e_i,\pi_\varphi(a)e_j}\rangle-\langle{\xi_i,\lambda(a)\xi_j}\rangle\big\lvert<\eta$$
On peut supposer que les fonctions $\xi_i$ sont \`a support compact et en posant $\varphi'(a)=\big[\langle{\xi_i,\lambda(a)\xi_j}\rangle\big]$, on obtient une application $C(X)$-lin\'eaire, compl\`etement positive et \`a support compact $\varphi' : \mathcal{A}\rightarrow{M_n\otimes{C(X)}}$ telle que pour tout $a$ dans $\mathcal{F}$, 
$$\lVert{\varphi(a)-\varphi'(a)}\lVert<\varepsilon$$
\end{proof}

%
\subsection{Moyennabilit\'e}
On consid\`ere l'espace produit $\mathcal{G}\times\mathcal{G}$ et pour $i=1,2$, on note $p_i : \mathcal{G}\times\mathcal{G}\rightarrow{\mathcal{G}}$ la projection sur la $i$-i\`eme composante. Soit $F$ un sous espace ferm\'e de l'espace $\mathcal{G}\times\mathcal{G}$, on dit que $F$ est $(p_1,p_2)$-propre si pour tout sous espace compact $K$ de $\mathcal{G}$, les ensembles $F\cap{(K\times\mathcal{G})}$ et $F\cap{\mathcal{G}\times{K}}$ sont compacts. 
\begin{Def}
On dit que le groupo\"{\i}de $\mathcal{G}$ a la propri\'et\'e $(W)$ si pour tout sous espace compact $K$ de $\mathcal{G}$ et pour tout r\'eel $\varepsilon$ strictement positif, il existe une fonction $h : \mathcal{G}\times{\mathcal{G}}\rightarrow\CC$ continue, born\'ee, de type positif et \`a support $(p_1,p_2)$-propre v\'erifiant pour tout $\gamma$ dans $K$,
$$\lvert{h(\gamma,\gamma)-1}\lvert<\varepsilon$$
\end{Def}
On consid\`ere $h : \mathcal{G}\times\mathcal{G}\rightarrow\CC$ une fonction continue, born\'ee, \`a support $(p_1,p_2)$-propre et de type positif. Pour tout $\gamma$ dans $\mathcal{G}$, on note $h_\gamma$ la fonction continue sur $\mathcal{G}$ et \`a support compact d\'efinie par $h_\gamma(\eta):=h(\gamma,\eta)$ pour tout $\eta$ dans $\mathcal{G}$.
\begin{Lemma}{\label{LemTypPosFon}}
Soit $h : \mathcal{G}\times\mathcal{G}\rightarrow\CC$ une fonction continue, born\'ee, \`a support $(p_1,p_2)$-propre et de type positif. L'application $\gamma\rightarrow{\lambda(h_\gamma)}$ continue sur $\mathcal{G}$ \`a valeurs dans $C^\ast_r(\mathcal{G})$ est de type positif.
\end{Lemma}
\begin{proof}
On doit prouver que pour tout entier $n$, tout $x$ dans $X$, tout $\gamma_1,...,\gamma_n$ dans $\mathcal{G}^x$, tout $\xi_1,...,\xi_n$ dans $L^2(\mathcal{G})$, on a, pour tout $y$ dans $X$
$$\sum_{i,j=1}^n\big\langle{\xi_i,\lambda(h_{\gamma_i^{-1}\gamma^{}_j})\xi_j}\big\rangle(y)\geq0$$
On peut supposer, pour tout $i$ dans $N=\{1,...,n\}$, que les fonctions $\xi_i$ sont dans $C_c(\mathcal{G})$. On consid\`ere $F:=\cup_{i=1}^nSupp(\xi_i\lvert_{\mathcal{G}_y})$ l'ensemble fini de $\mathcal{G}_y$. On a alors 
\begin{align*}
A&=\sum_{i,j=1}^n\big\langle{\xi_i,\lambda(h_{\gamma_i^{-1}\gamma^{}_j})\xi_j}\big\rangle(y)=\sum_{i,j=1}^n\sum_{\alpha\in\mathcal{G}_y}\overline{\xi_i(\alpha)}\big(h_{\gamma_i^{-1}\gamma^{}_j}\star\xi_j\big)({\alpha})
\\
&=\sum_{i,j=1}^n\sum_{\alpha\in\mathcal{G}_y}\sum_{\beta\in\mathcal{G}^{r(\alpha)}}\overline{\xi_i(\alpha)}h_{\gamma_i^{-1}\gamma^{}_j}(\alpha\beta^{-1})\xi_j(\beta)
\\
&=\sum_{i,j=1}^n\sum_{r\in{F}}\sum_{r'\in{F}}\overline{\xi_i(\alpha_r)}h(\gamma_i^{-1}\gamma^{}_j,\alpha_{r}^{}\alpha_{r'}^{-1})\xi_j(\alpha_{r'})
\end{align*}
Pour tout $i$ dans $N=\{1,...,n\}$ et $r$ dans $F$, on pose $\lambda_{i,r}:=\xi_i(\alpha_r)$, $\sigma_{i,r}:=\gamma_i$ et $\tau_{i,r}:=\alpha_r^{-1}$, et on a 
\begin{align*}
A&=\sum\limits_{\substack{i\in{N} \\ r\in{F}}}\sum\limits_{\substack{j\in{N} \\ r'\in{F}}}\overline{\lambda_{i,r}}\lambda_{j,r'}h(\sigma_{i,r}^{-1}\sigma_{j,r'}^{},\tau_{i,r}^{-1}\tau_{j,r'}^{})\geq{0}
\end{align*}
car par d\'efinition, l'application $h : \mathcal{G}\times\mathcal{G}\rightarrow\CC$ est de type positif.
\\
\end{proof}

On suppose que le groupo\"{\i}de $\mathcal{G}$ a la propri\'et\'e $(W)$. Soit $n$ un entier, on consid\`ere le sous espace compact $K_n$ de $\mathcal{G}$. Il existe une fonction $h : \mathcal{G}\times\mathcal{G}\rightarrow\CC$ continue, born\'ee, de type positif et \`a support $p_1,p_2$-propre v\'erifiant pour tout $\gamma$ dans $K_n$, 
$$\lvert{h(\gamma,\gamma)-1}\lvert<(2n)^{-1}$$ 
Pour tout $\beta$ dans $\mathcal{G}$, on note $h_\beta : \mathcal{G}\rightarrow{\CC}$ la fonction d\'efinie, pour tout $\alpha$ dans $\mathcal{G}$, par  $h_\beta(\alpha):=h(\beta,\alpha)$ qui est continue et \`a support compact sur $\mathcal{G}$. On consid\`ere $\mathcal{F}=\{\lambda(h_\beta)\textrm{ : }\beta\in{K_n}\}$ qui est une partie compacte de $C_r^\ast(\mathcal{G})$. D'apr\`es l'approximation par factorisation de la repr\'esentation r\'eguli\`ere faite dans $\ref{ApproxiFacto}$ , 
 il existe une application $C(X)$-lin\'eaire, compl\`etement positive et contractante $\phi_n : \mathcal{A}\rightarrow{\mathcal{L}\big(\mathcal{H}_{C(X)}^\infty\big)}$  telle que pour tout $\beta$ dans ${K_{n}}$, 
$$\big\lVert{\phi_n(h_\beta)-id_\mathcal{A}(h_\beta)}\big\lVert<(2n)^{-1}$$
et l'application $\phi_n$ est \`a support dans un compact $C_n$ de $\mathcal{G}$.
\\

Pour tout $\gamma$ dans $\mathcal{G}$, on consid\`ere $U_\gamma$ un voisinage ouvert de $\gamma$ sur lequel les restrictions des applications but et source sont hom\'eomorphismes sur leur image respective. On obtient ainsi un recouvrement d'ouvert $\{U_\gamma\}_{\gamma\in\mathcal{G}}$ de l'espace $\mathcal{G}$ duquel on peut extraire un sous recouvrement localement fini $\mathcal{U}'=\{U_{\gamma_i}\}_{i\in{I}}$, o\`u $I$ est d\'enombrable. On note $\mathcal{U}:=\{U_{\gamma_i}\times{U_{\gamma_i}}\}_{i\in{I}}$ le recouvrement ouvert localement fini de la diagonale $\Delta_{\mathcal{G}}$ de $\mathcal{G}\times\mathcal{G}$. On lui associe une partition de l'unit\'e $\{\psi_i\}_{i\in{I}}$, o\`u chaque fonction $\psi_i$ est continue, positive et \`a support compact dans l'ouvert $U_{\gamma_i}\times{U_{\gamma_i}}$ telles que 
$$\chi_{\Delta_\mathcal{G}}\leq\sum_{i\in{I}}\psi_i\leq{1}$$
On pose $\psi(\alpha,\beta)=\sum_{i\in{I}}\psi_i(\alpha,\beta)$ la fonction d\'efinie sur $\mathcal{G}\times\mathcal{G}$ continue et \`a support dans $\cup_{i\in{I}}(U_{\gamma_i}\times{U_{\gamma_i}})$.
Pour tout $\gamma$ dans $\mathcal{G}$, on note $f_\gamma : \mathcal{G}\rightarrow{\RR}$ d\'efinie par $f_\gamma(\alpha)=\psi(\gamma,\alpha)$ est une fonction continue et \`a support compact dans $W_\gamma=\cup_{i\in{I_\gamma}}{U_{\gamma_i}}$, o\`u $I_\gamma=\{i\in{I} : \gamma\in{U_{\gamma_i}}\}$ est une partie finie de $I$. La restriction des applications but et source \`a chaque $W_\gamma$ est un hom\'eomorphisme sur leur image respective.  
\\
Comme $\mathcal{G}$ est $\sigma$-compact, on suppose qu'il existe une suite $(K_n)_{n\in\NN}$ croissante de compacts de $\mathcal{G}$ telle que $K_n\subset{\mathring{K}_{n+1}}$ pour tout $n$ entier et $\bigcup_{n\in\NN}K_n=\mathcal{G}$.

 \begin{Pro}
Pour tout $n$ dans $\NN$, il existe une fonction $\mu_n$ continue sur $\mathcal{G}\ast_r\mathcal{G}$, \`a valeurs r\'eelles, de type positif et \`a support dans un tube telle que pour tout $(\gamma,\eta)$ dans le tube $T_{K_{n}}$, on a 
$$\lvert{\mu_n(\gamma,\eta)-1}\lvert<1/n$$
\vspace{-5mm}
\end{Pro}
\begin{proof}
On pose $\mu_n$ la fonction d\'efinie sur $\mathcal{G}\ast_r\mathcal{G}$ qui, \`a tout couple $(\gamma,\eta)$ de $\mathcal{G}\ast_r\mathcal{G}$, associe $\mu_n(\gamma,\eta)=\big\langle{f^\ast_\gamma,\phi_n(h_{\gamma^{-1}\eta})(f^\ast_\eta)}\big\rangle(x)$, o\`u $x=r(\gamma)=r(\eta)$. 
\\

Les applications $\gamma\rightarrow{f_\gamma}$ et $\beta\rightarrow{h_\beta}$ \'etant continue sur $\mathcal{G}$, alors $\mu_n : \mathcal{G}\ast_r\mathcal{G}\rightarrow\RR$ est continue comme composition d'applications continues.
\\

Le support $C_n$ de l'application $\phi_n$ un sous espace compact de $\mathcal{G}$ et la fonction $h$ est \`a support $(p_1,p_2))$-propre donc on a $\mathrm{Supp}(h)\cap{\mathcal{G}\times{C_n}}$ est compact. On en d\'eduit alors que $\{\alpha\in\mathcal{G}\textrm{ : }\phi_n(h_\alpha)\neq0\}$ est compact et que la fonction $\mu_n$ est \`a support dans un tube.  
\\

On montre que les fonctions $\mu_n$ sont de type positif : l'application $\phi_n$ \'etant compl\`etement positive, on a alors 
pour tout entier $n$, tout $(z_k)_{{1\leq{k}\leq{n}}}$ dans $\mathbb{C}$, pour tout $x$ dans $X$ et tout $(\gamma_p)_{1\leq{p}\leq{n}}$ dans $\mathcal{G}^x$, on a 
\begin{align*}
A&=\sum_{p,q=1}^n\overline{z_p}z_q\mu_n(\gamma_p,\gamma_q)
\\
&=\sum_{p,q=1}^n\overline{z_p}z_q\big\langle{f_{\gamma_p}^\ast;\phi_n\big(h_{\gamma_p^{-1}\gamma_q}\big)f^\ast_{\gamma_q}}\big\rangle(x)
\\
&=\sum_{p,q=1}^n\big\langle{z_pf_{\gamma_p}^\ast;\phi_n\big(h_{\gamma_p^{-1}\gamma_q}\big)(z_qf^\ast_{\gamma_q})}\big\rangle(x)
\end{align*}
or l'application $\gamma\to\lambda(h_\gamma)$ est de type positif d'apr\`es le lemme \ref{LemTypPosFon} donc la matrice $[\lambda(h_{\gamma^{-1}_p\gamma_q})]_{1\leq{p,q}\leq{n}}$ est positive et comme l'application $\phi$ est compl\`etement positive alors la matrice $\big[\phi(h_{\gamma^{-1}_p\gamma_q})\big]_{1\leq{p,q}\leq{n}}$ est positive et la fonction $\mu_n$ est donc de type positif pour tout entier $n$.
\\

Pour prouver que la suite $(\mu_n)_{n\in\NN}$ converge uniform\'ement vers un sur les tubes,  on a,  pour tout couple $(\gamma,\eta)$ dans le tube $T_{K_n}$, 
\begin{align*}
B&=\lvert{\mu_n(\gamma,\eta)-1}\lvert=\big\lvert{\langle{f^\ast_\gamma,\phi_n(h_{\gamma^{-1}\eta})(f^\ast_\eta)}\rangle}(x)-1\big\lvert
\\
&\leq\big\lvert\langle{f^\ast_\gamma,(\phi_n-id_{\mathcal{A}})(h_{\gamma^{-1}\eta})(f^\ast_\eta)}\rangle(x)\big\lvert+\big\lvert\langle{f^\ast_\gamma,id_{\mathcal{A}}(h_{\gamma^{-1}\eta})(f^\ast_\eta)}\rangle(x)-1\big\lvert
\\
&\leq{\big\lVert{(\phi_n-id_{\mathcal{A}})(h_{\gamma^{-1}\eta})}\big\lVert}+\bigg\lvert{\sum_{\alpha\in{\mathcal{G}_x}}\overline{f^\ast_\gamma(\alpha)}(h_{\gamma^{-1}\eta}\star{f^\ast_\eta})(\alpha)-1}\bigg\lvert
\\
&\leq{\big\lVert{(\phi_n-id_{\mathcal{A}})(h_{\gamma^{-1}\eta})}\big\lVert}+\bigg\lvert{\overline{f^\ast_\gamma(\gamma^{-1})}\sum_{\beta\in\mathcal{G}^{r(\gamma)}}h_{\gamma^{-1}\eta}(\gamma^{-1}\beta){f^\ast_\eta}(\beta^{-1})-1}\bigg\lvert
\\
&\leq{\big\lVert{(\phi_n-id_{\mathcal{A}})(h_{\gamma^{-1}\eta})}\big\lVert}+\big\lvert{h_{\gamma^{-1}\eta}(\gamma^{-1}\eta){f^\ast_\eta}(\eta^{-1})-1}\big\lvert
\\
&\leq{\big\lVert{(\phi_n-id_{\mathcal{A}})(h_{\gamma^{-1}\eta})}\big\lVert}+\big\lvert{h_{\gamma^{-1}\eta}(\gamma^{-1}\eta)-1}\big\lvert
\\
&<(2n)^{-1}+(2n)^{-1}
\\
&<1/n
\end{align*}
\end{proof}

\begin{Theo}
Soit $\mathcal{G}\rightrightarrows{X}$ un groupo\"{\i}de \'etale localement compact $\sigma$-compact et s\'epar\'e ayant un espace des unit\'es $X$ compact, satisfaisant la propri\'et\'e (W) et la condition $\mathbf{(Cont)}$. Si la $C^\ast$-alg\`ebre $C^\ast_r(\mathcal{G})$ est exacte, alors le groupo\"{\i}de $\beta_X\mathcal{G}\rtimes\mathcal{G}$ est moyennable. 
\end{Theo}

\begin{proof}
Pour prouver que groupo\"{\i}de $\beta_X\mathcal{G}\rtimes\mathcal{G}$ est moyennable, il suffit de montrer que pour tout r\'eel strictement positif $\varepsilon$ et tout sous espace compact $K$ de $\beta_X\mathcal{G}\rtimes{\mathcal{G}}$, il existe une fonction $\nu_{\varepsilon,K}$ dans $C_c(\beta_X\mathcal{G}\rtimes\mathcal{G})$, continue, de type positif et \`a support compact telle que pour tout $(z,\gamma)$ dans $K$, on a $$\lvert{\nu_{\varepsilon,K}(\gamma,\eta)-1}\lvert<\varepsilon$$
Pour cela on va utiliser la $\sigma$-compacit\'e de $\mathcal{G}$ : il existe une suite $(K_n)_{n\in\NN}$ croissante de compacts de $\mathcal{G}$ telle que $K_n\subset{\mathring{K}_{n+1}}$ pour tout $n$ entier et $\bigcup_{n\in\NN}K_n=\mathcal{G}$. 
\\

On a montr\'e que pour tout $n$ dans $\NN$, il existe une fonction $\mu_n : \mathcal{G}\ast_{r,r}\mathcal{G}\rightarrow\RR$ de type positif, \`a support dans le tube $T_{K_{n+1}}$ telle que pour tout couple $(\gamma,\eta)$ dans $T_{K_n}$, on a $\lvert{\mu_n(\gamma,\eta)-1}\lvert<n^{-1}$. 
\\
On note $p_n : \mathcal{G}^{(2)}\rightarrow\RR$ la fonction continue d\'efinie pour tout couple $(\gamma,\eta)$ dans $\mathcal{G}^{(2)}$ par $p_n(\gamma,\eta)=\mu_n(\gamma,\gamma\eta)$. On remarque que si $\eta$ est en dehors du compact $K_{n+1}$ alors pour tout $\gamma$ dans $\mathcal{G}_{r(\eta)}$, on a $p_n(\gamma,\eta)=0$.
\\
Pour tout $\gamma$ dans $K_{n}$, on consid\`ere un voisinage ouvert $U_{n}^\gamma$ de $\gamma$ inclus dans $\mathring{K}_{n}$ de sorte que les applications source et but restreintes \`a $U_{n}^\gamma$ soient des hom\'eomorphismes sur leur image respective. On peut alors d\'efinir un recouvrement ouvert fini du compact $K_{n}$ not\'e $\mathcal{U}_{n}=\{U_{n}^j\}_{j=1}^m$. L'application but \'etant un hom\'eomorphisme de $U_{n}^j$ sur $r(U_{n}^j)$ pour tout $j$ dans $\{1,...,m\}$, on note $\sigma_{n}^j : r(U_{n}^j)\rightarrow{U_{n}^j}$ une section de $r$. On note $\{\varphi_n^j\}_{j=1}^m$ une partition de l'unit\'e subordonn\'ee au recouvrement $\mathcal{U}_{n}$ telle que pour tout $\eta$ dans $K_{n}$, on a $\sum_{j=1}^m\varphi_{n}^j(\eta)=1$. 
\\
Pour tout $j$ dans $\{1,...,m\}$, on consid\`ere $p_n^j : \mathcal{G}^{(2)}\rightarrow\RR$ d\'efinie pour tout $(\gamma,\eta)$ par $p_n^j(\gamma,\eta)=p_n(\gamma,\eta){\varphi_{n,j}^{{1}/{2}}(\eta)}$ qui est \`a support dans $\mathcal{G}\ast_{s,r}{U_{n,j}}$
\\
On pose ${\nu}^j_n : \mathcal{G}\rightarrow\RR$ l'application d\'efinie pour tout $\gamma$ par 
$$
\nu^j_n(\gamma)=\left\{
    \begin{array}{ll}
        p_n^j\Big(\gamma,\sigma_{n}^j(s(\gamma))\Big) & \mbox{si } s(\gamma)\in{r(U_n^j)} \\
        0 & \mbox{si } s(\gamma)\notin\cup_{j=1}^mr(U_n^j)
    \end{array}
\right.
$$
On remarque que la fonction $\nu_n^j$ est continue, born\'ee sur $\mathcal{G}$ et pour tout $\gamma$ hors du compact $K_{n+1}$ on a $\nu_n^j(\gamma)=0$. Ainsi $\nu_n^j$ est dans la $C^\ast$-alg\`ebre $C_0^s$ donc dans $C_0(\beta_X\mathcal{G})$.
\\
On pose $\nu_n : \mathcal{G}^{(2)}\rightarrow\RR$ d\'efinie pour tout $(\gamma,\eta)$ par
$$
\nu_n(\gamma,\eta)=\left\{
    \begin{array}{ll}
        \sum_{j=1}^m\nu_n^j\Big(\gamma,\sigma_{n}^j(s(\gamma))\Big)\varphi^{1/2}_{n,j}(\eta), & \mbox{si } \eta\in{K_{n+1}} \\
        0, \textrm{ } \mbox{sinon } 
    \end{array}
\right.
$$
La fonction $\nu_n$ est dans $C_0(\beta_X\mathcal{G}\rtimes\mathcal{G})$ et prolonge par continuit\'e la fonction $p_n$; en effet, pour tout $(\gamma,\eta)$ dans $Supp(p_n)$, on a 
\begin{align*}
p_n(\gamma,\eta)-\nu_n(\gamma,\eta)&=p_n(\gamma,\eta)-\sum_{j=1}^m\nu_n^j\Big(\gamma,\sigma_{n}^j(s(\gamma))\Big)\varphi^{1/2}_{n,j}(\eta)
\\
&=p_n(\gamma,\eta)-\sum_{j=1}^mp_n^j\Big(\gamma,\sigma_{n}^j(s(\gamma))\Big)\varphi^{1/2}_{n,j}(\eta)
\\
&=p_n(\gamma,\eta)-\sum_{j=1}^mp_n\Big(\gamma,\sigma_{n}^j(s(\gamma))\Big)\varphi^{1/2}_{n,j}(\sigma_n^j(s(\gamma)))\varphi^{1/2}_{n,j}(\eta)
\\
&=p_n(\gamma,\eta)-\sum_{j=1}^mp_n\Big(\gamma,\eta\Big)\varphi^{1/2}_{n,j}(\eta)\varphi^{1/2}_{n,j}(\eta)
\\
&=p_n(\gamma,\eta)-p_n(\gamma,\eta)\sum_{j=1}^m\varphi_{n,j}(\eta)=0
\end{align*}
On montre que le support de chaque fonction $\nu_n$ est compact : le support de la fonction $\nu_n$ est dans $\beta_X\mathcal{G}\ast_{\tilde{s},r}K_{n+1}$ et $r(K_{n+1})$ est un sous espace de $X$ compact (par continuit\'e de l'application but). De plus l'application $\tilde{s} : \beta_X\mathcal{G}\rightarrow{X}$ \'etant propre, l'espace $\tilde{s}^{-1}(r(K_{n+1}))$ est compact dans $\beta_X\mathcal{G}$. Ainsi le support de $\nu_n$ qui est ferm\'e et inclus dans $\tilde{s}^{-1}(r(K_{n+1}))\ast_{\tilde{s},r}K_{n+1}$ est un compact de $\beta_X\mathcal{G}\rtimes\mathcal{G}$. Donc $\nu_n$ est \`a support compact.
\\

Puisque les fonctions $\mu_n$ sont de type positif, il est clair que les fonctions $\nu_n$ sont \'egalement de type positif.
\\ 

On montre maintenant que la suite $(\nu_n)_{n\in\NN}$ converge uniform\'ement sur les compacts de $\beta_X\mathcal{G}\rtimes\mathcal{G}$ : il suffit de consid\'erer le cas particulier $Z\ast{K}$ o\`u  $Z$ est un compact $\beta_X\mathcal{G}$ et $K$ un compact de $\mathcal{G}$. On note $\tilde{Z}=Z\cap\mathcal{G}$ le sous espace de $\mathcal{G}$ et on a alors pour tout $n$ entier
\begin{align*}
\sup_{(z,\eta)\in{Z\ast{K}}}\big\lvert{\nu_n(z,\eta)-1}\big\lvert&=\sup_{(\gamma,\eta)\in{\tilde{Z}\ast{K}}}\big\lvert{\nu_n(\gamma,\eta)-1}\big\lvert
\\
&=\sup_{(\gamma,\eta)\in\tilde{Z}\ast{K}}\big\lvert{\mu_n(\gamma,\gamma\eta)-1}\big\lvert
\\
&=\sup\big\{\lvert\mu_n(\gamma,\gamma')-1\lvert\textrm{ : }(\gamma,\gamma')\in{\tilde{Z}\ast\mathcal{G}}\textrm{ / }\gamma^{-1}\gamma'\in{K}\big\}
\end{align*}
et on en d\'eduit que la suite $(\nu_n)_{n\in\NN}$ converge uniform\'ement vers un sur les sous espaces compacts de $\beta_X\mathcal{G}\rtimes\mathcal{G}$. 
\vspace{+2mm}
\\
Ainsi le groupo\"{\i}de $\beta_X\mathcal{G}\rtimes\mathcal{G}$ est moyennable.
\\
\end{proof}

\bibliographystyle{amsalpha}
\bibliography{bibliographie}

\end{document}